\documentclass[a4paper,12pt]{article}
\usepackage{amssymb}
\usepackage{amsthm}
\usepackage{amsmath}
\usepackage[T1]{fontenc}
\usepackage[cp1250]{inputenc}
\usepackage{enumerate}
\usepackage{graphicx}
\usepackage[all]{xy}

\usepackage{tikz}
\usetikzlibrary{shapes}

\newtheorem{theorem}{Theorem}[section]
\newtheorem{prop}[theorem]{Proposition}
\newtheorem{lem}[theorem]{Lemma}
\theoremstyle{remark}
\newtheorem{rem}[theorem]{Remark} 
\theoremstyle{remark}
\newtheorem{df}[theorem]{\textsc{Definition}}
\theoremstyle{remark}
\newtheorem{assm}[theorem]{\textsc{Assumption}}
\theoremstyle{remark}
\newtheorem{exmp}[theorem]{Example}

\theoremstyle{plain}
\newtheorem*{mthm}{Main Theorem}
\newtheorem*{cor1}{Corollary 1}

\renewcommand{\mod}{\operatorname{mod}}
\newcommand{\umod}{\operatorname{\underline{mod}}}
\newcommand{\add}{\operatorname{add}}
\newcommand{\Hom}{\operatorname{Hom}}
\newcommand{\End}{\operatorname{End}}
\newcommand{\soc}{\operatorname{soc}}
\newcommand{\rad}{\operatorname{rad}}

\begin{document}

\begin{center}
\Large{\textbf{\textsc{Virtual mutations of weighted surface algebras}}} \\ 
\end{center}
\begin{center}
by
\end{center}

\begin{center}
Thorsten Holm$^a$, Andrzej Skowro\'nski$^b$, and Adam Skowyrski$^b$
\end{center}
\medskip

$^a$ \textit{Fakult\"at f\"ur Mathematik und Physik, Institut f\"ur Algebra, Zahlentheorie und Diskrete Mathematik, Leibniz Universit\"at Hannover, Welfengarten 1, 30167 Hannover, Germany}

$^b$ \textit{Faculty of Mathematics and Computer Science, Nicolaus Copernicus University, Chopina 12/18, 87-100 Toru\'n, Poland} 

\medskip\medskip

\textbf{Abstract.} The finite-dimensional symmetric algebras over an algebraically closed field, based on surface 
triangulations, motivated by the theory of cluster algebras, have been extensively investigated and applied. In particular, 
the weighted surface algebras and their deformations were introduced and studied in \cite{ES2}-\cite{ES6}, and it was 
shown that all these algebras, except few singular cases, are symmetric tame periodic algebras of period $4$. In this 
article, using the general form of a weighted surface algebra from \cite{ES5}, we introduce and study so called virtual 
mutations of weighted surface algebras, which constitute a new large class of symmetric tame periodic algebras of period $4$. 
We prove that all these algebras are derived equivalent but not isomorphic to weighted surface algebras. We associate such 
algebras to any triangulated surface, first taking blow-ups of a family of edges to $2$-triangle discs, and then virtual 
mutations of their weighted surface algebras. The results of this paper form an essential step towards a classification of 
all tame symmetric periodic algebras. 
\medskip

\medskip

\medskip

E-mail adresses: holm@math.uni-hannover.de (T. Holm), skowron@mat.umk.pl (A. Skowro\'nski), skowyr@mat.umk.pl (A. Skowyrski)

\section{Introduction and the main result}\label{sec:0} 

Throughout this paper, $K$ will denote a fixed algebraically closed field. By an algebra we mean an associative 
finite-dimensional $K$-algebra with an identity. For an algebra $A$, we denote by $\mod A$ the category of finite-dimensional 
right $A$-modules and by $D$ the standard duality $\Hom_K(-,K)$ on $\mod A$. An algebra $A$ is called \emph{self-injective} 
if $A_A$ is injective in $\mod A$, or equivalently, the projective modules in $\mod A$ are injective. A prominent class 
of self-injective algebras is formed by the \emph{symmetric algebras} for which there exists an associative, non-degenerate 
symmetric $K$-bilinear form $(-,-):A\times A\to K$. Classical examples of symmetric algebras are provided by blocks of 
group algebras of finite groups and the Hecke algebras of finite Coxeter groups. In fact, any algebra $A$ is the quotient 
algebra of its trivial extension algebra $T(A)=A\ltimes D(A)$, which is a symmetric algebra. 

Let $A$ be an algebra. Given a module $M$ in $\mod A$, its \emph{syzygy} is defined to be the kernel of a minimal projective 
cover of $M$ in $\mod A$. The syzygy operator $\Omega_A$ is a very important tool to construct modules in $\mod A$ and relate 
them. For a self-injective algebra $A$, it induces an equivalence of the stable module category $\umod A$, and its inverse 
is the shift of a triangulated structure on $\umod A$ \cite{Ha1}. A module $M$ in $\mod A$ is said to be \emph{periodic} if 
$\Omega_A^n(M)\cong M$, for some $n\geqslant 1$, and minimal such $n$ is called the \emph{period} of $M$. The action of 
$\Omega_A$ on $\mod A$ can effect the algebra structure of $A$. For example, if all simple modules in $\mod A$ are periodic, 
then $A$ is a self-injective algebra. An algebra $A$ is defined to be \emph{periodic} if it is periodic viewed as a module 
over the enveloping algebra $A^e=A^{op}\otimes_K A$, or equivalently, as an $A$-$A$-bimodule. It is known that if $A$ is 
a periodic algebra of period $n$, then for any indecomposable non-projective module $M$ in $\mod A$ the syzygy $\Omega_A^n(M)$ 
is isomorphic to $M$. Periodic algebras are self-injective and have periodic Hochschild cohomology. Periodicity of an algebra, 
as well as its period, are invariant under derived equivalences. Therefore, to study periodic algebras we may assume that 
the algebras are basic and indecomposable. 

Finding or possibly classifying periodic algebras is an important problem as it has connections with group theory, topology, 
singularity theory, cluster algebras and algebraic combinatorics. For details, we refer to the survey article \cite{ES1} 
and the introductions of \cite{BES, ES2, ES4}. 

We are concerned with the classification of all periodic tame symmetric algebras. Dugas proved in \cite{Du1} that every 
representation-finite self-injective algebra, without simple blocks, is a periodic algebra. The representation-infinite, 
indecomposable, periodic algebras of polynomial growth were classified in \cite{BES}. It is conjectured in \cite[Problem]{ES2} 
that every indecomposable symmetric periodic tame algebra of non-polynomial growth is of period $4$. The large class of 
tame symmetric algebras of period $4$ is provided by the weighted surface algebras associated to compact real surfaces 
(these and their deformations are investigated in \cite{ES2}-\cite{ES6}). Surface triangulations have been also used to 
study cluster algebraic structures in Teichm\"uller theory \cite{FG, GSV}, cluster algebras of topological origin \cite{FST2}, 
and the classification of all cluster algebras of finite mutation type with skew symmetric exchange matrices \cite{FST1}. 
We also mention that there exist wild symmetric periodic algebras of period $4$, with arbitrary large number (at least 
$4$) of pairwise non-isomorphic simple modules. These wild periodic algebras arise as stable endomorphism rings of 
clustr-tilting Cohen-Macaulay modules over one-dimensional hypersurface singularities (see \cite{BIKR} and 
\cite[Corollary 2]{ES4}). 

Periodic algebras based on surface triangulations lead also to interesting non-periodic symmetric tame algebras. Namely, 
we get new symmetric tame algebras by taking the idempotent algebras $e\Lambda e$ of periodic surface algebras $\Lambda$. 
In particular, every Brauer graph algebra is of this form (see \cite[Theorem 4]{ES7}). We refer also to \cite{ES8} for 
the related classification of all algebras of generalized dihedral type. Summing up, the classification of all symmetric 
tame periodic algebras of period $4$ is currently an important problem. 

In order to present our main result we recall briefly the nature of weighted surface algebras, introduced and investigated 
in \cite{ES2, ES5, ES9}. By a surface we mean a connected, compact, $2$-dimensional real manifold $S$, with or without 
boundary. Then $S$ admits a structure of a finite $2$-dimensional triangular cell complex, and hence a triangulation. We 
say that $(S,\overrightarrow{T})$ is a directed triangulated surface if $S$ is a surface, $T$ a triangulation of $S$ 
with at least two edges, and $\overrightarrow{T}$ an arbitrary choice of orientations of triangles in $T$. To such 
$(S,\overrightarrow{T})$ one associates a triangulation quiver $(Q(S,\overrightarrow{T}),f)$, where $Q(S,\overrightarrow{T})$ 
is a $2$-regular quiver, that is every vertex is a source and target of exactly two arrows. The vertices of this quiver are 
the edges of $T$, and $f$ is a permutation of the arrows of $Q(S,\overrightarrow{T})$ reflecting the orientation 
$\overrightarrow{T}$ of triangles in $T$. Since $Q(S,\overrightarrow{T})$ is $2$-regular there is a second permutation, 
denoted by $g$, of the arrows of $Q(S,\overrightarrow{T})$. For the set $\mathcal{O}(g)$ of $g$-orbits of arrows in 
$Q(S,\overrightarrow{T})$, two functions $m_\bullet:\mathcal{O}(g)\to\mathbb{N}^*$ and $c_\bullet:\mathcal{O}(g)\to K^*$, 
called weight and parameter functions, are considered. Then, under some restrictions on $m_\bullet$ and $c_\bullet$, the 
weighted surface algebra $\Lambda(S,\overrightarrow{T},m_\bullet,c_\bullet)$ is defined as a quotient algebra 
$KQ(S,\overrightarrow{T})/I(S,\overrightarrow{T},m_\bullet,c_\bullet)$ of the path algebra $KQ(S,\overrightarrow{T})$ of 
$Q(S,\overrightarrow{T})$ over $K$ by an ideal $I(S,\overrightarrow{T},m_\bullet,c_\bullet)$ of $KQ(S,\overrightarrow{T})$. 
It has been proved in \cite[Theorems 1.1-1.3]{ES5} that, if $\Lambda=\Lambda(S,\overrightarrow{T},m_\bullet,c_\bullet)$ is 
a weighted surface algebra other than a singular disc, triangle, tetrahedral or spherical algebra, then $\Lambda$ is a 
tame symmetric periodic algebra of period $4$. We mention that the Gabriel quiver $Q_\Lambda$ of such an algebra is at 
most $2$-regular, which means that every vertex is a source and target of at most two arrows. Moreover, $Q_\Lambda$ is 
$2$-regular if and only if $Q_\Lambda=Q(S,\overrightarrow{T})$. This holds exactly when $m_{\mathcal{O}}|\mathcal{O}| 
\geqslant 3$ for any orbit $\mathcal{O}$ in $\mathcal{O}(g)$. In general, it is assumed only $m_\mathcal{O}|\mathcal{O}| 
\geqslant 2$ for any orbit $\mathcal{O}$ in $\mathcal{O}(g)$, and some other minor restrictions (see Section \ref{sec:2}). 
Conversly, it has been shown in \cite[Main Theorem]{ES4} that a basic indecomposable algebra $A$ with $2$-regular Gabriel 
quiver $Q_A$ having at least three vertices is a tame symmetric periodic algebra of period $4$ (more generally, algebra 
of generalized quaterion type) if and only if $A$ is socle equivalent to a weighted surface algebra 
$\Lambda(S,\overrightarrow{T},m_\bullet,c_\bullet)$, different from a singular tetrahedral algebra, or is isomorphic 
to a higher tetrahedral algebra $\Lambda(m,\lambda)$, $m\geqslant 2$, $\lambda\in K^*$ (investigated in \cite{ES3}). 
Recently, new exotic families of algebras derived equivalent to the algebras $\Lambda(m,\lambda)$ were discovered: 
the higher spherical algebras $S(m,\lambda)$ from \cite{ES6}, and the families $E(m,\lambda)$ and $F(m,\lambda)$ 
from \cite{Sk}. These four exotic families may be regarded as higher deformations of the known four families of the 
trivial extensions algebras of the tubular algebras of tubular type $(2,2,2,2)$, presented long time ago in 
\cite[Example 3.3]{S1}. Furthermore, new relevant exotic families occured in the recent authors work concerning 
the classification of algebras derived equivalent to higher tetrahedral algebras. We also note that some relevant 
exotic families of algebras with $2$ and $3$ simple modules occured in the Erdmann's classification of algebras of 
quaternion type \cite{E} (see also \cite{BIKR, ES0, Ho}). 

The aim of this article is to show that there are numerous tame symmetric periodic algebras of period $4$, with arbitrary 
large ranks of the Grothendieck group, which are not isomorphic to weighted surface algebras. In particular, this leads 
the study of derived equivalences of weighted surface algebras into new exciting directions. The idea is as follows. The 
general version of weighted surface algebras introduced in \cite{ES5} allows to consider the algebras $\Lambda=
\Lambda(S,\overrightarrow{T},m_\bullet,c_\bullet)$ whose triangulation quiver $(Q(S,\overrightarrow{T}),f)$ admits arrows 
which do not occur in their Gabriel quivers $Q_\Lambda$, called virtual arrows. For example, the directed triangulated 
surface $(S,\overrightarrow{T})$ may contain $2$-triangle discs 
$$
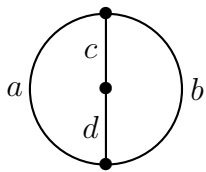
\begin{tikzpicture}
\draw[thick] (0,-1)--(0,1); 
\draw[thick](0,0) circle (1);
\node() at (0,0){$\bullet$};
\node() at (0,-1){$\bullet$}; 
\node() at (0,1){$\bullet$}; 
\node() at (1.2,0){$b$};
\node() at (-1.2,0){$a$};
\node() at (-0.2,0.5){$c$};
\node() at (-0.2,-0.5){$d$};
\end{tikzpicture}$$
with coherent orientation $(a\mbox{ }c\mbox{ }d)$ and $(c\mbox{ }b\mbox{ }d)$ of triangles, which provide in 
$Q(S,\overrightarrow{T})$ $g$-orbits of arrows $\xymatrix{c\ar@<0.1cm>[r]&d\ar@<0.1cm>[l]}$ of length $2$. Assume that 
$\mathcal{O}_1,\dots,\mathcal{O}_r$ is a family of pairwise different orbits in $\mathcal{O}(g)$ of length $2$ with 
$m_{\mathcal{O}_1}=\dots=m_{\mathcal{O}_r}=1$, and $\xi=(\xi_1,\dots,\xi_r)$ is an element in 
$\mathcal{O}_1\times\dots\times\mathcal{O}_r$. Then we define a new algebra $\Lambda(\xi)=\Lambda(S,\overrightarrow{T},m_\bullet,c_\bullet,\xi)$ 
by a quiver $Q(S,\overrightarrow{T},\xi)$ and relations, which we call a \emph{virtual mutation} of $\Lambda$ with respect 
to the chosen sequence $\xi$ of virtual arrows of $Q(S,\overrightarrow{T})$. 

The following theorem is the main result of this paper. 

\begin{mthm} Let $\Lambda(\xi)=\Lambda(S,\overrightarrow{T},m_\bullet,c_\bullet,\xi)$ be a virtual mutation of a 
weighted surface algebra $\Lambda=\Lambda(S,\overrightarrow{T},m_\bullet,c_\bullet)$. Then the following statements 
hold.
\begin{enumerate}[$(1)$]
\item $\Lambda(\xi)$ is a finite-dimensional symmetric algebra. 

\item $\Lambda(\xi)$ is derived equivalent to $\Lambda$.

\item $\Lambda(\xi)$ is not isomorphic to a weighted surface algebra.

\item $\Lambda(\xi)$ is a representation-infinite tame algebra.

\item $\Lambda(\xi)$ is a periodic algebra of period $4$.
\end{enumerate}
\end{mthm}

The quiver $Q(S,\overrightarrow{T},\xi)$ describing $\Lambda(\xi)$ has the same vertices as $Q(S,\overrightarrow{T})$, 
which are the edges of $T$. All arrows of the quiver $Q(S,\overrightarrow{T})$ which are not connected to the sources 
of the chosen virtual arrows $\xi_1,\dots,\xi_r$ are the arrows of $Q(S,\overrightarrow{T},\xi)$. On the other hand, 
the quiver $Q(S,\overrightarrow{T},\xi)$ contains (with few exceptions) vertices which are sources or targets of $3$ 
or $4$ arrows, so this explains the statement (3). We describe a canonical basis for the algebra $\Lambda(\xi)$ and 
then provide formula for its dimension over $K$ (see Section \ref{sec:3}). For the proof of (2) we construct in Section 
\ref{sec:5} a tilting complex $T^\xi$ in the homotopy category $K^b(P_\Lambda)$ of bounded complexes of projective 
modules in $\mod\Lambda$ and prove that the endomorphism algebra $\End_{K^b(P_\Lambda)}(T^\xi)$ is isomorphic to $\Lambda(\xi)$. 
Then the remaining statements of the above theorem follow from known general results. 

We note that in general the directed triangulated surface $(S,\overrightarrow{T})$ of a weighted surface algebra $\Lambda = 
\Lambda(S,\overrightarrow{T},m_\bullet,c_\bullet)$ may not contain $2$-triangle discs, and then $\Lambda$ does not admit 
a virtual mutation. In Section \ref{sec:6} we discuss the following construction of weighted surface algebras which admit 
virtual mutations. 

Let $\Lambda=\Lambda(S,\overrightarrow{T},m_\bullet,c_\bullet)$ be a weighted surface algebra, say with $S$ an orientable 
surface and $\overrightarrow{T}$ a coherent orientation of triangles in $T$ (due to \cite[Theorem 3.1]{ES7} it is not 
restriction of generality). We take a non-empty set $I$ of edges in $T$ (possibly all edges of $T$) and a function $\epsilon: 
I\to\{-1,1\}$. Then we may associate to $(S,\overrightarrow{T})$, in a unique way, a directed triangulated surface 
$(S,\overrightarrow{T_I})$, where $T_I$ is the new triangulation of $S$ obtained from $T$ by blowing-up of every edge 
of $I$ to a $2$-triangle disc, and in such a manner that $\overrightarrow{T}$ is extended to a coherent orientation 
$\overrightarrow{T_I}$ of triangles in $T_I$. This creates canonically a new weighted surface algebra $\Lambda_I=
\Lambda(S,\overrightarrow{T_I},m^I_\bullet,c^I_\bullet)$ and a sequence $\underline{\epsilon}=(\epsilon_i)_{i\in I}$ of 
virtual arrows of $Q(S,\overrightarrow{T_I})$, defined by the function $\epsilon$. Then the virtual mutation $\Lambda_I^\epsilon=
\Lambda_I(\underline{\epsilon})$ of $\Lambda_I$ with respect to $\underline{\epsilon}$ is defined as a \emph{deformation} 
of $\Lambda$ at the set of edges $I$, with respect to function $\epsilon$. 

We have the following consequence of the Main Theorem. 

\begin{cor1} Let $\Lambda=\Lambda(S,\overrightarrow{T},m_\bullet,c_\bullet)$ be a weighted surface algebra, $I$ a set of 
edges of $T$, $\epsilon:I\to\{-1,1\}$ a function, and $\Lambda_I^\epsilon$ the associated deformation of $\Lambda$. Then 
$\Lambda_I^\epsilon$ is a finite-dimensional, tame, symmetric and periodic algebra of period $4$. Moreover, $\Lambda_I^\epsilon$ 
is not isomorphic to a weighted surface algebra.
\end{cor1} 

The paper is organized as follows. In Sections \ref{sec:1} and \ref{sec:2} we present necessary background on derived 
equivalences of algebras and weighted triangulation (surface) algebras, needed for further parts of the article. In 
Section \ref{sec:3} we introduce and study the virtual mutations of weighted triangulation (surface) algebras. 
Section \ref{sec:4} presents several examples illustrating the concept of a virtual mutation of a weighted surface algebra. 
Section \ref{sec:5} is devoted to the proof of Main Theorem, and two complementary results justifying the name and definition 
of the main object investigated in the paper. In final Section \ref{sec:6} we introduce deformations of weighted surface 
algebras with respect to collections of edges of the underlying surface triangulation. 

For general background on the relevant representation theory we refer to the books \cite{ASS, E, Ha2, SS, SY} and the 
survey article \cite{S2}. 

\section{Derived equivalences of algebras}\label{sec:1}

In this section, we recall basic facts on derived equivalences of algebras, needed in our article. 

For an algebra $A$, we denote by $K^b(\mod A)$ the homotopy category of bounded complexes of modules 
in $\mod A$ and by $K^b(P_A)$ its subcategory formed by bounded complexes of projective modules. The 
derived category $D^b(\mod A)$ of $A$ is the localization of $K^b(\mod A)$ with respect to quasi-isomorphisms, 
and admits structure of a triangulated category, where the suspension functor is given by left shift 
$(-)[1]$ (see \cite{Ha1}). Two algebras $A$ and $B$ are called \emph{derived equivalent} provided 
their derived categories $D^b(\mod A)$ and $D^b(\mod B)$ are equivalent as triangulated categories. 
Moreover, a complex $T$ in $K^b(P_A)$ is called a \emph{tilting complex} \cite{Ric1}, if the 
following conditions are satisfied: 
\begin{enumerate}
\item[(T1)] $\Hom_{K^b(P_A)}(T,T[i])=0$, for all integers $i\neq 0$, 
\item[(T2)] $\add(T)$ generates $K^b(P_A)$ as triangulated category. 
\end{enumerate} 

We have the following handy criterion for verifying derived equivalence of algebras, due to Rickard 
\cite[Theorem 6.4]{Ric1}.

\begin{theorem}\label{thm:1.1} Two algebras $A$ and $B$ are derived equivalent if and only if there exists 
a tilting complex $T$ in $K^b(P_A)$ such that $\End_{K^b(P_A)}\cong B$. 
\end{theorem}

We recall also the following two theorems (see \cite[Corollary 5.3]{Ric3} and \cite[Theorem 2.9]{ES1}).

\begin{theorem}\label{thm:1.2} Let $A$ and $B$ be derived equivalent algebras. Then $A$ is symmetric if 
and only if $B$ is symmetric.
\end{theorem}

\begin{theorem}\label{thm:1.3} Let $A$ and $B$ be derived equivalent algebras. Then $A$ is periodic if 
and only if $B$ is periodic. Moreover, if this is the case, then $A$ and $B$ have the same period. 
\end{theorem}

Beacause in the class of self-injective algebras, derived equivalence implies stable equivalence (see 
\cite[Corollary 2.2]{Ric2} and \cite[Corollary 5.3]{Ric3}), one may conclude from \cite[Theorems 4.4 and 5.6]{CB} 
and \cite[Corollary 2]{KZ} that the following theorem holds.

\begin{theorem}\label{thm:1.4} Let $A$ and $B$ be derived equivalent self-injective algebras. Then the 
following equivalences are valid.
\begin{enumerate}[$(1)$]
\item $A$ is tame if and only if $B$ is tame.
\item $A$ is of polynomial growth if and only if $B$ is of polynomial growth. 
\end{enumerate}
\end{theorem}

We present now a simple construction of tilting complexes of lenght 2 over symmetric algebras, observed 
first by Okuyama \cite{O} and Rickard \cite{Ric2}. These tilting complexes have been used extensively to 
verify various cases of Brou\'e's abelian defect group conjecture \cite{CR}, as well as in realizing 
derived equivalences between symmetric algebras (see \cite{BiHS}, \cite{BoHS}, \cite{Ho}, \cite{Ka}, 
\cite{MS}, \cite{Ric2}). 

Let $A$ be a basic, indecomposable, symmetric algebra with the Grothendieck group $K_0(A)$ of rank at 
least 2 and $A=A_A=P\oplus Q$ a proper decomposition in $\mod A$. Consider a left $\add(Q)$-approximation 
$f:P\to Q'$ of $P$, that is $Q'$ is a module in $\add(Q)$ and $f$ induces surjective map 
$$\Hom_A(f,Q''):\Hom_A(Q',Q'')\to\Hom(P,Q''),$$ 
for any module $Q''$ in $\add(Q)$. 

In this case, we may consider two complexes 
$$T_1:\qquad \xymatrix{0\ar[r]& Q \ar[r] & 0}\quad\mbox{and}\quad T_2:
\xymatrix{0\ar[r]&P\ar[r]^{f}&Q'\ar[r]&0}$$ 
concentrated, respectively in degree $0$ and in degrees $1$ and $0$. Then we have the following proposition 
(for a proof we refer to \cite[Proposition 2.1]{Du2}). 

\begin{prop}\label{pro:1.5}
$T:=T_1\oplus T_2$ is a tilting complex in $K^b(P_A)$. 
\end{prop} 

\section{Weighted triangulation algebras}\label{sec:2} 

A quiver is a quadruple $Q=(Q_0,Q_1,s,t)$ consisting of a finite set $Q_0$ of vertices, a finite set $Q_1$ of arrows 
and two maps $s,t:Q_1\to Q_0$ assigning to each arrow $\alpha\in Q_1$ its source $s(\alpha)\in Q_0$ and target 
$t(\alpha)\in Q_0$. We denote by $KQ$ the path algebra of $Q$ over $K$, where underlying $K$-vector space has as 
its basis the set of all paths in $Q$ of length $\geqslant 0$, and by $R_Q$ the arrow ideal of $KQ$ generated by all 
paths in $Q$ of length $\geqslant 1$. For a vertex $i\in Q_0$, let $e_i$ be the path of length $0$ at $i$, and then 
$e_i$ are pairwise orthogonal idempotents, which sum up to identity of $KQ$. We will consider algebras of the form 
$A=KQ/I$, where $I$ is an ideal of $KQ$ such that $R_Q^m\subseteq I\subseteq R_Q$ for some $m\geqslant 2$, so that 
$A$ will be a basic finite-dimensional algebra. Then the \emph{Gabriel quiver} $Q_A$ of $A$ is the full subquiver 
of $Q$ obtained by removing all arrows $\alpha$ with $\alpha+I\in R_Q^2+I$.

A quiver $Q$ is $2$-regular if for each vertex $i\in Q_0$ there are precisely two arrows with source $i$ and precisely 
two arrows with target $i$. Such a quiver $Q$ has \emph{involution} $\bar{ }:Q_1\to Q_1$, which is a function 
assigning to each arrow $\alpha\in Q_1$ the unique arrow $\bar{\alpha}\neq \alpha$ with $s(\alpha)=s(\bar{\alpha})$. 

Following \cite{ES2, La}, a triangulation quiver is a pair $(Q,f)$, where $Q=(Q_0,Q_1,s,t)$ is a connected $2$-regular 
quiver and $f:Q_1\to Q_1$ is a permutation such that $s(f(\alpha))=t(\alpha)$, for any arrow $\alpha\in Q_1$, and 
$f^3$ is the identity on $Q_1$. In particular, all cycles in $Q_1$ induced by $f$ (that is the orbits of $f$) have length 
$1$ or $3$. We also assume that $|Q_0|\geqslant 2$. It was shown in \cite[Theorem 4.11]{ES2} (see also 
\cite[Theorem 3.1]{ES7}) that every triangulation quiver $(Q,f)$ is the triangulation quiver 
$(Q(S,\overrightarrow{T}),f)$ coming from a triangulation $T$ of a compact connected real surface $S$, with or without 
boundary, and where $\overrightarrow{T}$ is an arbitrary choice of orientation of triangles in $T$. We may even 
assume that $S$ is an orientable surface and $\overrightarrow{T}$ is a coherent orientation of triangles in $T$. We refer 
to Section \ref{sec:6} for more details.

Let $(Q,f)$ be a triangulation quiver. Then we have the composed permutation $g:Q_1\to Q_1$, where $g(\alpha)=\overline{f(\alpha)}$, 
if $\alpha\in Q_1$. For each arrow $\alpha\in Q_1$, we denote by $\mathcal{O}(\alpha)$ the $g$-orbit of $\alpha$ in $Q_1$, 
and set $n_\alpha = n_{\mathcal{O}(\alpha)}=|\mathcal{O}(\alpha)|$. Hence the $g$-orbit $\mathcal{O}(\alpha)$ is of the 
form $\mathcal{O}(\alpha)=(\alpha\mbox{ }g(\alpha)\mbox{ }\dots g^{n_\alpha-1}(\alpha))$. We note that $n_\alpha$ may be 
arbitrary large natural number. We write $\mathcal{O}(g)$ for the set of all $g$-orbits in $Q_1$. Following \cite{ES2}, we 
call a function 
$$m_\bullet:\mathcal{O}(g)\to \mathbb{N}^*:=\mathbb{N}\setminus\{0\}$$ 
a \emph{weight function} of $(Q,f)$, and a function 
$$c_\bullet:\mathcal{O}(g)\to K^*:=K\setminus\{0\}$$
a \emph{parameter function} of $(Q,f)$. We write briefly $m_\alpha:=m_{\mathcal{O}(\alpha)}$ and $c_\alpha:=c_{\mathcal{O}(\alpha)}$, 
for $\alpha\in Q_1$.

\begin{df}\label{def:2.1} We say that an arrow $\alpha\in Q_1$ is \emph{virtual} if $m_\alpha n_\alpha=2$. Note that this 
condition is preserved under permutation of $g$, and the virtual arrows form $g$-orbits of size $1$ or $2$. 
\end{df}

We have also the following general assumption \cite[Assumption 2.7]{ES5}. 

\begin{assm}\label{ass} We assume that a weight function $m_\bullet$ of $(Q,f)$ satisfies the following conditions: 
\begin{enumerate}[(1)]
\item $m_\alpha n_\alpha\geqslant 2$ for all arrows $\alpha$,
\item $m_\alpha n_\alpha\geqslant 3$ for all arrows $\alpha$ such that $\bar{\alpha}$ is virtual but not a loop, 
\item $m_\alpha n_\alpha\geqslant 4$ for all arrows $\alpha$ such that $\bar{\alpha}$ is a virtual loop. 
\end{enumerate}
\end{assm}

For each arrow $\alpha\in Q_1$, we consider the path 
$$A_\alpha:=\alpha g(\alpha) \dots g^{m_\alpha n_\alpha-2}(\alpha)$$
along the $g$-cycle of $\alpha$ of length $m_\alpha n_\alpha-1$, 
and the $g$-cycle 
$$B_\alpha:=\alpha g(\alpha) \dots g^{m_\alpha n_\alpha-1}(\alpha)$$ 
of $\alpha$ of length $m_\alpha n_\alpha$. We observe that $B_\alpha=A_\alpha g^{n_{\alpha}-1}(\alpha)$. Moreover, if 
$n_\alpha\geqslant 3$, we consider also the path 
$$A'_\alpha:=\alpha g(\alpha) \dots g^{m_\alpha n_\alpha-3}(\alpha)$$ 
along the $g$-cycle of $\alpha$ of length $m_\alpha n_\alpha-2$. Let us only mention that $f^2=g^{n_\alpha-1}(\bar{\alpha})$, 
for any arrow $\alpha\in Q_1$ (see \cite[Lemma 5.3]{ES2}). 

The definition of a weighted triangulation algebra looks as follows. 

\begin{df}\label{def:2.3} The algebra $\Lambda=\Lambda(Q,f,m_\bullet,c_\bullet)=KQ/I$ is a said to be a \emph{weighted triangulation 
algebra} if $(Q,f)$ is a triangulation quiver and $I=I(Q,f,m_\bullet,c_\bullet)$ is the ideal of $KQ$ generated by the following 
relations:
\begin{enumerate}[(1)]
\item $\alpha f(\alpha)-c_{\bar{\alpha}}A_{\bar{\alpha}}$, for any arrow $\alpha\in Q_1$, 
\item $\alpha f(\alpha)g(f(\alpha))$, for all arrows $\alpha\in Q_1$ unless $f^2(\alpha)$ is virtual or unless $f(\bar{\alpha})$ 
is virtual with $m_{\bar{\alpha}}=1$ and $n_{\bar{\alpha}}=3$,
\item $\alpha g(\alpha)f(g(\alpha))$, for all arrows $\alpha\in Q_1$ unless $f(\alpha)$ is virtual or unless $f^2(\alpha)$ 
is virtual with $m_{f(\alpha)}=1$ and $n_{f(\alpha)}=3$.
\end{enumerate}
\end{df}

We note that the relations (2) and (3) are corrections of the relations (2) and (3) in \cite[Definition 2.8]{ES5} (see \cite{ES9} 
for a discussion). 

The following theorem is a consequence of the main results proved in \cite{ES5}. 

\begin{theorem}\label{thm:2.4} Let $\Lambda=\Lambda(Q,f,m_\bullet,c_\bullet)$ be a weighted triangulation algebra other than a 
singular disc, triangle, tetrahedral or spherical algebra. Then the following statements hold:
\begin{enumerate}[$(1)$]
\item $\Lambda$ is a finite-dimensional algebra with $\dim_K\Lambda=\sum_{\mathcal{O}\in\mathcal{O}(g)}m_\mathcal{O}n^2_\mathcal{O}$. 
\item $\Lambda$ is a symmetric algebra. 
\item $\Lambda$ is a tame algebra.
\item $\Lambda$ is a periodic algebra of period $4$. 
\end{enumerate}
\end{theorem}

\begin{df}\label{def:2.5} Let $(S,\overrightarrow{T})$ be a directed triangulated surface, $(Q(S,\overrightarrow{T}),f)$ the 
associated triangulation quiver and $m_\bullet,c_\bullet$  weight and parameter functions of $(Q(S,\overrightarrow{T}),f)$. Then 
the weighted triangulation algebra 
$$\Lambda(Q(S,\overrightarrow{T}),f,m_\bullet,c_\bullet)$$ 
is called a \emph{weighted surface algebra}, and it is denoted by $\Lambda(S,\overrightarrow{T},m_\bullet,c_\bullet)$.
\end{df}

We recall also the following description of bases of indecomposable projective modules over a weighted triangulation algebra, 
established in \cite[Lemma 4.7]{ES5}.

\begin{prop}\label{prop:2.6} Let $\Lambda=\Lambda(Q,f,m_\bullet,c_\bullet)$ be a weighted triangulation algebra, $i$ a vertex 
of $Q$ and $\alpha,\bar{\alpha}$ the two arrows in $Q_1$ starting at $i$. Then the following statements hold. 
\begin{enumerate}[$(1)$]
\item Assume that $\alpha$ is virtual. Then the module $e_i\Lambda$ has basis $\mathcal{B}_i$ formed by all initial submonomials of 
$B_{\bar{\alpha}}$ together with $e_i$ and $\bar{\alpha}f(\bar{\alpha})$.

\item If $\alpha$ and $\bar{\alpha}$ are not virtual, then $e_i\Lambda$ has basis $\mathcal{B}_i$ formed by all proper initial submonomials 
of $B_\alpha$ and $B_{\bar{\alpha}}$ together with $e_i$ and $B_\alpha$. 

\item We have the equalities 
$$\alpha f(\alpha)f^2(\alpha)=c_\alpha B_\alpha=c_{\bar{\alpha}}B_{\bar{\alpha}}=\bar{\alpha}f(\bar{\alpha})f^2(\bar{\alpha})$$ 
and this element generates the socle of $e_i\Lambda$.
\end{enumerate}
\end{prop}

\begin{exmp}\label{ex:2.7} Following \cite{ES8}, by a \emph{$2$-triangle disc} we mean the unit disc $D=D^2$ in $\mathbb{R}^2$ with the 
triangulation $T$ given by two triangles 
$$
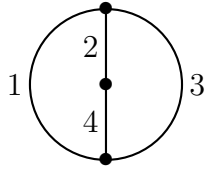
\begin{tikzpicture}
\draw[thick] (0,-1)--(0,1); 
\draw[thick](0,0) circle (1);
\node() at (0,0){$\bullet$};
\node() at (0,-1){$\bullet$}; 
\node() at (0,1){$\bullet$}; 
\node() at (1.2,0){$3$};
\node() at (-1.2,0){$1$};
\node() at (-0.2,0.5){$2$};
\node() at (-0.2,-0.5){$4$};
\end{tikzpicture}$$
with $1$ and $3$ boundary edges, and the coherent orientation $\overrightarrow{T}$: $(1\mbox{ }2\mbox{ }4)$ and $(2\mbox{ }3\mbox{ }4)$ 
of these two triangles. Then the associated triangulation quiver $(Q(D,\overrightarrow{T}),f)$ is the following quiver 

$$\qquad\qquad\qquad
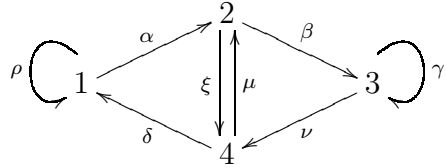
\xymatrix@R=0.01cm@C=1.5cm{
&&2\ar[rddd]^{\beta}\ar@<-0.1cm>[dddddd]_{\xi}&& \\
&&&& \\
&\ar@(ul,dl)[d]_{\rho}&&\ar@(ur,dr)[d]^{\gamma} & \\ 
&1\ar[ruuu]^{\alpha}&&3\ar[lddd]^{\nu}&
&&&& \\
&&&& \\ 
&&&& \\
&&4\ar[luuu]^{\delta}\ar@<-0.1cm>[uuuuuu]_{\mu}&&}$$ 
with $f$-orbits $(\alpha\mbox{ }\xi\mbox{ }\delta)$, $(\beta\mbox{ }\nu\mbox{ }\mu)$, $(\rho)$ and $(\gamma)$. We mention that 
the associated permutation $g=\bar{f}$ has two orbits $\mathcal{O}(\alpha)=
(\alpha\mbox{ }\beta\mbox{ }\gamma\mbox{ }\nu\mbox{ }\delta\mbox{ }\rho)$ and $\mathcal{O}(\xi)=(\xi\mbox{ }\mu)$. 

Let $m\in\mathbb{N}^*$ and $\lambda\in K^*$. Consider the weight function $m_\bullet:\mathcal{O}(g)\to\mathbb{N}^*$ and the 
parameter function $c_\bullet:\mathcal{O}(g)\to K^*$ such that $m_{\mathcal{O}(\alpha)}=m$, $m_{\mathcal{O}(\xi)}=1$, 
$c_{\mathcal{O}(\alpha)}=\lambda$, $c_{\mathcal{O}(\xi)}=1$, and the associated weighted triangulation algebra 
$$D(m,\lambda)=\Lambda(D,\overrightarrow{T},m_\bullet,c_\bullet)=\Lambda(Q(D,\overrightarrow{T}),f,m_\bullet,c_\bullet),$$ 
which we call the \emph{disc algebra of degree $m$}. 

We note that $\xi$ and $\mu$ are virtual arrows, and $\xi=\beta\nu$ and $\mu=\delta\alpha$ in $D(m,\lambda)$, so 
the disc algebra $D(m,\lambda)$ is given by its Gabriel quiver 
$$\qquad\qquad\qquad
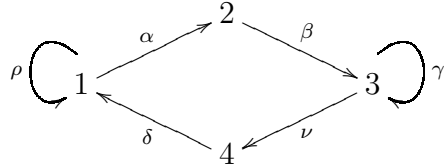
\xymatrix@R=0.01cm@C=1.5cm{
&&2\ar[rddd]^{\beta}&& \\
&&&& \\
&\ar@(ul,dl)[d]_{\rho}&&\ar@(ur,dr)[d]^{\gamma} & \\ 
&1\ar[ruuu]^{\alpha}&&3\ar[lddd]^{\nu}&
&&&& \\
&&&& \\ 
&&&& \\
&&4\ar[luuu]^{\delta}&&}$$
and the relations 
$$\alpha\beta\nu=\lambda(\rho\alpha\beta\gamma\nu\delta)^{m-1}\rho\alpha\beta\gamma\nu,\mbox{ }\beta\nu\delta= 
\lambda(\beta\gamma\nu\delta\rho\alpha)^{m-1}\beta\gamma\nu\delta\rho,$$
$$\delta\alpha\beta=\lambda(\delta\rho\alpha\beta\gamma\nu)^{m-1}\delta\rho\alpha\beta\gamma,\mbox{ }\nu\delta\alpha= 
\lambda(\gamma\nu\delta\rho\alpha\beta)^{m-1}\gamma\nu\delta\rho\alpha,$$
$$\rho^2=\lambda(\alpha\beta\gamma\nu\delta\rho)^{m-1}\alpha\beta\gamma\nu\delta,\mbox{ }\gamma^2= 
\lambda(\nu\delta\rho\alpha\beta\gamma)^{m-1}\nu\delta\rho\alpha\beta,$$
$$\alpha\beta\nu\delta\alpha=0,\mbox{ } \beta\nu\delta\rho=0,\mbox{ } \rho^2\alpha=0,$$
$$\nu\delta\alpha\beta\nu=0,\mbox{ } \delta\alpha\beta\gamma=0,\mbox{ }\gamma^2\nu=0,$$
$$\beta\nu\delta\alpha\beta=0,\mbox{ } \rho\alpha\beta\nu=0,\mbox{ } \delta\rho^2=0,$$
$$\delta\alpha\beta\nu\delta=0,\mbox{ }\gamma\nu\delta\alpha=0,\mbox{ }\beta\gamma^2=0.$$

Observe also that we have no zero relations of the forms: 
$$\delta\alpha\beta=\delta f(\delta)g(f(\delta))=0,\quad\mbox{because }f^2(\delta)=\xi\mbox{ is virtual,}$$
$$\beta\nu\delta=\beta f(\beta) g(f(\beta))=0,\quad\mbox{because }f^2(\beta)=\mu \mbox{ is virtual,}$$
$$\alpha\beta\nu=\alpha g(\alpha) f(g(\alpha))=0,\quad\mbox{because }f(\alpha)=\xi\mbox{ is virtual, and }$$
$$\nu\delta\alpha=\nu g(\nu) f(g(\nu))=0,\quad\mbox{because }f(\nu)=\mu\mbox{ is virtual.}$$

Moreover, $\dim_K D(m,\lambda)=m\cdot 6^2+2^2=36m+4$ and, according to Proposition \ref{prop:2.6}, $D(m,\lambda)$ admits basis 
$\mathcal{B}=\mathcal{B}_1\cup \mathcal{B}_2\cup \mathcal{B}_3\cup \mathcal{B}_4$, where 
$$\mathcal{B}_1=\left\{(\alpha\beta\gamma\nu\delta\rho)^r\alpha,(\alpha\beta\gamma\nu\delta\rho)^r\alpha\beta,(\alpha\beta\gamma\nu\delta\rho)^r\alpha\beta\gamma,(\alpha\beta\gamma\nu\delta\rho)^r\alpha\beta\gamma\nu,\right.$$ 
$$\left.(\alpha\beta\gamma\nu\delta\rho)^r\alpha\beta\gamma\nu\delta;\,0\leqslant r\leqslant m-1\right\}\cup
\left\{(\rho\alpha\beta\gamma\nu\delta)^r\rho,(\rho\alpha\beta\gamma\nu\delta)^r\rho\alpha\right.,$$
$$\left.(\rho\alpha\beta\gamma\nu\delta)^r\rho\alpha\beta,(\rho\alpha\beta\gamma\nu\delta)^r\rho\alpha\beta\gamma,(\rho\alpha\beta\gamma\nu\delta)^r\rho\alpha\beta\gamma\nu;\, 0\leqslant r\leqslant m-1\right\}$$
$$\cup\left\{e_1,(\alpha\beta\gamma\nu\delta\rho)^s,(\rho\alpha\beta\gamma\nu\delta)^t;\, 1\leqslant s\leqslant m, 1\leqslant t\leqslant m-1\right\},$$ \smallskip
$$\mathcal{B}_2=\left\{(\beta\gamma\nu\delta\rho\alpha)^r\beta,(\beta\gamma\nu\delta\rho\alpha)^r\beta\gamma,
(\beta\gamma\nu\delta\rho\alpha)^r\beta\gamma\nu,(\beta\gamma\nu\delta\rho\alpha)^r\beta\gamma\nu\delta,\right.$$
$$\left.(\beta\gamma\nu\delta\rho\alpha)^r\beta\gamma\nu\delta\rho;\,0\leqslant r\leqslant m-1 \right\}
\cup\left\{e_2,\beta\nu,(\beta\gamma\nu\delta\rho\alpha)^s;\,1\leqslant s\leqslant m\right\},$$ \smallskip
$$\mathcal{B}_3=\left\{(\nu\delta\rho\alpha\beta\gamma)^t\nu,(\nu\delta\rho\alpha\beta\gamma)^t\nu\delta, 
(\nu\delta\rho\alpha\beta\gamma)^t\nu\delta\rho,(\nu\delta\rho\alpha\beta\gamma)^t\nu\delta\rho\alpha,\right.$$
$$\left.(\nu\delta\rho\alpha\beta\gamma)^t\nu\delta\rho\alpha\beta;\,0\leqslant r\leqslant m-1\right\}\cup 
\left\{(\gamma\nu\delta\rho\alpha\beta)^t\gamma,(\gamma\nu\delta\rho\alpha\beta)^t\gamma\nu, \right.$$
$$\left.(\gamma\nu\delta\rho\alpha\beta)^t\gamma\nu\delta,(\gamma\nu\delta\rho\alpha\beta)^t\gamma\nu\delta\rho,(\gamma\nu\delta\rho\alpha\beta)^t\gamma\nu\delta\rho\alpha;\, 0\leqslant r\leqslant m-1 \right\}$$
$$\cup\left\{e_3,(\nu\delta\rho\alpha\beta\gamma)^s,(\gamma\nu\delta\rho\alpha\beta)^t;\, 1\leqslant s\leqslant m, 1\leqslant t\leqslant m-1\right\},$$ \smallskip
$$\mathcal{B}_4=\left\{(\delta\rho\alpha\beta\gamma\nu)^r\delta,(\delta\rho\alpha\beta\gamma\nu)^r\delta\rho, 
(\delta\rho\alpha\beta\gamma\nu)^r\delta\rho\alpha,(\delta\rho\alpha\beta\gamma\nu)^r\delta\rho\alpha\beta, 
\right.$$
$$\left.(\delta\rho\alpha\beta\gamma\nu)^r\delta\rho\alpha\beta\gamma;\,0\leqslant r\leqslant m-1\right\}\cup
\left\{e_4,\delta\alpha,(\delta\rho\alpha\beta\gamma\nu)^s;\, 1\leqslant s\leqslant m \right\}.$$
In particular, the Cartan matrix $C_{D(m,\lambda)}$ of $D(m,\lambda)$ is of the form
$$\left[\begin{array}{cccc} 4m & 2m & 4m & 2m \\ 2m & m+1 & 2m & m+1 \\ 4m & 2m & 4m & 2m \\ 2m & m+1 & 2m & m+1
\end{array}\right].$$ \end{exmp}

\section{Virtual mutations of weighted surface algebras}\label{sec:3}

The aim of this section is to introduce the main objective of the paper, namely the virtual mutations of weighted triangulation 
(surface) algebras, and describe their linear bases. In particular, we will show that these are finite-dimensional algebras and 
determine their dimensions. 

Let $(Q,f)$ be a triangulation quiver, $Q=(Q_0,Q_1,s,t)$, $g=\bar{f}$ the associated permutation of $Q_1$, and $\mathcal{O}(g)$ 
the set of all $g$-orbits in $Q_1$. Moreover, let $m_\bullet:\mathcal{O}(g)\to\mathbb{N}^*$ be a weight function, $c_\bullet: 
\mathcal{O}(g)\to K^*$ a parameter function, and $\Lambda=\Lambda(Q,f,m_\bullet,c_\bullet)$ the associated weighted triangulation 
algebra. We keep the notation introduced in the previous section. We also assume that $|Q_0|\geqslant 4$ and that $\Lambda$ is 
not a singular spherical algebra, introduced in \cite[Example 3.6]{ES5}. 

Assume that $\mathcal{O}(g)$ contains a family $\mathcal{O}_1,\dots,\mathcal{O}_r$ of orbits with $|\mathcal{O}_i|=2$ and 
$m_{\mathcal{O}_i}=1$, for any $i\in\{1,\dots,r\}$ (we mention that $\mathcal{O}(g)$ may contain other virtual orbits 
besides the chosen ones). For a given element 
$$\xi=(\xi_1,\dots,\xi_r)\in\mathcal{O}_1\times\dots\times\mathcal{O}_r$$ 
we shall define a virtual mutation 
$$\Lambda(\xi)=\Lambda(Q,f,m_\bullet,c_\bullet,\xi)$$  
of $\Lambda$ with respect to the sequence $\xi$ of virtual arrows.

Observe that, for each $i\in\{1,\dots,r\}$, the triangulation quiver $(Q,f)$ contains a subquiver of the form 
$$
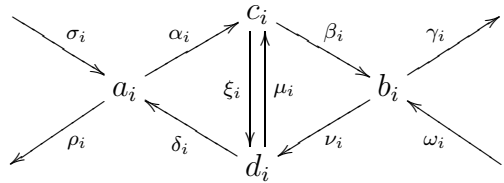
\xymatrix@R=0.5cm@C=1.2cm{\ar[rd]^{\sigma_i}&&c_i\ar[rd]^{\beta_i}\ar@<-0.1cm>[dd]_{\xi_i}&& \\ 
&a_i\ar[ld]^{\rho_i}\ar[ru]^{\alpha_i}&&b_i\ar[ld]^{\nu_i}\ar[ru]^{\gamma_i} 
& \\ &&d_i\ar[lu]^{\delta_i}\ar@<-0.1cm>[uu]_{\mu_i}&&\ar[lu]^{\omega_i} }$$
with $f$-orbits $(\alpha_i\mbox{ }\xi_i\mbox{ }\delta_i)$ and $(\beta_i\mbox{ }\nu_i\mbox{ }\mu_i)$ and $f(\sigma_i)=\rho_i$, 
$f(\omega_i)=\gamma_i$, $g(\sigma_i)=\alpha_i$, $g(\alpha_i)=\beta_i$, $g(\beta_i)=\gamma_i$, $g(\omega_i)=\nu_i$, $g(\nu_i)=\delta_i$ and 
$g(\delta_i)=\rho_i$. We also note that, because of assumption $|Q_0|\geqslant 4$, the $g$-orbits 
$\mathcal{O}(\alpha_i)=(\alpha_i\mbox{ }\beta_i\mbox{ }\gamma_i\dots\sigma_i)$ 
and $\mathcal{O}(\nu_i)=(\nu_i\mbox{ }\delta_i\mbox{ }\rho_i\dots\omega_i)$ are of length at least $3$, and may coincide (as 
shown in Example \ref{ex:2.7}). But we have always $|\mathcal{O}(\alpha_i)|\geqslant 4$ or $|\mathcal{O}(\nu_i)|\geqslant 4$, since 
otherwise, due to $2$-regularity of $Q$, there is an $f$-orbit of length $2$, which is impossible. The special cases 
$|\mathcal{O}(\alpha_i)|=3$ or $|\mathcal{O}(\nu_i)|=3$ are as follows 
$$
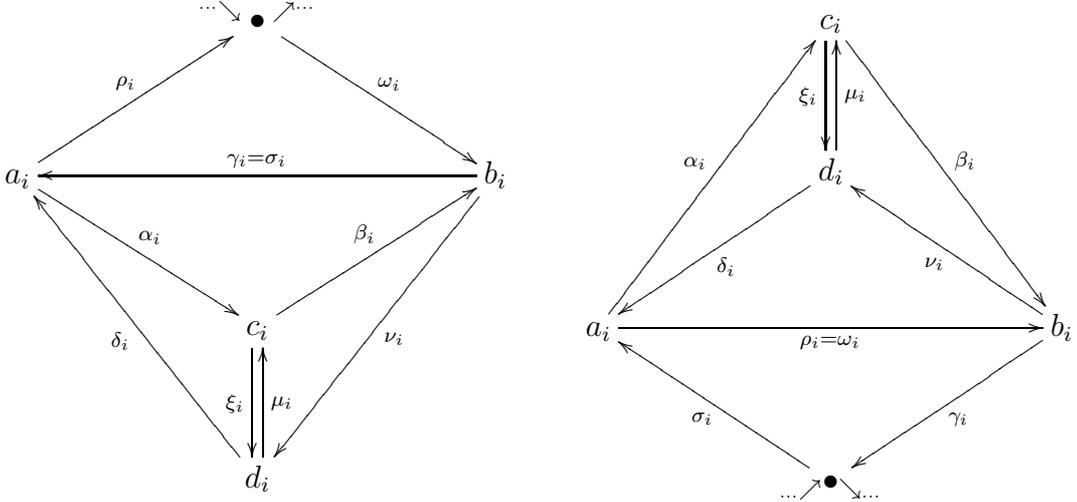
\xymatrix@R=1.5cm@C=2cm{&{}^{{}^{\dots}\searrow}\bullet{}^{\nearrow^{\dots}} \ar[rd]^{\omega_i}& \\
a_i\ar[ru]^{\rho_i}\ar[rd]^{\alpha_i}& &  b_i\ar[ll]_{\gamma_i=\sigma_i}\ar[ldd]^{\nu_i} \\ 
&c_i\ar[ru]^{\beta_i}\ar@<-0.07cm>[d]_{\xi_i}& \\ 
&d_i\ar@<-0.07cm>[u]_{\mu_i}\ar[luu]^{\delta_i}&}\qquad
\xymatrix@R=1.5cm@C=2cm{&c_i\ar[rdd]^{\beta_i}\ar@<-0.07cm>[d]_{\xi_i} & \\ 
& d_i\ar[ld]^{\delta_i}\ar@<-0.07cm>[u]_{\mu_i} \\ 
a_i\ar[ruu]^{\alpha_i}\ar[rr]_{\rho_i=\omega_i} & & \ar[ld]^{\gamma_i} b_i\ar[lu]^{\nu_i} \\ 
& \ar[lu]^{\sigma_i} {}_{{}_{\dots}\nearrow}\bullet_{\searrow_{\dots}} &}$$ 

Further, since the arrows $\xi_i$ and $\mu_i$ are virtual, we have in $\Lambda$ the equalities
$$\delta_i\alpha_i=c_{\mu_i}\mu_i\mbox{ and }\beta_i\nu_i=c_{\xi_i}\xi_i,\mbox{ with }c_{\xi_i}=c_{\mu_i}=c_{\mathcal{O}_i}.$$
Replacing $\xi_i$ by $c_{\xi_i}\xi_i$ and $\mu_i$ by $c_{\mu_i}\mu_i$, we may assume that $c_{\mathcal{O}_i}=1$, for all $i\in 
\{1,\dots,r\}$. We also note that in the presentation of $\Lambda$ by its Gabriel quiver $Q_\Lambda$ and the induced relations, 
the virtual arrows $\xi_i$ and $\mu_i$, $i\in\{1,\dots,r\}$, are removed (as well as all other virtual arrows of $Q$). 

We will define the algebra $\Lambda(\xi)=\Lambda(Q,f,m_\bullet,c_\bullet,\xi)$ by a quiver $Q(\xi)$ and a set of relations, 
keeping triangle nature of the most of the relations defining $\Lambda$. There are also added some new types of the relations. 

The quiver $Q(\xi)=(Q(\xi)_0,Q(\xi)_1,s,t)$ is defined in the following way. We take $Q(\xi)_0=Q_0$ and the set $Q(\xi)_1$ of 
arrows is obtained from the set of arrows $Q_1$ by three types of operations:
\begin{itemize}
\item removing the virtual arrows $\xi_i$ and $\mu_i$, for $i\in\{1,\dots,r\}$, 

\item replacing the arrows $\xymatrix{a_i\ar[r]^{\alpha_i}&c_i}$ and $\xymatrix{c_i\ar[r]^{\beta_i}&b_i}$ by arrows 
$\xymatrix{c_i\ar[r]^{\alpha_i}&a_i}$ and $\xymatrix{b_i\ar[r]^{\beta_i}&c_i}$, for $i\in\{1,\dots,r\}$, 

\item adding the arrows $\xymatrix{a_i\ar[r]^{\tau_i}&b_i}$, for $i\in\{1,\dots,r\}$.
\end{itemize}
Therefore, for any $i\in\{1,\dots,r\}$, the quiver $Q(\xi)$ contains a subquiver of one of the forms:
\begin{enumerate}[(1)]
\item 
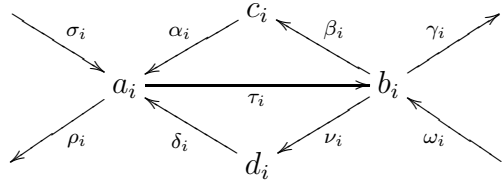
$$\xymatrix@R=0.5cm@C=1.2cm{\ar[rd]^{\sigma_i}&&c_i\ar[ld]_{\alpha_i}&& \\ 
&a_i\ar[rr]_{\tau_i}\ar[ld]^{\rho_i}&&b_i\ar[ld]^{\nu_i}\ar[ru]^{\gamma_i}\ar[lu]_{\beta_i} & \\ 
&&d_i\ar[lu]^{\delta_i}&&\ar[lu]^{\omega_i} }$$ 
if $|\mathcal{O}(\alpha_i)|\geqslant 4$ and $|\mathcal{O}(\nu_i)|\geqslant 4$, 

\item 
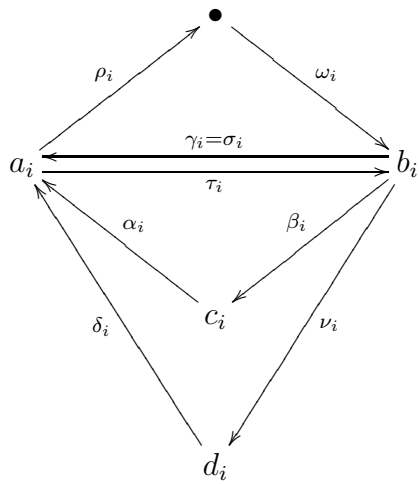
$$\xymatrix@R=1.5cm@C=2cm{&\bullet \ar[rd]^{\omega_i}& \\
a_i\ar[ru]^{\rho_i}\ar@<-0.1cm>[rr]_{\tau_i}& &  b_i\ar@<-0.1cm>[ll]_{\gamma_i=\sigma_i}\ar[ldd]^{\nu_i}\ar[ld]_{\beta_i}\\ 
&c_i\ar[lu]_{\alpha_i}& \\ 
&d_i\ar[luu]^{\delta_i}&}$$ 
if $|\mathcal{O}(\alpha_i)|=3$, and 

\item 
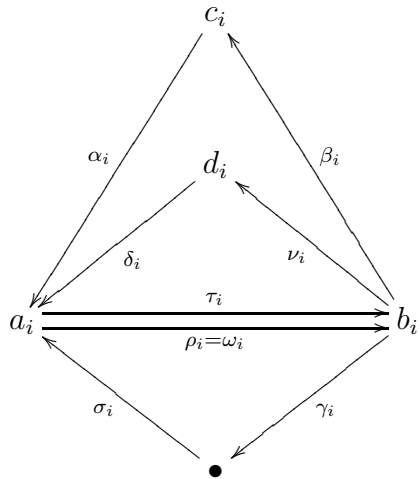
$$\xymatrix@R=1.5cm@C=2cm{&c_i\ar[ldd]_{\alpha_i} & \\ 
& d_i\ar[ld]^{\delta_i} \\ 
a_i\ar@<0.1cm>[rr]^{\tau_i}\ar@<-0.1cm>[rr]_{\rho_i=\omega_i} & & \ar[ld]^{\gamma_i} 
b_i\ar[luu]_{\beta_i}\ar[lu]^{\nu_i} \\ 
& \ar[lu]^{\sigma_i} \bullet &}$$ 
for $|\mathcal{O}(\nu_i)|=3$.
\end{enumerate}

We consider also the quiver $Q(\xi)^*$ obtained from $Q(\xi)$ by removing all vertices $c_i$, $i\in\{1,\dots,r\}$, and the arrows 
attached to them (that is the arrows $\alpha_i,\beta_i$). Note that every vertex of $Q(\xi)^*$ except $d_i$, $i\in\{1,\dots,r\}$, 
is $2$-regular. Consequently, for each arrow $\eta$ of $Q(\xi)^*$ different from $\delta_i$, $i\in\{1,\dots,r\}$, there is the 
second arrow $\tilde{\eta}$ with $s(\tilde{\eta})=s(\eta)$. Setting $\tilde{\delta_i}:=\delta_i$, for each $i\in\{1,\dots,r\}$, 
we obtain an involution $\tilde{\mbox{ }}:Q(\xi)^*_1\to Q(\xi)^*_1$. 

The quiver $Q(\xi)^*$ also admits a triangulation-like structure given by the permutation $f^*:Q(\xi)^*_1\to Q(\xi)^*_1$ 
such that $s(f^*(\eta))=t(\eta)$, for each arrow $\eta\in Q(\xi)^*_1$, and $(f^*)^3$ is the identity on the set of arrows. 
Indeed, for any arrow $\eta\in Q(\xi)^*_1$ different from $\nu_i,\delta_i,\tau_i$, $i\in\{1,\dots,r\}$, we set 
$f^*(\eta)=f(\eta)$, whereas for any $i\in\{1,\dots,r\}$, we put $f^*(\nu_i)=\delta_i$, $f^*(\delta_i)=\tau_i$ and 
$f^*(\tau_i)=\nu_i$. Composing $f^*$ with the involution $\tilde{\mbox{ }}:Q(\xi)^*_1\to Q(\xi)^*_1$, we obtain also the 
permutation $g^*=\widetilde{f^*}:Q(\xi)^*_1\to Q(\xi)^*_1$ on the set of arrows of $Q(\xi)^*$.   

For each arrow $\eta\in Q(\xi)^*_1$, we denote by $\mathcal{O}^*(\eta)$ the $g^*$-orbit of $\eta$ in $Q(\xi)^*_1$, and put 
$n^*_\eta=|\mathcal{O}^*(\eta)|$. By $\mathcal{O}(g^*)$ we denote the set of all $g^*$-orbits in $Q(\xi)^*_1$. We note that for any arrow 
$\eta\in Q_1$ with $\mathcal{O}(\eta)$ different from $\mathcal{O}(\alpha_1),\dots,\mathcal{O}(\alpha_r)$, we have 
$\mathcal{O}^*(\eta)=\mathcal{O}(\eta)$. 

The following lemma describes the orbits in $\mathcal{O}(g^*)$ without arrows in $Q_1$. 

\begin{lem}\label{lem:3.1} Let $\mathcal{O}^*$ be an orbit in $\mathcal{O}(g^*)$. The following statements are equivalent.
\begin{enumerate}[$(1)$]
\item $\mathcal{O}^*$ does not contain an arrow $\eta\in Q_1$.

\item $(Q,f)$ is of the form
$$
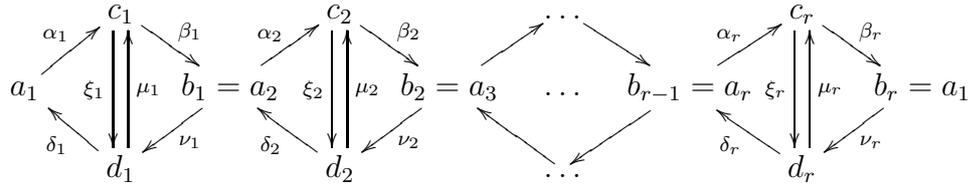
\xymatrix@R=0.5cm@C=0.25cm{&c_1\ar[rd]^{\beta_1}\ar@<-0.1cm>[dd]_{\xi_1}& &c_2\ar[rd]^{\beta_2}
\ar@<-0.1cm>[dd]_{\xi_2}& &\dots\ar[rd]& &c_r\ar[rd]^{\beta_r}\ar@<-0.1cm>[dd]_{\xi_r}& \\
\quad a_1\quad \ar[ru]^{\alpha_1}&&\mbox{ }b_1=a_2\mbox{ }\ar[ld]^{\nu_1}\ar[ru]^{\alpha_2}&&\mbox{ }b_2=a_3\mbox{ }\ar[ld]^{\nu_2} 
\ar[ru]& \dots &\mbox{ }b_{r-1}=a_r\ar[ld] \ar[ru]^{\alpha_r}&&\mbox{ }\mbox{ }b_r=a_1\ar[ld]^{\nu_r} \\
&d_1\ar[lu]^{\delta_1}\ar@<-0.1cm>[uu]_{\mu_1}& &d_2\ar[lu]^{\delta_2}\ar@<-0.1cm>[uu]_{\mu_2}&&\dots
\ar[lu]& &d_r\ar[lu]^{\delta_r}\ar@<-0.1cm>[uu]_{\mu_r}&} $$ 
with $r\geqslant 2$, the $f$-orbits $(\alpha_i\mbox{ }\xi_i\mbox{ }\delta_i)$ and $(\beta_i\mbox{ }\nu_i\mbox{ }\mu_i)$, 
$i\in\{1,\dots,r\}$, the $g$-orbits $\mathcal{O}(\alpha_1)=
(\alpha_1\mbox{ }\beta_1\mbox{ }\alpha_2\mbox{ }\beta_2\dots\alpha_r\mbox{ }\beta_r)$, $\mathcal{O}(\nu_r)=
(\nu_r\mbox{ }\delta_r\dots\nu_2\mbox{ }\delta_2\mbox{ }\nu_1\mbox{ }\delta_1)$ and $\mathcal{O}(\xi_i)=(\xi_i\mbox{ }\mu_i)=
\mathcal{O}(\mu_i)$, $i\in\{1,\dots,r\}$, $m_{\mathcal{O}(\xi_i)}=1$, for any $i\in\{1,\dots,r\}$, and $\xi=(\xi_1,\dots,\xi_r)$.

\item $Q(\xi)$ is of the form
$$
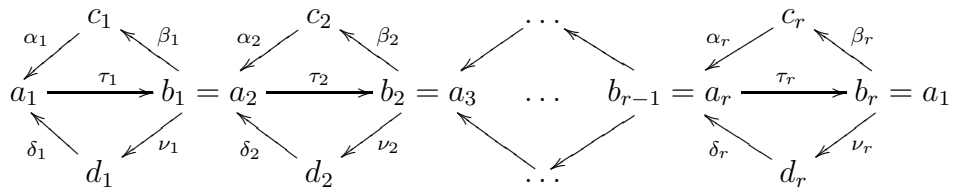
\xymatrix@R=0.5cm@C=0.4cm{&c_1\ar[ld]_{\alpha_1}& &c_2\ar[ld]_{\alpha_2}
& &\dots\ar[ld]& &c_r\ar[ld]_{\alpha_r}& \\
\quad a_1\ar[rr]^{\tau_1}&&b_1=a_2\ar[ld]^{\nu_1}\ar[lu]_{\beta_1}\ar[rr]^{\tau_2}&&b_2=a_3\ar[ld]^{\nu_2}
\ar[lu]_{\beta_2} & \dots &b_{r-1}=a_r\ar[rr]^{\tau_r}\ar[ld]\ar[lu] &&b_r=a_1\ar[ld]^{\nu_r}\ar[lu]_{\beta_r} \\
&d_1\ar[lu]^{\delta_1}& &d_2\ar[lu]^{\delta_2}&&\dots
\ar[lu]& &d_r\ar[lu]^{\delta_r}&} $$
where $r\geqslant 2$ and $\mathcal{O}^*=\mathcal{O}^*(\tau_1)=(\tau_1\mbox{ }\tau_2\dots\tau_r)$. 
\end{enumerate}
\end{lem}

\begin{proof} If $\mathcal{O}^*$ does not contain an arrow from $Q_1$, then by definition of $Q(\xi)^*$ and $g^*$, 
$\mathcal{O}^*$ consists only of arrows of type $\tau_j$, and this yields that $(Q,f)$ and $Q(\xi)$ have shapes 
required in (2) and (3), respectively. The remaining implications are obvious from the definition of $g^*$. 
\end{proof}

\begin{rem}\label{remark} The triangulation quiver $(Q,f)$ occuring in the above lemma is the triangulation quiver 
$(Q(S^2,\overrightarrow{T(r)}),f)$ of the following triangulation $T(r)$ of the unit sphere $S^2$ in $\mathbb{R}^3$ 
$$
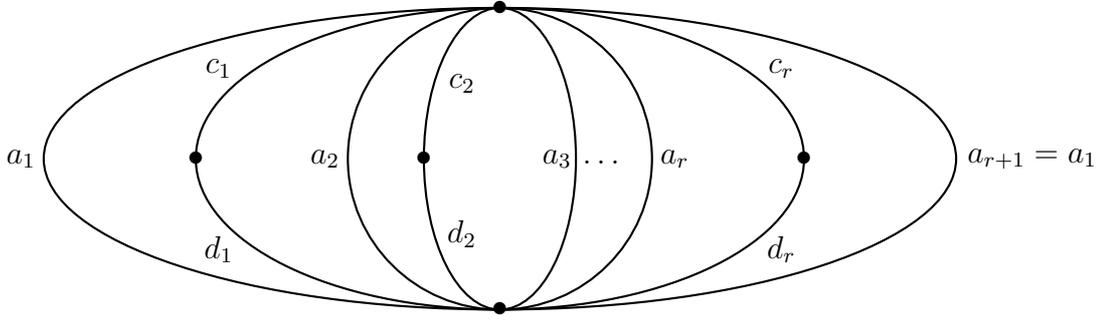
\begin{tikzpicture}
\draw[thick](0,0) circle [x radius = 6, y radius = 2];
\draw[thick](0,0) circle [x radius = 4, y radius = 2];
\draw[thick](0,0) circle [x radius = 2, y radius = 2];
\draw[thick](0,0) circle [x radius = 1, y radius = 2];
\node() at (-4,0){$\bullet$};
\node() at (-1,0){$\bullet$};
\node() at (4,0){$\bullet$};
\node() at (0,2){$\bullet$};
\node() at (0,-2){$\bullet$};
\node() at (-6.3,0){$a_1$};
\node() at (-2.3,0){$a_2$};
\node() at (1.1,0){$a_3$ $\dots$};
\node() at (2.3,0){$a_r$};
\node() at (7,0){$a_{r+1}=a_1$};
\node() at (-3.7,1.2){$c_1$};
\node() at (-3.7,-1.2){$d_1$};
\node() at (-0.5,1){$c_2$};
\node() at (-0.5,-1){$d_2$};
\node() at (3.7,1.2){$c_r$};
\node() at (3.7,-1.2){$d_r$};

\end{tikzpicture}$$ 
with $r\geqslant 2$ and coherent orientation $\overrightarrow{T(r)}$ of triangles in $T(r)$: $(a_i\mbox{ }c_i\mbox{ }d_i)$ 
and $(c_i\mbox{ }a_{i+1}\mbox{ }d_i)$, $i\in\{1,\dots,r\}$ (see \cite[Example 7.5]{ES8}).
\end{rem}

We also have the following lemma.

\begin{lem}\label{lem:3.3} Let $\eta$ be an arrow in $Q(\xi)^*_1$. Then the following equivalences hold.
\begin{enumerate}[$(1)$]
\item $|\mathcal{O}^*(\eta)|=1$ if and only if $\eta\in Q_1$ and $|\mathcal{O}(\eta)|=1$. In this case, we also have 
$\mathcal{O}^*(\eta)=\mathcal{O}(\eta)$.

\item $|\mathcal{O}^*(\eta)|=2$ if and only if one of the cases holds: 
\begin{enumerate}[$(a)$]
\item $\eta\in Q_1$ with $|\mathcal{O}(\eta)|=2$;  

\item $\mathcal{O}^*(\eta)=\mathcal{O}^*(\tau_i)$ for some $i\in\{1,\dots,r\}$ with $|\mathcal{O}(\alpha_i)|=3$;

\item $\mathcal{O}^*(\eta)=(\tau_1\tau_2)$ for the case $r=2$ described in the previous lemma.\end{enumerate}
\end{enumerate}
\end{lem} 

\begin{proof} This is straightforward from the definition of $Q(\xi)^*$. \end{proof}

We define now two functions 
$$m^*_\bullet:\mathcal{O}(g^*)\to\mathbb{N}^*\mbox{ and }c^*_\bullet:\mathcal{O}(g^*)\to K^*$$ 
which assign to each orbit $\mathcal{O}^*$ in $\mathcal{O}(g^*)$ the elements 
$$m^*_{\mathcal{O}^*}=\left\{\begin{array}{cc}
m_{\mathcal{O}(\eta)}, & \mbox{if }\mathcal{O}^*=\mathcal{O}^*(\eta)\mbox{ for some arrow }\eta\in Q_1,\\
m_{\mathcal{O}(\alpha_1)}, & \mbox{otherwise (see Lemma \ref{lem:3.1})},
\end{array}\right.$$
$$c^*_{\mathcal{O}^*}=\left\{\begin{array}{cc}
c_{\mathcal{O}(\eta)}, & \mbox{if }\mathcal{O}^*=\mathcal{O}^*(\eta)\mbox{ for some arrow }\eta\in Q_1,\\
c_{\mathcal{O}(\alpha_1)}, &  \mbox{otherwise (see Lemma \ref{lem:3.1}).}
\end{array}\right.$$
Note that the two conditions exclude in both definitions. We abbreviate $m^*_\eta=m^*_{\mathcal{O}^*(\eta)}$ and 
$c^*_\eta=c^*_{\mathcal{O}^*(\eta)}$, for any arrow $\eta\in Q(\xi)^*_1$. Using assumptions imposed on $m_\bullet$ 
one can deduce from Lemma \ref{lem:3.1} that $m^*_\eta n^*_\eta \geqslant 2$, for any arrow $\eta\in Q(\xi)^*_1$. 
The description of all arrows $\eta\in Q(\xi)^*_1$ with $m^*_\eta n^*_\eta =2$ follows from Lemma \ref{lem:3.3}, and 
we call these arrows $\eta$ \emph{virtual arrows}. 

For each $\eta\in Q(\xi)^*_1$, we define the path $A^*_\eta$ as follows: 
\begin{itemize}
\item $A^*_\eta=(\eta g^*(\eta)\dots (g^*)^{n^*_\eta-1}(\eta))^{m^*_\eta-1}\eta 
g^*(\eta)\dots (g^*)^{n^*_\eta-2}(\eta)$, if $n^*_\eta\geqslant 2$,

\item $A^*_\eta=A_\eta=\eta^{m_\eta-1}$, if $n^*_\eta=1$ (equivalently, $\eta\in Q_1$ with $n_\eta=1$). 
\end{itemize}
Moreover, we set 
$$B^*_\eta=A^*_\eta(g^*)^{n^*_\eta-1}(\eta)=(\eta g^*(\eta)\dots (g^*)^{n^*_\eta-1}(\eta))^{m^*_\eta}.$$

For each $i\in\{1,\dots,r\}$, we denote by $C^*_{\delta_i}$ the subpath of $A^*_{\delta_i}$ such that 
$A^*_{\delta_i}=\delta_iC^*_{\delta_i}$, and let 
\begin{itemize}
\item $A^*_{\alpha_i}=\alpha_i C^*_{\delta_i}$, and 
\item $B^*_{\alpha_i}=A^*_{\alpha_i}\beta_i=\alpha_i C^*_{\delta_i}\beta_i$.
\end{itemize}
We note that $C^*_{\delta_i}$ is of length $\geqslant 1$, since $|\mathcal{O}(\delta_i)|\geqslant 3$. 

The definition of a virtual mutation of a weighted triangulation algebra is now as follows. 

\begin{df}\label{def:3.4} The algebra $\Lambda(\xi)=\Lambda(Q,f,m_\bullet,c_\bullet,\xi):=KQ(\xi)/I(\xi)$ is called a \emph{virtual 
mutation of a weighted triangulation algebra} $\Lambda=\Lambda(Q,f,m_\bullet,c_\bullet)$ with respect to a sequence 
$\xi=(\xi_1,\dots,\xi_r)$ of virtual arrows if $I(\xi)$ is the ideal of the path algebra $KQ(\xi)$ generated by the elements: 
\begin{enumerate}[(1)]
\item $\nu_i\delta_i-\beta_i\alpha_i-c_{\tilde{\nu_i}}^* A^*_{\tilde{\nu_i}}$, $\alpha_i\tau_i$ and $\tau_i\beta_i$, for all 
$i\in\{1,\dots,r\}$, 

\item $\eta f^*(\eta)-c^*_{\tilde{\eta}}A^*_{\tilde{\eta}}$, for all arrows $\eta\in Q(\xi)^*_1$ different from $\nu_i$, $i\in
\{1,\dots,r\}$,

\item $\eta f^*(\eta)g^*(f^*(\eta))$, for all arrows $\eta=\delta_i$, $i\in\{1,\dots,r\}$, and all arrows $\eta\in Q(\xi)^*_1\cap Q_1$ different 
from $\nu_i,\delta_i$, for $i\in\{1,\dots,r\}$, and unless $f^2(\eta)$ is virtual or unless $f(\bar{\eta})$ is virtual with $m_{\bar{\eta}}=1$ and 
$n_{\bar{\eta}}=3$, 

\item $\eta g^*(\eta)f^*(g^*(\eta))$ for all arrows $\eta\in Q(\xi)^*_1$ different from $\tau_i$ with $m_{\nu_i}=1$ and $n_{\nu_i}=3$, $\nu_i$, 
$(g^*)^{-1}(\nu_i)$, for $i\in\{1,\dots,r\}$, and $\eta\in Q(\xi)^*_1\cap Q_1$ such that $f(\eta)$ is virtual or $f^2(\eta)$ is virtual with $m_{f(\eta)}=1$ and 
$n_{f(\eta)}=3$. 

\end{enumerate}
\end{df}

We will identify an arrow $\eta$ of $Q(\xi)$ with the corresponding element of $\Lambda(\xi)=KQ(\xi)/I(\xi)$. 

\begin{df}\label{def:3.5} Let $(S,\overrightarrow{T})$ be a directed triangulated surface, 
$(Q(S,\overrightarrow{T}),f)$ the associated triangulation quiver, $m_\bullet$ and $c_\bullet$ be weight 
and parameter functions of $(Q(S,\overrightarrow{T}),f)$, and $\xi=(\xi_1,\dots,\xi_r)$ a sequence of virtual arrows 
from pairwise different orbits in $\mathcal{O}(g)$ of length $2$ (and trivial weights). Then the virtual mutation 
$\Lambda(Q(S,\overrightarrow{T}),f,m_\bullet,c_\bullet,\xi)$ is said to be a \emph{virtual mutation of the weighted surface algebra} 
$\Lambda(S,\overrightarrow{T},m_\bullet,c_\bullet)$, and denoted by $\Lambda(S,\overrightarrow{T},m_\bullet,c_\bullet,\xi)$. 
\end{df}

We shall present now some consequences of the relations defining a virtual mutation $\Lambda(\xi)$ of a weighted 
triangulation algebra $\Lambda=\Lambda(Q,f,m_\bullet,c_\bullet)$. 

\begin{lem}\label{lem:3.6} Assume that $i\in\{1,\dots,r\}$ and $m_{\alpha_i}=1$ and $n_{\alpha_i}=
|\mathcal{O}(\alpha_i)|=3$. Then the following relations hold in $\Lambda(\xi)$.
\begin{enumerate}[$(1)$] 
\item $\nu_i\delta_i=\beta_i\alpha_i+c_{\gamma_i}\gamma_i$ and $\rho_i\omega_i=c_{\gamma_i}\tau_i$.

\item $\omega_i\nu_i\delta_i=\omega_i\beta_i\alpha_i+c_{\gamma_i}c_{\rho_i}A^*_{g(\rho_i)}$.

\item $\nu_i\delta_i\rho_i=\beta_i\alpha_i\rho_i+c_{\gamma_i}c_{\nu_i}A^*_{\nu_i}$.

\item $\rho_i\omega_i\nu_i=c_{\gamma_i}c_{\rho_i}A^*_{\rho_i}$ and $\rho_i\omega_i\nu_i\delta_i=c_{\rho_i}B^*_{\rho_i}$.

\item $\delta_i\rho_i\omega_i=c_{\gamma_i}c_{\delta_i}A^*_{\delta_i}$ and $\delta_i\rho_i\omega_i\nu_i=c_{\delta_i}B^*_{\delta_i}$.

\item $\alpha_i\rho_i\omega_i=0$ and $\rho_i\omega_i\beta_i=0$. 
\end{enumerate} 
In particular, the arrows $\gamma_i$ and $\tau_i$ do not occur in the Gabriel quiver of $\Lambda(\xi)$.
\end{lem}

\begin{proof} For (1), we note that $\tilde{\nu_i}=\gamma_i$, $\tilde{\rho_i}=\tau_i$ and $A^*_{\gamma_i}=\gamma_i$ and 
$A^*_{\tau_i}=\tau_i$, because $\omega_i=f(\rho_i)=f^*(\rho_i)$ and $\mathcal{O}^*=\mathcal{O}^*(\gamma_i)=
\mathcal{O}^*(\tau_i)$ is an orbit of length $2$ with weight $m^*_{\mathcal{O}^*}=m_{\mathcal{O}(\alpha_i)}=1$. 
Then required relations follow from (1) and (2) in Definition \ref{def:3.4}.

For (2) and (3), observe that $\omega_i\gamma_i=\omega_i f^*(\omega_i)=c^*_{\tilde{\omega_i}}A^*_{\tilde{\omega_i}}$ and 
$\gamma_i\rho_i=\gamma_i f^*(\gamma_i)=c_{\tilde{\gamma_i}}A^*_{\tilde{\gamma_i}}$, by relations (2) defining 
$\Lambda(\xi)$, so the required relations follow, since $\tilde{\gamma}_i=\nu_i$ and $\tilde{\omega}_i=g(\rho_i)$. 

The equalities in (4) and (5) follow from (1) and the relations $\tau_i\nu_i=c^*_{\tilde{\tau_i}}A^*_{\tilde{\tau_i}}$ and 
$\delta_i\tau_i=c^*_{\tilde{\delta_i}}A^*_{\tilde{\delta_i}}$, because $\tilde{\tau_i}=\rho_i$ and 
$\tilde{\delta_i}=\delta_i$. 

Finally, (6) is a consequence of (1) and the relations $\alpha_i\tau_i=0$ and $\tau_i\beta_i=0$. 
\end{proof} 

\begin{lem}\label{lem:3.7} We have $A^*_{\alpha_i}\nu_i=0$ and $A^*_{\delta_i}\beta_i=0$, for all $i\in\{1,\dots,r\}$ for all $i\in\{1,\dots,r\}$. \end{lem} 

\begin{proof} Fix $i\in\{1,\dots,r\}$. We have the equalities 
$$c_{\delta_i}A^*_{\alpha_i}\nu_i=\alpha_i(c_{\delta_i}C^*_{\delta_i}\nu_i)=\alpha_i(c^*_{g^*(\delta_i)A^*_{g^*(\delta_i)}})=\alpha_i\tau_i\nu_i=0,$$ 
$$c_{\delta_i}A^*_{\delta_i}\beta_i=c^*_{\tilde{\delta_i}}A^*_{\tilde{\delta_i}}=\delta_i\tau_i\beta_i=0,$$ 
because $\alpha_i\tau_i=0$ and $\tau_i\beta_i=0$. Since $c_{\delta_i}\in K^*$, we conclude that the required relations hold. \end{proof}

\begin{prop}\label{prop:3.7} Let $\eta$ be an arrow in $Q(\xi)^*_1$. Then the following statements hold.
\begin{enumerate}[$(1)$]
\item $B^*_\eta$ is a non-zero element of the right (respectively, left) socle of $\Lambda(\xi)$.

\item $c^*_\eta B^*_\eta=\eta f^*(\eta)(f^*)^2(\eta)=\tilde{\eta}f^*(\tilde{\eta})(f^*)^2(\tilde{\eta})=c^*_{\tilde{\eta}}B^*_{\tilde{\eta}}$.
\end{enumerate}
\end{prop}

\begin{proof}
(1) For $\eta\in Q_1$, the cycle $B^*_\eta$ is obtained from the cycle $B_\eta$ of $Q$ by replacing all subpaths $\alpha_j\beta_j$, 
$j\in\{1,\dots,r\}$, by arrows $\tau_j$. Similarily, if $\eta=\tau_i$ for some $i\in\{1,\dots,r\}$, then $B^*_\eta$ is obtained 
from the cycle $B_{\alpha_i}$ of $Q$ in the same way. We know from Proposition \ref{prop:2.6}(3) that every cycle $B_{\alpha_i}$ 
in $Q$, $\alpha\in Q_1$, is a non-zero element of $\soc\Lambda=\soc(\Lambda_\Lambda)=\soc({}_\Lambda\Lambda)$, and hence is a 
non-zero element with $\gamma B_\alpha=0$ and $B_\alpha\sigma=0$ for all arrows $\gamma,\sigma\in Q_1$. Hence it follows from 
the relations defining $\Lambda(\xi)$ that $B^*_\eta$ is a non-zero element of $\Lambda(\xi)$ satisfying $\phi B^*_\eta=0$ and 
$B^*_\eta \psi=0$ for all arrows $\phi,\psi\in Q(\xi)^*_1$. Since the only arrows in $Q(\xi)_1$ which do not belong to $Q(\xi)^*_1$ 
are the arrows $\alpha_i,\beta_i$, for $i\in\{1,\dots,r\}$, it is now sufficient to observe that the equalities  
$$\alpha_i B^*_{g^*(\delta_i)}=A^*_{\alpha_i}\nu_i\delta_i=0 \mbox{ and } B^*_{\nu_i}\beta_i=\nu_i A^*_{\delta_i}\beta_i=0,$$ 
are consequences of Lemma \ref{lem:3.7}. Summing up, we conclude that $B^*_\eta$ is a non-zero element of 
the right (or left) socle of $\Lambda(\xi)$, for any arrow $\eta\in Q(\xi)^*_1$.\smallskip

(2) It follows from Proposition \ref{prop:2.6} that for any arrow $\alpha\in Q_1$ we have the equalities in $\Lambda$ 
$$\alpha f(\alpha) f^2(\alpha)=c_\alpha B_\alpha=c_{\bar{\alpha}} B_{\bar{\alpha}}=\bar{\alpha} f(\bar{\alpha}) f^2(\bar{\alpha}).$$ 
Hence, if $\eta$ is an arrow of $Q_1$ different from $\nu_i$, $\delta_i$, $i\in\{1,\dots,r\}$, then the following equalities hold 
in $\Lambda(\xi)$
$$\eta f^*(\eta)(f^*)^2(\eta)=c^*_\eta B^*_\eta=c^*_{\tilde{\eta}}B^*_{\tilde{\eta}}.$$ 
We note that in this case $\tilde{\eta}=\bar{\eta}$. 

Assume now that $\eta\in\{\tau_i,\nu_i,\delta_i\}$, for some $i\in\{1,\dots,r\}$. Then, by definition of $I(\xi)$, we have 
the equalities in $\Lambda(\xi)$ 
$$\tau_i f^*(\tau_i) (f^*)^2(\tau_i)=\tau_i\nu_i\delta_i=c^*_{\tilde{\tau_i}}A^*_{\tilde{\tau_i}}\delta_i=c^*_{\tilde{\tau_i}}
B^*_{\tilde{\tau_i}},$$
$$\tau_i\nu_i\delta_i=\tau_i\beta_i\alpha_i+c^*_{\tilde{\nu_i}}\tau_i A^*_{\tilde{\nu_i}}=c^*_{\tau_i}B^*_{\tau_i},$$
$$\nu_i f^*(\nu_i) (f^*)^2(\nu_i)=\nu_i\delta_i\tau_i=c_{\tilde{\delta_i}}\nu_i A^*_{\tilde{\delta_i}}=c^*_{\nu_i}B^*_{\nu_i},$$
$$\nu_i\delta_i\tau_i=\beta_i\alpha_i\tau_i+c^*_{\tilde{\nu_i}}A^*_{\tilde{\nu_i}}\tau_i =c^*_{\tilde{\nu_i}}B^*_{\tilde{\nu_i}},$$
$$\delta_i f^*(\delta_i)(f^*)^2(\delta_i)=\delta_i\tau_i\nu_i=c^*_{\tilde{\delta_i}}A^*_{\tilde{\delta_i}}\nu_i=
c^*_{\tilde{\delta_i}}B^*_{\tilde{\delta_i}},$$
$$\delta_i\tau_i\nu_i=c^*_{\tilde{\tau_i}}\delta_i A^*_{\tilde{\tau_i}}=c^*_{g^*(\delta_i)}B^*_{\delta_i}=c^*_{\delta_i}
B^*_{\delta_i}.$$

This proves the statement (2).\end{proof}

\begin{lem}\label{lem:3.8} If $i\in\{1,\dots,r\}$, then $B^*_{\alpha_i}$ is a non-zero element of the right (respectively, left) 
socle of $\Lambda(\xi)$.
\end{lem} 

\begin{proof}
The fact that $B^*_{\alpha_i}$ is a non-zero element of $\Lambda(\xi)$ follows from the defining relations. Moreover, $\alpha_i$ 
is the unique arrow in $Q(\xi)$ with source $c_i$ and $\beta_i$ is the unique arrow in $Q(\xi)$ with target $c_i$. It follows 
from Proposition \ref{prop:3.7}(1) that 
$$B^*_{\alpha_i}\alpha_i=\alpha_i C^*_{\delta_i}\beta_i\alpha_i=\alpha_i B^*_{g^*(\delta_i)}=0,$$
because $g^*(\delta_i)\in Q(\xi)^*_1$. Moreover, we have 
$$\beta_i B^*_{\alpha_i}=\beta_i\alpha_i C^*_{\delta_i}\beta_i=\nu_i\delta_i C^*_{\delta_i}\beta_i-c^*_{\tilde{\nu_i}} 
A^*_{\tilde{\nu_i}}C^*_{\delta_i}\beta_i.$$ 
Observe also that $\nu_i\delta_iC^*_{\delta_i}\beta_i=\nu_i A^*_{\delta_i}\beta_i=B^*_{\nu_i}\beta_i=0$, again by (1) in 
Proposition \ref{prop:3.7}, since $\nu_i\in Q(\xi)^*_1$. Now we claim that also $A^*_{\tilde{\nu_i}}C^*_{\delta_i}\beta_i=0$, and 
consequently $\beta_i B^*_{\alpha_i}=0$. By definition of $g^*$ we have $(g^*)^{-1}(\tilde{\nu_i})=\tau_i$, and there are three 
cases to consider.
\begin{enumerate}[(a)]
\item Assume that $(g^*)^{-1}(\tau_i)=\tau_j$ for some $j\in\{1,\dots,r\}$. Then $A^*_{\tilde{\nu_i}}=\tilde{\nu_i}\dots \tau_j$ and 
$C^*_{\delta_i}=\nu_j\delta_j\dots (g^*)^{-1}(\nu_i)$. Moreover, it follows from Proposition \ref{prop:3.7} that the element 
$\tau_j\nu_j\delta_j=\tau_j f^*(\tau_j)(f^*)^2(\tau_j)=c^*_{\tau_j}B^*_{\tau_j}$ belongs to the right socle of $\Lambda(\xi)$, and 
so $A^*_{\tilde{\nu_i}}C^*_{\delta_i}\beta_i=0$.

\item Assume that $(g^*)^{-1}(\tau_i)=\sigma_i$ is an arrow in $Q_1$ with $\bar{\sigma_i}$ virtual but not in one of the orbits 
$\mathcal{O}_1,\dots,\mathcal{O}_r$. Then $Q(\xi)$ contains a subquiver of the form 
$$
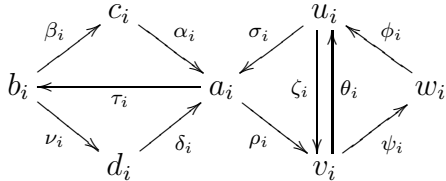
\xymatrix@R=0.5cm@C=0.8cm{ &c_i\ar[rd]^{\alpha_i}&&u_i\ar[ld]_{\sigma_i}\ar@<-0.1cm>[dd]_{\zeta_i}& \\ 
b_i\ar[ru]^{\beta_i}\ar[rd]_{\nu_i} &&a_i\ar[ll]^{\tau_i}\ar[rd]_{\rho_i}&& w_i\ar[lu]_{\phi_i} \\ 
&d_i\ar[ru]_{\delta_i}&&v_i\ar[ru]_{\psi_i}\ar@<-0.1cm>[uu]_{\theta_i}&}$$
where $\bar{\sigma_i}=\zeta_i$ and $\bar{\psi_i}=\theta_i$ are virtual arrows of $Q_1$, and $g^*(\phi_i)=\sigma_i$, $g^*(\sigma_i) 
=\tau_i$, $g^*(\delta_i)=\rho_i$ and $g^*(\rho_i)=\psi_i$. Hence we have $A^*_{\tilde{\nu_i}}=\tilde{\nu_i}\dots\phi_i\sigma_i$ and 
$C^*_{\delta_i}=\rho_i\psi_i\dots(g^*)^{-1}(\nu_i)$. Further, the following equalities hold 
$$\phi_i\sigma_i\rho_i\psi_i=c_{\zeta_i}\phi_i\zeta_i\psi_i=c_{\zeta_i}\phi_if(\phi_i)f^2(\phi_i)=c_{\zeta_i}c_{\phi_i}B^*_{\phi_i}.$$ 
Since $B^*_{\phi_i}$ is in the right socle of $\Lambda(\xi)$, we conclude that $A^*_{\tilde{\nu_i}}C^*_{\delta_i}\beta_i=0$. 

\item Finally, let $(g^*)^{-1}(\tau_i)=\sigma_i$ be an arrow in $Q_1$ for which $\bar{\sigma_i}$ is not virtual. Let also $\phi_i:= 
(g^*)^{-1}(\sigma_i)$ and $\rho_i:=g^*(\delta_i)$. We note that $\rho_i=f(\sigma_i)$ and $f(\phi_i)=\bar{\sigma_i}$ is not virtual 
in $(Q,f)$, and hence $\phi_i\sigma_i\rho_i=\phi_i g(\phi_i)f(g(\phi_i))=0$. Finally, we have $A^*_{\tilde{\nu_i}}=\tilde{\nu_i} 
\dots\phi_i\sigma_i$, while $C^*_{\delta_i}=\rho_i\dots (g^*)^{-1}(\nu_i)$, and therefore $A^*_{\tilde{\nu_i}}C^*_{\delta_i}
\beta_i=0$.
\end{enumerate}
\end{proof}

We shall now describe a basis of the virtually mutated weighted triangulation algebra $\Lambda(\xi)$.

Let $\eta$ be an arrow in $Q(\xi)_1\setminus\{\beta_1,\dots,\beta_r\}$. Then we denote by $\mathcal{B}_\eta$ the set of all proper 
initial submonomials of $B^*_\eta$, and by $\mathcal{B}^c_\eta$ the set of all paths of the form $u\beta_i$, for all paths $u$ such 
that $u\nu_i$ is a subpath of $B^*_\eta$ for some $i\in\{1,\dots,r\}$. 

\begin{prop}\label{prop:3.9} Let $x$ be a vertex in $Q(\xi)_0$. The following statements hold. 
\begin{enumerate}[$(1)$]
\item Assume $x$ is the starting vertex of two arrows $\eta$ and $\tilde{\eta}$ in $Q(\xi)^*_1$, which are not virtual. Then the 
module $e_x\Lambda(\xi)$ has basis of the form 
$$\mathcal{B}_x=\mathcal{B}_\eta\cup\mathcal{B}^c_\eta\cup\mathcal{B}_{\tilde{\eta}}\cup\mathcal{B}^c_{\tilde{\eta}}\cup 
\{e_x,B^*_\eta\}.$$

\item Assume $x$ is the starting vertex of two arrows $\eta$ and $\tilde{\eta}$ in $Q(\xi)^*_1$, and $\tilde{\eta}$ is virtual. 
Then the module $e_x\Lambda(\xi)$ has basis 
$$\mathcal{B}_x=\mathcal{B}_\eta\cup\mathcal{B}^c_\eta\cup\{e_x,B^*_\eta,\eta f(\eta)\}.$$ 

\item Assume that $x=c_i$ for some $i\in\{1,\dots,r\}$. Then the module $e_x\Lambda(\xi)$ has basis 
$$\mathcal{B}_x=\mathcal{B}_{\alpha_i}\cup\mathcal{B}^c_{\alpha_i}\cup\{e_x,B^*_{\alpha_i}\}.$$

\item Assume that $x=d_i$ for some $i\in\{1,\dots,r\}$. Then the module $e_x\Lambda(\xi)$ has basis  
$$\mathcal{B}_{\delta_i}\cup\mathcal{B}^c_{\delta_i}\cup\{e_x,B^*_{\delta_i}\}.$$
\end{enumerate}
\end{prop} 

\begin{proof} It follows from the relations defining $\Lambda(\xi)$, Proposition \ref{prop:3.7} and Lemmas \ref{lem:3.6} and 
\ref{lem:3.8}. \end{proof} 

The next aim is to determine the dimension of $\Lambda(\xi)$. 

For each arrow $\eta\in Q(\xi)^*_1$, we denote 
$$n^\nu_\eta=|\{i\in\{1,\dots,r\}:\, \mathcal{O}^*(\nu_i)=\mathcal{O}^*(\eta)\}|.$$ 

\begin{lem}\label{lem:3.10} For each $i\in\{1,\dots,r\}$, we have $|\mathcal{B}_{c_i}|=|\mathcal{B}_{d_i}|=m^*_{\delta_i} 
(n^*_{\delta_i}+n^\nu_{\delta_i})$. \end{lem} 

\begin{proof} Clearly $|\mathcal{B}_{c_i}|=|\mathcal{B}_{d_i}|$. Moreover, $m^*_{\delta_i}=m_{\delta_i}$ and 
$|\mathcal{B}_{\delta_i}|=m_{\delta_i}n^*_{\delta_i}-1$, and $|\mathcal{B}^c_{\delta_i}|=m_{\delta_i}n^\nu_{\delta_i}-1$, 
hence the required equality follows.
\end{proof}

\begin{lem}\label{lem:3.11} Let $\eta$ be an arrow in $Q(\xi)^*_1$ such that $\tilde{\eta}$ is virtual, and let $x=s(\eta)$. Then 
we have $$|\mathcal{B}_x|=m^*_\eta(n^*_\eta+n^\nu_\eta)+2.$$ \end{lem} 

\begin{proof} We note that $m^*_\eta=m_\eta$, $|\mathcal{B}_\eta|=m_\eta n^*_\eta -1$ and $|\mathcal{B}^c_\eta|=m_\eta n^\nu_\eta$, 
so the required equality holds by Proposition \ref{prop:3.9}(2). \end{proof} 

\begin{lem}\label{lem:3.12} Let $x$ be the starting vertex of two arrows $\eta$ and $\tilde{\eta}$ in $Q(\xi)^*_1$, which are not 
virtual. Then $$|\mathcal{B}_x|=m^*_\eta(n^*_\eta+n^\nu_\eta)+m^*_{\tilde{\eta}}(n^*_{\tilde{\eta}}+n^\nu_{\tilde{\eta}}).$$
\end{lem} 

\begin{proof} We have $|\mathcal{B}_\eta|=m^*_{\eta}n^*_{\eta}-1$, $|\mathcal{B}^c_\eta|=m^*_{\eta}n^\nu_{\eta}$, 
$|\mathcal{B}_{\tilde{\eta}}|=m^*_{\tilde{\eta}}n^*_{\tilde{\eta}}-1$ and $|\mathcal{B}^c_{\tilde{\eta}}|=m^*_{\tilde{\eta}} 
n^\nu_{\tilde{\eta}}$, and the required equality follows from Proposition \ref{prop:3.9}. \end{proof} 

\begin{rem}\label{rem:3.13} We note that if  $\varepsilon$ is a virtual arrow of $Q(\xi)^*_1$, then $m^*_\varepsilon(n^*_\varepsilon+ 
n^\nu_\varepsilon)=m_\varepsilon^* n_\varepsilon^* =2$. \end{rem} 

Summing up, we obtain the following theorem. 

\begin{theorem}\label{thm:3.14} $\Lambda(\xi)$ is a finite-dimensional algebra with 
$$\dim_K\Lambda(\xi)=\sum_{\eta\in Q(\xi)^*_1} m^*_\eta(n^*_\eta+n^\nu_\eta)+\sum_{i=1}^r m^*_{\delta_i}(n^*_{\delta_i}+n^\nu_{\delta_i}).$$
\end{theorem}

Moreover, we have the following proposition.

\begin{prop}\label{prop:3.15} $\Lambda(\xi)$ is not isomorphic to a weighted triangulation algebra. \end{prop} 

\begin{proof} This is a consequence of the shape of the quiver $Q(\xi)$ and relations of type (1) in Definition \ref{def:3.4} 
(see also Lemma \ref{lem:3.6}). \end{proof} 

\section{Examples}\label{sec:4} 

In this section we present several examples of virtually mutated weighted triangulation (surface) algebras. 

\begin{exmp}\label{ex:4.1} Let $D(m,\lambda)=\Lambda(Q,f,m_\bullet,c_\bullet)$ be the disc algebra of degree $m$ described 
in Example \ref{ex:2.7}. Hence $(Q,f)=(Q(D,\overrightarrow{T}),f)$ and we have two $g$-orbits $\mathcal{O}(\alpha)= 
(\alpha\mbox{ }\beta\mbox{ }\gamma\mbox{ }\nu\mbox{ }\delta\mbox{ }\rho)=\mathcal{O}(\nu)$ and $\mathcal{O}(\xi)=(\xi\mbox{ }\mu)$, 
where $\xi$ and $\mu$ are unique virtual arrows. We take $\xi$ as the chosen element of the orbit $\mathcal{O}(\xi)$. Then 
the virtually mutated algebra 
$$D(m,\lambda,\xi)=\Lambda(\xi)=\Lambda(Q,f,m_\bullet,c_\bullet,\xi)$$ 
is the algebra given by the quiver $Q(\xi)$ of the form 
$$\qquad\qquad\qquad\qquad
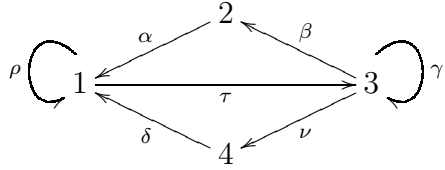
\xymatrix@R=0.01cm@C=1.5cm{
&&2\ar[lddd]_{\alpha}&& \\
&&&& \\
&\ar@(ul,dl)[d]_{\rho}&&\ar@(ur,dr)[d]^{\gamma} & \\ 
&1\ar[rr]_{\tau}&&3\ar[lddd]^{\nu}\ar[luuu]_{\beta}&
&&&& \\
&&&& \\ 
&&&& \\
&&4\ar[luuu]^{\delta}&&}$$
and the relations: 
$$\nu\delta=\beta\alpha+\lambda(\gamma\nu\delta\rho\tau)^{m-1}\gamma\nu\delta\rho,\alpha\tau=0,\tau\beta=0,$$
$$\tau\nu=\lambda(\rho\tau\gamma\nu\delta)^{m-1}\rho\tau\gamma\nu,\delta\tau=
\lambda(\delta\rho\tau\gamma\nu)^{m-1}\delta\rho\tau\gamma,$$
$$\rho^2=\lambda(\tau\gamma\nu\delta\rho)^{m-1}\tau\gamma\nu\delta,\gamma^2=
\lambda(\nu\delta\rho\tau\gamma)^{m-1}\nu\delta\rho\tau,$$
$$\rho^2\tau=0,\gamma^2\nu=0,\delta\tau\gamma=0,$$
$$\delta\rho^2=0,\tau\gamma^2=0,\rho\tau\nu=0.$$ 
We note the following consequences of those relations: 
$$\rho^3=\lambda(\tau\gamma\nu\delta\rho)^m=\lambda(\rho\tau\gamma\nu\delta)^m=\tau\nu\delta,$$
$$\gamma^3=\lambda(\gamma\nu\delta\rho\tau)^m=\lambda(\nu\delta\rho\tau\gamma)^m=\nu\delta\tau,$$
$$\beta\alpha\rho=\nu\delta\rho,\gamma\beta\alpha=\gamma\nu\delta.$$
We observe that, according to Proposition \ref{prop:3.9}, $D(m,\lambda,\xi)$ has the basis $\mathcal{B}^\xi=\mathcal{B}^\xi_1\cup 
\mathcal{B}^\xi_2\cup \mathcal{B}^\xi_3 \cup \mathcal{B}^\xi_4$ given by the sets 
$$\mathcal{B}^\xi_1=\left\{(\rho\tau\gamma\nu\delta)^r\rho,(\rho\tau\gamma\nu\delta)^r\rho\tau,
(\rho\tau\gamma\nu\delta)^r\rho\tau\gamma,(\rho\tau\gamma\nu\delta)^r\rho\tau\gamma\beta,
(\rho\tau\gamma\nu\delta)^r\rho\tau\gamma\nu\right\}_{0\leqslant r\leqslant m-1}$$ 
$$\cup\left\{(\tau\gamma\nu\delta\rho)^r\tau,(\tau\gamma\nu\delta\rho)^r\tau\gamma,
(\tau\gamma\nu\delta\rho)^r\tau\gamma\beta,(\tau\gamma\nu\delta\rho)^r\tau\gamma\nu,
(\tau\gamma\nu\delta\rho)^r\tau\gamma\nu\delta\right\}_{0\leqslant r\leqslant m-1}$$
$$\cup\left\{e_1,(\rho\tau\gamma\nu\delta)^s,(\tau\gamma\nu\delta\rho)^t\right\}_{1\leqslant s\leqslant m,1\leqslant t\leqslant m-1},$$

$$\mathcal{B}^\xi_2=\left\{\alpha(\rho\tau\gamma\nu\delta)^r,\alpha(\rho\tau\gamma\nu\delta)^r\rho,
\alpha(\rho\tau\gamma\nu\delta)^r\rho\tau,\alpha(\rho\tau\gamma\nu\delta)^r\rho\tau\gamma\right
\}_{1\leqslant r\leqslant m-1}$$
$$\cup\left\{e_2,\alpha(\rho\tau\gamma\nu\delta)^s\rho\tau\gamma\beta, 
\alpha(\rho\tau\gamma\nu\delta)^{t-1}\rho\tau\gamma\nu
\right\}_{0\leqslant s\leqslant m-1,1\leqslant t\leqslant m-1},$$

$$\mathcal{B}^\xi_3=\left\{(\gamma\nu\delta\rho\tau)^r\gamma,(\gamma\nu\delta\rho\tau)^r\gamma\beta,(\gamma\nu\delta\rho\tau)^r\gamma\nu,
(\gamma\nu\delta\rho\tau)^r\gamma\nu\delta,(\gamma\nu\delta\rho\tau)^r\gamma\nu\delta\rho\right\}_{0\leqslant r \leqslant m-1}$$
$$\cup\left\{(\nu\delta\rho\tau\gamma)^r\beta,(\nu\delta\rho\tau\gamma)^r\nu, (\nu\delta\rho\tau\gamma)^r\nu\delta, 
(\nu\delta\rho\tau\gamma)^r\nu\delta\rho,(\nu\delta\rho\tau\gamma)^r\nu\delta\rho\tau\right\}_{0\leqslant r \leqslant m-1}$$ 
$$\cup\left\{e_3,(\gamma\nu\delta\rho\tau)^s,(\nu\delta\rho\tau\gamma)^t 
\right\}_{1\leqslant s \leqslant m,1\leqslant t\leqslant m-1},$$

$$\mathcal{B}^\xi_4=\left\{\delta(\rho\tau\gamma\nu\delta)^r,\delta(\rho\tau\gamma\nu\delta)^r\rho,\delta(\rho\tau\gamma\nu\delta)^r\rho\tau, 
\delta(\rho\tau\gamma\nu\delta)^r\rho\tau\gamma\right\}_{0\leqslant r \leqslant m-1}$$
$$\cup\left\{e_4,\delta(\rho\tau\gamma\nu\delta)^s\rho\tau\gamma\nu,
\delta(\rho\tau\gamma\nu\delta)^{t-1}\rho\tau\gamma\beta\right\}_{0\leqslant s \leqslant m-1,1\leqslant t\leqslant m-1}.$$ 

In particular, the Cartan matrix $C_{D(m,\lambda,\xi)}$ of $D(m,\lambda,\xi)$ is of the form 
$$\left[\begin{array}{cccc} 4m & 2m & 4m & 2m \\ 2m & m+1 & 2m & m-1 \\ 4m & 2m & 4m & 2m \\ 2m & m-1 & 2m & m+1
\end{array}\right].$$
and $\dim_K D(m,\lambda,\xi)=36m$. Independently, observe that $m^*_\eta=m$, $n^*_\eta=5$ and $n^\nu_\eta=1$ for $\eta\in Q(\xi)_1^*=
\{\gamma,\nu,\delta,\rho,\tau\}$, so we have 
$$\sum_{\eta\in Q(\xi)_1^*}m^*_\eta(n^*_\eta+n^\nu_\eta)+m^*_\delta(n^*_\delta+n^\nu_\delta)=5m(5+1)+m(5+1)=36m.$$

We also note that, for the choice $\mu\in\mathcal{O}(\xi)$, the algebra $D(m,\lambda,\mu)$ is isomorphic to $D(m,\lambda,\xi)$. 
\end{exmp}

\begin{exmp}\label{ex:4.2} Let $(Q,f)$ be the triangulation quiver 
$$\resizebox{1.3\textwidth}{0.04\textwidth}{$\qquad\qquad\qquad\qquad\quad
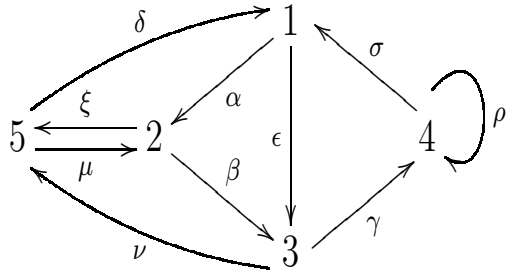
\xymatrix@R=0.05cm@C=1.2cm{
&&1\ar[lddd]^{\alpha}\ar[dddddd]_{\epsilon}&& \\
&&&& \\
&&&\ar@(ur,dr)[d]^{\rho} & \\ 
5\ar@<-0.1cm>[r]_{\mu}\ar@<0.15cm>@/^/[rruuu]^{\delta}&2\ar@<-0.1cm>[l]_{\xi}\ar[rddd]^{\beta}&&4\ar@<-0.15cm>[luuu]_{\sigma}&
&&&& \\
&&&& \\ 
&&&& \\
&&3\ar@<0.15cm>@/^/[lluuu]^{\nu}\ar@<-0.09cm>[ruuu]_{\gamma}&&}$}$$

with $f$-orbits $(\alpha\mbox{ }\xi\mbox{ }\delta)$, $(\beta\mbox{ }\nu\mbox{ }\mu)$, $(\epsilon\mbox{ }\gamma\mbox{ }\sigma)$ 
and $(\rho)$. Then for the associated permutation $g=\bar{f}$ we have orbits 
$$\mathcal{O}(\alpha)=(\alpha\mbox{ }\beta\mbox{ }\gamma\mbox{ }\rho\mbox{ }\sigma),\,
\mathcal{O}(\nu)=(\nu\mbox{ }\delta\mbox{ }\epsilon),\,\mathcal{O}(\xi)=(\xi\mbox{ }\mu).$$

Let $m\in\mathbb{N}^*$ and $\lambda\in K^*$. We consider the weight function $m_\bullet:\mathcal{O}(g)\to\mathbb{N}^*$ and 
the parameter function $c_\bullet:\mathcal{O}(g)\to K^*$ given by 
$$m_{\mathcal{O}(\alpha)}=m,\, m_{\mathcal{O}(\nu)}=1,\, m_{\mathcal{O}(\xi)}=1,$$ 
$$c_{\mathcal{O}(\alpha)}=\lambda,\, c_{\mathcal{O}(\nu)}=1,\, c_{\mathcal{O}(\xi)}=1.$$ 
Then the associated weighted triangulation algebra $A=A(m,\lambda)=\Lambda(Q,f,m_\bullet,c_\bullet)$ is given by its 
Gabriel quiver $Q_A$ of the form  
$$\resizebox{1.3\textwidth}{0.04\textwidth}{$\qquad\qquad\qquad\qquad\quad
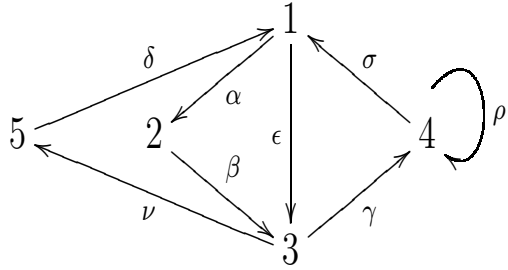
\xymatrix@R=0.05cm@C=1.2cm{
&&1\ar[lddd]^{\alpha}\ar[dddddd]_{\epsilon}&& \\
&&&& \\
&&&\ar@(ur,dr)[d]^{\rho} & \\ 
5\ar[rruuu]^{\delta}&2 \ar[rddd]^{\beta}&&4\ar[luuu]_{\sigma}&
&&&& \\
&&&& \\ 
&&&& \\
&&3\ar[lluuu]^{\nu}\ar[ruuu]_{\gamma}&&}$}$$
and the relations: 
$$\beta\nu\delta=\lambda(\beta\gamma\rho\sigma\alpha)^{m-1}\beta\gamma\rho\sigma, 
\quad\alpha\beta\nu=\epsilon\nu,$$ 
$$\nu\delta\alpha=\lambda(\gamma\rho\sigma\alpha\beta)^{m-1}\gamma\rho\sigma\alpha, 
\quad\delta\alpha\beta=\delta\epsilon,$$
$$\epsilon\gamma=\lambda(\alpha\beta\gamma\rho\sigma)^{m-1}\alpha\beta\gamma\rho, 
\quad\gamma\sigma=\nu\delta,$$ 
$$\sigma\epsilon=\lambda(\rho\sigma\alpha\beta\gamma)^{m-1}\rho\sigma\alpha\beta, 
\quad \rho^2\sigma=0,$$ 
$$\rho^2=\lambda(\sigma\alpha\beta\gamma\rho)^{m-1}\sigma\alpha\beta\gamma,\quad\gamma\rho^2=0,$$
$$\alpha\beta\nu\delta\alpha=0,\beta\nu\delta\epsilon=0,\nu\delta\alpha\beta\nu=0,\delta\alpha\beta\gamma=0,$$
$$\epsilon\gamma\rho=0, \sigma\epsilon\nu=0, \delta\epsilon\gamma=0, \rho\sigma\epsilon=0,$$
$$\epsilon\nu\delta\alpha=0, \sigma\alpha\beta\nu=0, \beta\nu\delta\alpha\beta=0,\delta\alpha\beta\nu\delta=0.$$
Observe that we have no zero relations of the shape 
$$\delta\alpha\beta=\delta f(\delta)g(f(\delta))=0,\qquad\mbox{because }f^2(\delta)=\xi\mbox{ is virtual},$$
$$\beta\nu\delta=\beta f(\beta)g(f(\beta))=0,\qquad\mbox{because }f^2(\beta)=\mu\mbox{ is virtual},$$
$$\gamma\sigma\alpha=\gamma f(\gamma)g(f(\gamma))=0,\mbox{ because }f(\bar{\gamma})=\mu
\mbox{ is virtual and }m_{\bar{\gamma}}=1,n_{\bar{\gamma}}=3,$$
$$\alpha\beta\nu=\alpha g(\alpha )f(g(\alpha))=0,\qquad\mbox{because }f(\alpha)=\xi\mbox{ is virtual},$$
$$\nu\delta\alpha=\nu g(\nu)f(g(\nu))=0,\qquad\mbox{because }f(\nu)=\mu\mbox{ is virtual},$$
$$\beta\gamma\sigma=\beta g(\beta)f(g(\beta))=0,\mbox{ because }f^2(\beta)=\mu
\mbox{ is virtual and }m_{f(\beta)}=1,n_{f(\beta)}=3.$$
We also observe that $\gamma\sigma\alpha=\nu\delta\alpha$ and $\beta\gamma\sigma=\beta\nu\delta$. Moreover, 
$\dim_K A(m,\lambda)=m\cdot 5^2+ 3^2 + 2^2=25m+13$. In fact, by Proposition \ref{prop:2.6}, the Cartan matrix $C_{A(m,\lambda)}$ of 
$A(m,\lambda)$ is of the form 
$$\left[\begin{array}{ccccc} 
m+1 & m & m+1 & 2m & 1 \\ 
m & m+1 & m & 2m & 1 \\ 
m+1 & m & m+1 & 2m & 1 \\ 
2m & 2m & 2m & 4m & 0 \\ 
1 & 1 & 1 & 0 & 2 
 \end{array}\right].$$

We now take $\mathcal{O}=\mathcal{O}(\xi)$ and consider two virtually mutated algebras with respect to two possible 
choices of elements from $\mathcal{O}$. 

(1) Let first $A(\xi)=A(m,\lambda,\xi)=\Lambda(Q,f,m_\bullet,c_\bullet,\xi)$. Then the algebra $A(\xi)$ is given by the 
quiver $Q(\xi)$ of the form
$$\resizebox{1.3\textwidth}{0.04\textwidth}{$\qquad\qquad\qquad\qquad\qquad
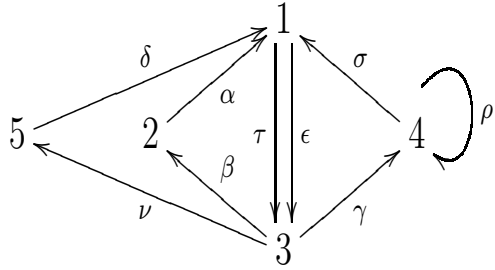
\xymatrix@R=0.05cm@C=1.2cm{
&&1\ar@<0.1cm>[dddddd]^{\epsilon}\ar@<-0.1cm>[dddddd]_{\tau}&& \\
&&&& \\
&&&\ar@(ur,dr)[d]^{\rho} & \\ 
5\ar[rruuu]^{\delta}&2 \ar[ruuu]_{\alpha}&&4\ar[luuu]_{\sigma}&
&&&& \\
&&&& \\ 
&&&& \\
&&3\ar[lluuu]^{\nu}\ar[ruuu]_{\gamma}\ar[luuu]_{\beta}&&}$}$$
and the relations: 
$$\nu\delta=\beta\alpha+\lambda(\gamma\rho\sigma\tau)^{m-1}\gamma\rho\sigma,\tau\nu=\epsilon\nu,
\delta\tau=\delta\epsilon,$$
$$\alpha\tau=0,\tau\beta=0,\gamma\sigma=\nu\delta,
\epsilon\gamma=\lambda(\tau\gamma\rho\sigma)^{m-1}\tau\gamma\rho,$$ 
$$\sigma\epsilon=\lambda(\rho\sigma\tau\gamma)^{m-1}\rho\sigma\tau,
\rho^2=\lambda(\sigma\tau\gamma\rho)^{m-1}\sigma\tau\gamma,$$
$$\delta\tau\gamma=0,\epsilon\gamma\rho=0,\sigma\epsilon\nu=0,\rho^2\sigma=0,$$
$$\delta\epsilon\gamma=0,\gamma\rho^2=0,\rho\sigma\epsilon=0.$$ 
Applying Proposition \ref{prop:3.9} we conclude that the Cartan matrix $C_{A(\xi)}$ of $A(\xi)$ is of the form 
$$\left[\begin{array}{ccccc} 
m+1 & 1 & m+1 & 2m & 1 \\ 
1 & 2 & 1 & 0 & 0 \\ 
m+1 & 1 & m+1 & 2m & 1 \\ 
2m & 0 & 2m & 4m & 0 \\ 
1 & 0 & 1 & 0 & 2 
 \end{array}\right].$$

In particular, we have $\dim_K A(\xi)=16m+16$. 

(2) Now, consider $A(\mu)=A(m,\lambda,\mu)=\Lambda(Q,f,m_\bullet,c_\bullet,\mu)$. Then $A(\mu)$ is given by the quiver 
$Q(\mu)$ of the form 
$$\resizebox{1.3\textwidth}{0.04\textwidth}{$\qquad\qquad\qquad\qquad\qquad
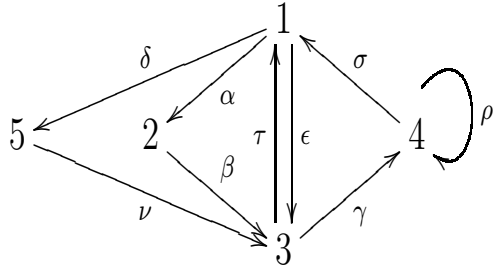
\xymatrix@R=0.05cm@C=1.2cm{
&&1\ar[lddd]^{\alpha}\ar@<0.1cm>[dddddd]^{\epsilon}\ar[llddd]_{\delta}&& \\
&&&& \\
&&&\ar@(ur,dr)[d]^{\rho} & \\ 
5\ar[rrddd]_{\nu}&2 \ar[rddd]^{\beta}&&4\ar[luuu]_{\sigma}&
&&&& \\
&&&& \\ 
&&&& \\
&&3\ar[ruuu]_{\gamma}\ar@<0.1cm>[uuuuuu]^{\tau}&&}$}$$
and the relations (we note that $m_{\mathcal{\nu}}n_{\mathcal{\nu}}=3$): 
$$\alpha\beta=\delta\nu+c^*_{\tilde{\alpha}}A^*_{\tilde{\alpha}}=\delta\nu+A^*_\epsilon=\delta\nu+\epsilon,$$
$$\tau\alpha=c^*_{\tilde{\tau}}A^*_{\tilde{\tau}}=c_{\gamma}A^*_\gamma=
\lambda(\gamma\rho\sigma\alpha\beta)^{m-1}\gamma\rho\sigma\alpha,$$
$$\beta\tau=c_{\beta}A^*_\beta=\lambda(\beta\gamma\rho\sigma\alpha)^{m-1}\beta\gamma\rho\sigma,\tau\delta=0,\nu\tau=0,$$ 
$$\epsilon\gamma=\lambda(\alpha\beta\gamma\rho\sigma)^{m-1}\alpha\beta\gamma\rho,\sigma\epsilon=
\lambda(\rho\sigma\alpha\beta\gamma)^{m-1}\rho\sigma\alpha\beta,$$ 
$$\gamma\sigma=c^*_{\tilde{\gamma}}A^*_{\tilde{\gamma}}=c^*_\tau A^*_\tau=\tau,\rho^2=\lambda(\sigma\alpha\beta\gamma\rho)^{m-1}\sigma\alpha\beta\gamma,$$
$$\sigma\epsilon\tau=0,\epsilon\gamma\rho=0,\rho^2\sigma=0,\beta\tau\epsilon=0,$$
$$\tau\epsilon\gamma=0,\epsilon\tau\alpha=0,\gamma\rho^2=0,\rho\sigma\epsilon=0.$$ 
In particular, we have $\epsilon=\alpha\beta-\delta\nu$ and $\tau=\gamma\sigma$. Hence $A(\mu)$ is given by its Gabriel quiver 
$$\resizebox{1.3\textwidth}{0.04\textwidth}{$\qquad\qquad\qquad\qquad\qquad
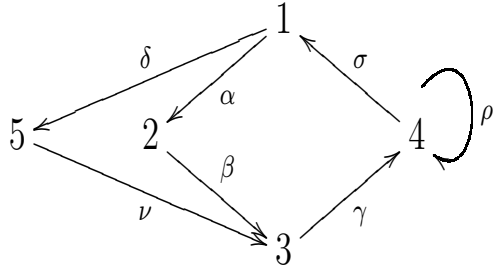
\xymatrix@R=0.05cm@C=1.2cm{
&&1\ar[lddd]^{\alpha}\ar[llddd]_{\delta}&& \\
&&&& \\
&&&\ar@(ur,dr)[d]^{\rho} & \\ 
5\ar[rrddd]_{\nu}&2 \ar[rddd]^{\beta}&&4\ar[luuu]_{\sigma}&
&&&& \\
&&&& \\ 
&&&& \\
&&3\ar[ruuu]_{\gamma}&&}$}$$ 
and the following relations: 
$$\alpha\beta\gamma=\delta\nu\gamma+\lambda(\alpha\beta\gamma\rho\sigma)^{m-1}\alpha\beta\gamma\rho,\beta\gamma\sigma=
\lambda(\beta\gamma\rho\sigma\alpha)^{m-1}\beta\gamma\rho\sigma,$$
$$\sigma\alpha\beta=\sigma\delta\nu+\lambda(\rho\sigma\alpha\beta\gamma)^{m-1}\rho\sigma\alpha\beta,\gamma\sigma\alpha=
\lambda(\gamma\rho\sigma\alpha\beta)^{m-1}\gamma\rho\sigma\alpha,$$
$$\rho^2=\lambda(\sigma\alpha\beta\gamma\rho)^{m-1}\sigma\alpha\beta\gamma,\gamma\sigma\delta=0,\nu\gamma\sigma=0,\rho^2\sigma=0,\gamma\rho^2=0,$$
$$\alpha\beta\gamma\sigma\alpha=0,\beta\gamma\sigma\alpha\beta=0, \gamma\sigma\alpha\beta\gamma=0,\sigma\alpha\beta\gamma\sigma=0.$$

Applying Lemma \ref{lem:3.6} and Proposition \ref{prop:3.9} we infer that the Cartan matrix $C_{A(\mu)}$ of $A(\mu)$ is of the form 
$$\left[\begin{array}{ccccc} 
m+1 & m & m+1 & 2m & m \\ 
m & m+1 & m & 2m & m-1 \\ 
m+1 & m & m+1 & 2m & m \\ 
2m & 2m & 2m & 4m & 2m \\ 
m & m-1 & m & 2m & m+1 
\end{array}\right].$$

In particular, $\dim_K A(\mu)=36m+4$.

We note that the algebras $A(\xi)$ and $A(\mu)$ are not isomorphic, since their Gabriel quivers are different, as well 
as their dimensions never coincide.
\end{exmp}

\begin{exmp}\label{ex:4.3} (\emph{Virtually mutated generalized spherical algebras}). 

We considerthe following triangulation $T(2)$ of the sphere $S^2$ in $\mathbb{R}^3$ 
$$
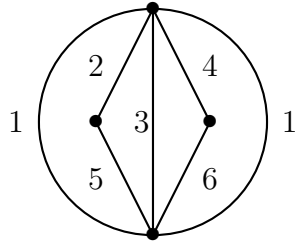
\begin{tikzpicture}[scale=1.5]
\draw[thick] (0,-1)--(0,1); 
\draw[thick] (0,-1)--(0.5,0);
\draw[thick] (0,-1)--(-0.5,0);
\draw[thick] (-0.5,0)--(0,1);
\draw[thick] (0.5,0)--(0,1);
\draw[thick](0,0) circle (1);
\node() at (-0.5,0){$\bullet$};
\node() at (0.5,0){$\bullet$};
\node() at (0,-1){$\bullet$}; 
\node() at (0,1){$\bullet$}; 
\node() at (1.2,0){$1$};
\node() at (-1.2,0){$1$};
\node() at (-0.5,0.5){$2$};
\node() at (-0.1,0){$3$};
\node() at (0.5,0.5){$4$};
\node() at (-0.5,-0.5){$5$};
\node() at (0.5,-0.5){$6$};
\end{tikzpicture}$$ 
and the coherent orientation $\overrightarrow{T(2)}$ of triangles in $T(2)$: $(1\mbox{ }2\mbox{ }5)$, $(2\mbox{ }3\mbox{ }5)$, 
$(3\mbox{ }4\mbox{ }6)$ and $(4\mbox{ }1\mbox{ }6)$. Then the associated triangulation quiver $(Q,f)=(Q(S^2,\overrightarrow{T(2)}),f)$ 
is of the form 
$$
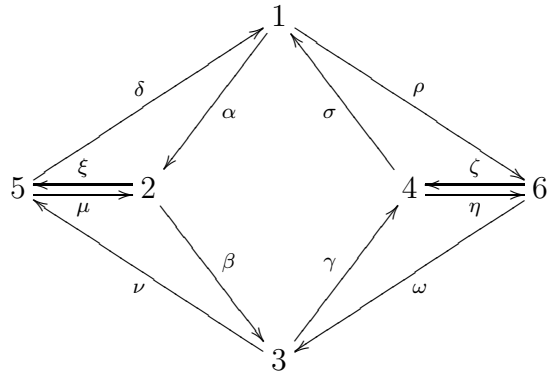
\xymatrix@R=1.8cm@C=1.3cm{&& 1\ar[ld]^{\alpha}\ar[rrd]^{\rho} && \\ 
5\ar[rru]^{\delta}\ar@<-0.07cm>[r]_{\mu} & 2\ar[rd]^{\beta}\ar@<-0.07cm>[l]_{\xi}& &
4\ar[lu]^{\sigma}\ar@<-0.07cm>[r]_{\eta}&6\ar[lld]^{\omega}\ar@<-0.07cm>[l]_{\zeta} \\ 
&&3\ar[llu]^{\nu}\ar[ru]^{\gamma} &&}$$ 
with $f$-orbits $(\alpha\mbox{ }\xi\mbox{ }\delta)$, $(\beta\mbox{ }\nu\mbox{ }\mu)$, $(\gamma\mbox{ }\eta\mbox{ }\omega)$, 
$(\sigma\mbox{ }\rho\mbox{ }\zeta)$. Then the the permutation $g=\bar{f}$ has four orbits 
$$\mathcal{O}(\alpha)=(\alpha\mbox{ }\beta\mbox{ }\gamma\mbox{ }\sigma),\,\mathcal{O}(\rho)=(\rho\mbox{ }\omega\mbox{ }\nu\mbox{ }\delta),\,
\mathcal{O}(\xi)=(\xi\mbox{ }\mu),\, \mathcal{O}(\eta)=(\eta\mbox{ }\zeta).$$
Let $m,n\in\mathbb{N}^*$ and $a,b\in K^*$. We define the weight function $m_\bullet:\mathcal{O}(g)\to\mathbb{N}^*$ and 
the parameter function $c_\bullet:\mathcal{O}(g)\to K^*$ such that 
$$m_{\mathcal{O}(\alpha)}=m,\, m_{\mathcal{O}(\rho)}=n,\, m_{\mathcal{O}(\xi)}=1,\, m_{\mathcal{O}(\eta)}=1,$$ 
$$c_{\mathcal{O}(\alpha)}=a,\, c_{\mathcal{O}(\rho)}=b,\, c_{\mathcal{O}(\xi)}=1,\, c_{\mathcal{O}(\eta)}=1.$$
Hence the orbits $\mathcal{O}(\xi)$ and $\mathcal{O}(\eta)$ consist of virtual arrows. Moreover, if $m=n=1$, then 
we assume that $ab\neq 1$ (see \cite[Example 3.6]{ES5}). We consider the associated triangulation algebra 
$S(m,n,a,b)=\Lambda(Q,f,m_\bullet,c_\bullet)$, and call it a \emph{generalized spherical algebra}. Note that $S(1,1,a,b)$ 
is isomorphic to the non-singular spherical algebra $S(ab)$ introduced in \cite{ES5}. 

The algebra $S(m,n,a,b)$ is given by its Gabriel quiver $Q_S$ of the form 
$$
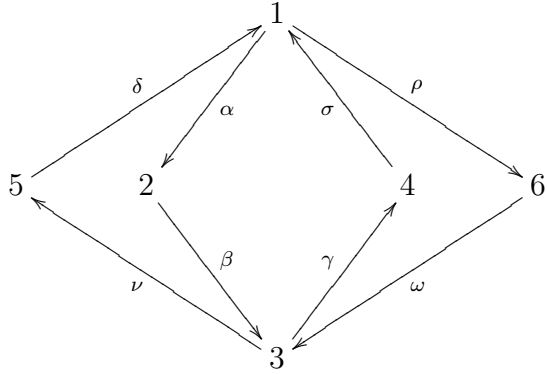
\xymatrix@R=1.8cm@C=1.3cm{&& 1\ar[ld]^{\alpha}\ar[rrd]^{\rho} && \\ 
5\ar[rru]^{\delta}& 2\ar[rd]^{\beta} & &
4\ar[lu]^{\sigma}&6\ar[lld]^{\omega} \\ 
&&3\ar[llu]^{\nu}\ar[ru]^{\gamma} &&}$$ 
and the relations 
$$\alpha\beta\nu=b(\rho\omega\nu\delta)^{n-1}\rho\omega\nu,\quad 
\beta\nu\delta=a(\beta\gamma\sigma\alpha)^{m-1}\beta\gamma\sigma,$$ 
$$\delta\alpha\beta=b(\delta\rho\omega\nu)^{n-1}\delta\rho\omega,\quad 
\nu\delta\alpha=a(\gamma\sigma\alpha\beta)^{m-1}\gamma\sigma\alpha,$$ 
$$\gamma\sigma\rho=b(\nu\delta\rho\omega)^{n-1}\nu\delta\rho,\quad 
\sigma\rho\omega=a(\sigma\alpha\beta\gamma)^{m-1}\sigma\alpha\beta,$$ 
$$\omega\gamma\sigma=b(\omega\nu\delta\rho)^{n-1}\omega\nu\delta,\quad 
\rho\omega\gamma=a(\alpha\beta\gamma\sigma)^{m-1}\alpha\beta\gamma,$$
$$\alpha\beta\nu\delta\alpha=0,\beta\nu\delta\rho=0,\nu\delta\alpha\beta\nu=0,\delta\alpha\beta\gamma=0,$$
$$\gamma\sigma\rho\omega\gamma=0,\sigma\rho\omega\nu=0,\rho\omega\gamma\sigma\rho=0,\omega\gamma\sigma\alpha=0,$$
$$\beta\gamma\sigma\rho=0,\sigma\alpha\beta\nu=0,\delta\rho\omega\gamma=0,\omega\nu\delta\alpha=0,$$
$$\beta\nu\delta\alpha\beta=0,\delta\alpha\beta\nu\delta=0,\sigma\rho\omega\gamma\sigma=0,
\omega\gamma\sigma\rho\omega=0.$$
Moreover, a minimal set of relations defining $S(m,n,a,b)$ is given by the above eight commutativity relations and the 
four zero relations: 
$$\beta\nu\delta\rho=0,\delta\alpha\beta\gamma=0,\sigma\rho\omega\nu=0,\omega\gamma\sigma\alpha=0.$$
Further, we have 
$$\dim_K S(m,n,a,b)=16m+16n+4+4=16(m+n)+8.$$

We shall now show some representative families of virtually mutated algebras of $S(m,n,a,b)$.  

\begin{enumerate}[(1)]
\item Let $\mathcal{O}=\mathcal{O}(\xi)$ and $\xi$ be the chosen element of $\mathcal{O}$. We use the following notation 
$$S(m,n,a,b,\xi):=\Lambda(Q,f,m_\bullet,c_\bullet,\xi).$$ 
We note that $s(\xi)=2$. The algebra $S(m,n,a,b,\xi)$ is given by the quiver $Q(\xi)$ of the form 
$$
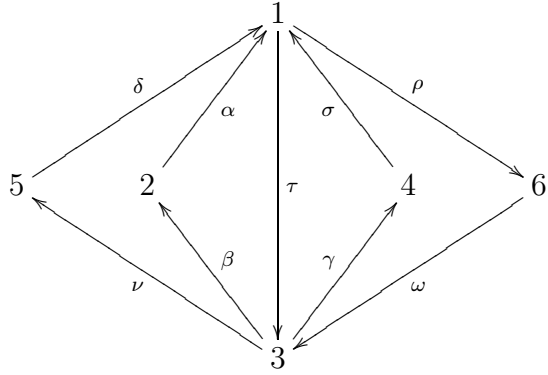
\xymatrix@R=1.8cm@C=1.3cm{&& 1\ar[rrd]^{\rho}\ar[dd]^{\tau} && \\ 
5\ar[rru]^{\delta}& 2\ar[ru]_{\alpha} & &
4\ar[lu]^{\sigma}&6\ar[lld]^{\omega} \\ 
&&3\ar[llu]^{\nu}\ar[ru]^{\gamma}\ar[lu]_{\beta} &&}$$ 
and the relations: 
$$\nu\delta=\beta\alpha+a(\gamma\sigma\tau)^{m-1}\gamma\sigma,\alpha\tau=0,\tau\beta=0,$$
$$\tau\nu=b(\rho\omega\nu\delta)^{n-1}\rho\omega\nu,\delta\tau=b(\delta\rho\omega\nu)^{n-1}\delta\rho\omega,$$
$$\gamma\sigma\rho=b(\nu\delta\rho\omega)^{n-1}\nu\delta\rho,\sigma\rho\omega=a(\sigma\tau\gamma)^{m-1}\sigma\tau,$$
$$\omega\gamma\sigma=b(\omega\nu\delta\rho)^{n-1}\omega\nu\delta,\rho\omega\gamma=a(\tau\gamma\sigma)^{m-1}\tau\gamma,$$
$$\delta\tau\gamma=0,\rho\omega\gamma\sigma\rho=0,\gamma\sigma\rho\omega\gamma=0,$$
$$\delta\rho\omega\gamma=0,\tau\gamma\sigma\rho=0,\sigma\tau\nu=0.$$
It follows from Theorem \ref{thm:3.14} that 
$$\dim_K S(m,n,a,b,\xi)=9m+25n+4.$$
In particular, we have 
$$\dim_K S(1,n,a,b,\xi)=25n+13,\mbox{ and}$$ 
$$\dim_K S(m,1,a,b,\xi)=9m+29.$$

\item Let $\mathcal{O}_1=\mathcal{O}(\xi)$ and $\mathcal{O}_2=\mathcal{O}(\eta)$ and $(\xi,\eta)$ be the chosen element 
of $\mathcal{O}_1\times\mathcal{O}_2$. We set 
$$S(m,n,a,b,\xi,\eta):=\Lambda(Q,f,m_\bullet,c_\bullet,(\xi,\eta)).$$ 
Then obviously $s(\xi)=2=t(\alpha)$, $s(\eta)=4=t(\gamma)$ and $\alpha,\gamma$ belong to the same $g$-orbit. The algebra 
$S(m,n,a,b,\xi,\eta)$ is given by the quiver $Q(\xi,\eta)$ of the form 
$$
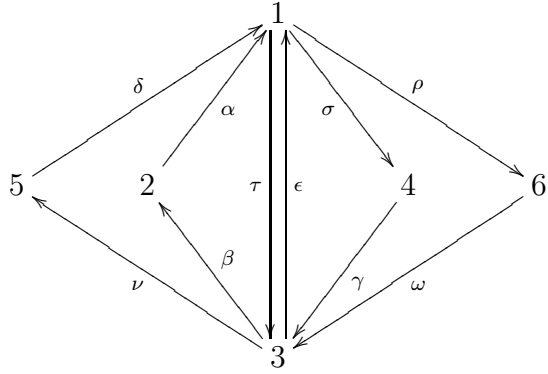
\xymatrix@R=1.8cm@C=1.3cm{
&& 1\ar[rrd]^{\rho}\ar[rd]_{\sigma} \ar@<-0.1cm>[dd]_{\tau} && \\ 
5\ar[rru]^{\delta}& 2\ar[ru]_{\alpha} & &
4\ar[ld]^{\gamma}&6\ar[lld]^{\omega} \\ 
&&3\ar[llu]^{\nu}\ar[lu]_{\beta}\ar@<-0.1cm>[uu]_{\epsilon} &&}$$ 
and the relations: 
$$\nu\delta=\beta\alpha+a(\epsilon\tau)^{m-1}\epsilon,\alpha\tau=0,\tau\beta=0,$$
$$\tau\nu=b(\rho\omega\nu\delta)^{n-1}\rho\omega\nu,\delta\tau=b(\delta\rho\omega\nu)^{n-1}\delta\rho\omega,$$
$$\rho\omega=\sigma\gamma+a(\tau\epsilon)^{m-1}\tau,\gamma\epsilon=0,\epsilon\sigma=0,$$
$$\epsilon\rho=b(\nu\delta\rho\omega)^{n-1}\nu\delta\rho,\omega\epsilon=b(\omega\nu\delta\rho)^{n-1}\omega\nu\delta,$$
$$\tau\epsilon\rho=0,\epsilon\tau\nu=0.$$ 
Moreover, it follows from Theorem \ref{thm:3.14} that 
$$\dim_K S(m,n,a,b,\xi,\eta)=4m+36n.$$ 
Note that if $m\geqslant 2$ then the above quiver is the Gabriel quiver of $S(m,n,a,b,\xi,\eta)$. 

Now, assume that $m=1$. Then we have the equalities $\nu\delta=\beta\alpha+a\epsilon$ and $\rho\omega=\sigma\gamma+a\tau$, so 
$\tau,\epsilon$ are not occurring in the Gabriel quiver, and hence $S(1,n,a,b,\xi,\eta)$ is given by a quiver isomorphic 
to $Q_S$ and the following relations: 
$$\rho\omega\nu=\sigma\gamma\nu+ab(\rho\omega\nu\delta)^{n-1}\rho\omega\nu,\alpha\rho\omega=\alpha\sigma\gamma,$$
$$\delta\rho\omega=\delta\sigma\gamma+ab(\delta\rho\omega\nu)^{n-1}\delta\rho\omega,\rho\omega\beta=\gamma\sigma\beta,$$
$$\nu\delta\rho=\beta\alpha\rho+ab(\nu\delta\rho\omega)^{n-1}\nu\delta\rho,\gamma\nu\delta=\gamma\beta\alpha,$$
$$\omega\nu\delta=\omega\beta\alpha+ab(\omega\nu\delta\rho)^{n-1}\omega\nu\delta,\nu\delta\sigma=\beta\alpha\sigma,$$
$$(\rho\omega\nu\delta)^{n}\rho=0,(\nu\delta\rho\omega)^{n}\nu=0.$$
We observe that the last two zero relations are consequences of the zero relations of the form 
$(\rho\omega-\sigma\gamma)(\nu\delta-\beta\alpha)\rho=0$, $(\nu\delta-\beta\alpha)(\rho\omega-\sigma\gamma)\nu=0$, and 
some of the above relations. Therefore, we conclude that, for $n\geqslant 2$ and $\lambda=ab$, the algebra 
$S(1,n,a,b,\xi,\eta)$ is isomorphic to the so called higher spherical algebra $S(n,\lambda)$ investigated in \cite{ES6}. 

\item Let now $\mathcal{O}_1$ and $\mathcal{O}_2$ be as in previous case (2), but take $(\xi,\zeta)$ as the chosen 
element of $\mathcal{O}_1\times\mathcal{O}_2$. We consider the virtually mutated algebra 
$$S(m,n,a,b,\xi,\zeta):=\Lambda(Q,f,m_\bullet,c_\bullet,(\xi,\zeta)).$$ 
Note that $s(\xi)=2=t(\alpha)$ and $s(\zeta)=6=t(\omega)$ and $\alpha,\omega$ belong to different $g$-orbits in $(Q,f)$. 
The algebra $S(m,n,a,b,\xi,\zeta)$ is given by the quiver $Q(\xi,\zeta)$ of the form 
$$
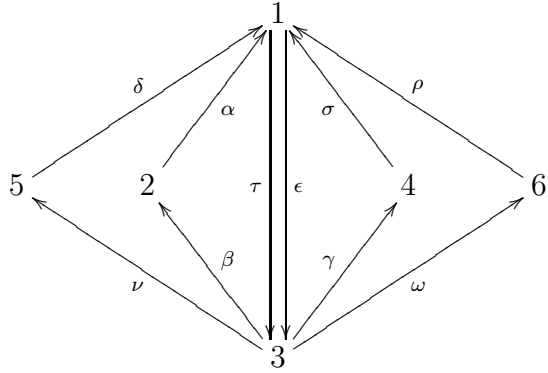
\xymatrix@R=1.8cm@C=1.3cm{&& 1\ar@<-0.1cm>[dd]_{\tau}\ar@<0.1cm>[dd]^{\epsilon} && \\ 
5\ar[rru]^{\delta}& 2\ar[ru]_{\alpha} & &
4\ar[lu]^{\sigma}&6\ar[llu]_{\rho} \\ 
&&3\ar[llu]^{\nu}\ar[ru]^{\gamma}\ar[lu]_{\beta}\ar[rru]_{\omega} &&}$$ 
and the relations: 
$$\nu\delta=\beta\alpha+a(\gamma\sigma\tau)^{m-1}\gamma\sigma,\alpha\tau=0,\tau\beta=0,$$
$$\gamma\sigma=\rho\omega+b(\nu\delta\epsilon)^{n-1}\nu\delta,\rho\epsilon=0,\epsilon\omega=0,$$
$$\tau\nu=b(\epsilon\nu\delta)^{n-1}\epsilon\nu,\delta\tau=b(\delta\epsilon\nu)^{n-1}\delta\epsilon,$$
$$\epsilon\gamma=a(\tau\gamma\sigma)^{m-1}\tau\gamma,\sigma\epsilon=a(\sigma\tau\gamma)^{m-1}\sigma\tau,$$
$$\delta\epsilon\gamma=0,\sigma\tau\nu=0.$$ 
Applying Theorem \ref{thm:3.14}, we conclude that 
$$\dim_K S(m,n,a,b,\xi,\zeta)=16(m+n).$$ 
\end{enumerate}

\end{exmp}

\section{Proof of Main Theorem}\label{sec:5} 

Let $\Lambda=\Lambda(Q,f,m_\bullet,c_\bullet)$ be a weighted triangulation algebra. For each vertex $i$ of the quiver $Q$, we 
denote by $P_i=e_i\Lambda$ the associated projective module in $\mod \Lambda$. Moreover, for any arrow $\theta$ from $j$ to $k$ 
in $Q$, we identify $\theta$ with the homomorphism $\theta:P_k\to P_j$ in $\mod\Lambda$ given by the left multiplication by $\theta$. 

Assume that $\mathcal{O}(g)$ contains a family $\mathcal{O}_1,\dots,\mathcal{O}_r$ of orbits with $|\mathcal{O}_i|=2$ and 
$m_{\mathcal{O}_i}=1$, for any $i\in\{1,\dots,r\}$, and $\xi=(\xi_1,\dots,\xi_r)\in\mathcal{O}_1\times\dots\times\mathcal{O}_r$. 
In Section \ref{sec:3} we defined the virtual mutation $\Lambda(\xi)=\Lambda(Q,f,m_\bullet,c_\bullet,\xi)$ of $\Lambda$ with 
respect to the sequence $\xi$ of virtual arrows. We use the notation established in Section \ref{sec:3}. In particular, for each 
$i\in\{1,\dots,r\}$, the triangulation quiver $(Q,f)$ contains a subquiver of the form 
$$
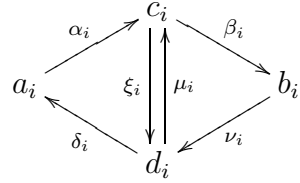
\xymatrix@R=0.5cm@C=1.2cm{&c_i\ar[rd]^{\beta_i}\ar@<-0.1cm>[dd]_{\xi_i}& \\ 
a_i\ar[ru]^{\alpha_i}&&b_i\ar[ld]^{\nu_i} \\ &d_i\ar[lu]^{\delta_i}\ar@<-0.1cm>[uu]_{\mu_i}& }$$ 
with $f$-orbits $(\alpha_i\mbox{ }\xi_i\mbox{ }\delta_i)$ and $(\beta_i\mbox{ }\nu_i\mbox{ }\mu_i)$, and $\mathcal{O}_i=
(\xi_i\mbox{ }\mu_i)$. We also note that in the Gabriel quiver $Q_\Lambda$ of $\Lambda$ the virtual arrows do not occur, and 
hence $Q_\Lambda$ has subquivers 
$$
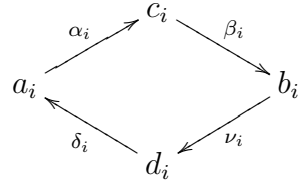
\xymatrix@R=0.5cm@C=1.2cm{&c_i\ar[rd]^{\beta_i}& \\ 
a_i\ar[ru]^{\alpha_i}&&b_i\ar[ld]^{\nu_i} \\ 
&d_i\ar[lu]^{\delta_i}& }$$ 
where $\alpha_i$ is the unique arrow of $Q_\Lambda$ with target $c_i$ and $\beta_i$ is the unique arrow of $Q_\Lambda$ with 
source $c_i$, for $i\in\{1,\dots,r\}$. We consider the decomposition $\Lambda_\Lambda=P'\oplus P''$ in $\mod\Lambda$, where 
$$P'=\bigoplus_{i=1}^r P_{c_i}\oplus\bigoplus_{x\in Q_0\setminus\{c_1,\dots,c_r\}} P_x.$$ 
Moreover, we consider the following complexes in the homotopy category $K^b(P_\Lambda)$ of projective modules 
modules in $\mod \Lambda$: 
$$T^\xi_x:\xymatrix{0\ar[r]&P_x\ar[r]&0}$$ 
concentrated in degree $0$, for all vertices $x\in Q_0$ different from $c_1,\dots,c_r$, and 
$$T^\xi_{c_i}: \xymatrix{0\ar[r]&P_{c_i}\ar[r]^{\alpha_i}&P_{a_i}\ar[r]&0}$$ 
concentrated in degrees $1$ and $0$, for all $i\in\{1,\dots,r\}$. We also set 
$$T^\xi=\bigoplus_{x\in Q_0}T^\xi_x.$$

\begin{lem}\label{lem:5.1} $T^\xi$ is a tilting complex in $K^b(P_\Lambda)$. \end{lem} 

\begin{proof} Let $\alpha:\bigoplus_{i=1}^r P_{c_i}\to\bigoplus_{i=1}^r P_{a_i}$ be the diagonal 
homomorphism given by the homomorphisms $\alpha_i:P_{c_i}\to P_{a_i}$. Since $\alpha_i$ is the 
unique arrow in $Q_\Lambda$ ending at $c_i$, for any $i\in\{1,\dots,r\}$, we conclude that $\alpha$ 
is a left $\add(P'')$-approximation of $P'$, so $T^\xi$ is indeed a tilting complex in $K^b(P_\Lambda)$, 
by Proposition \ref{pro:1.5}. \end{proof}

We define $\Gamma=\Gamma(\xi):=\End_{K^b(P_\Lambda)}(T^\xi)$. The following theorem is a direct 
consequence of Theorems \ref{thm:1.2}, \ref{thm:1.3}, \ref{thm:1.4} and \ref{thm:2.4}. 

\begin{theorem}\label{thm:5.2} The following statements hold.
\begin{enumerate}[$(1)$]
\item $\Gamma(\xi)$ is a finite-dimensional symmetric algebra. 

\item $\Gamma(\xi)$ is a representation-infinite tame algebra. 

\item $\Gamma(\xi)$ is a periodic algebra of period $4$.
\end{enumerate}
\end{theorem}

We shall prove now that the algebras $\Lambda(\xi)$ and $\Gamma(\xi)$ are isomorphic. 

Observe first that $\tilde{P_x}=\Hom_{K^b(P_\Lambda)}(T^\xi,T^\xi_x)$, $x\in Q_0$, form a complete family of 
pairwise non-isomorphic indecomposable projective modules in $\mod\Gamma$. Moreover, we define the following 
morphisms in $K^b(P_\Lambda)$ between indecomposable direct summands of $T^\xi$: 
$$\hat{\alpha_i}:T^\xi_{a_i}\to T^\xi_{c_i} \qquad\mbox{given by }id:P_{a_i}\to P_{a_i},$$
$$\hat{\beta_i}:T^\xi_{c_i}\to T^\xi_{b_i} \qquad\mbox{given by }\nu_i\delta_i-c_{\gamma_i}A'_{\gamma_i}:P_{a_i}\to P_{b_i},$$ 
$$\hat{\tau_i}:T^\xi_{b_i}\to T^\xi_{a_i} \qquad\mbox{given by }\alpha_i\beta_i:P_{b_i}\to P_{a_i},$$ 
for all $i\in\{1,\dots,r\}$, and 
$$\hat{\eta}:T^\xi_{t(\eta)}\to T^\xi_{s(\eta)} \qquad\mbox{given by }\eta:P_{t(\eta)}\to P_{s(\eta)},$$ 
for any arrow $\eta\in Q_1$ different from the arrows $\alpha_i,\beta_i$, $i\in\{1,\dots,r\}$. 

Then applying covariant functor $\Hom_{K^b(P_\Lambda)}(T^\xi,-)$ to above morphisms, we obtain the following 
homomorphisms between indecomposable projective modules in $\mod\Gamma$: 
$$\alpha_i=\Hom_{K^b(P_\Lambda)}(T^\xi,\hat{\alpha_i}):\widetilde{P}_{a_i}\to\widetilde{P}_{c_i},$$
$$\beta_i=\Hom_{K^b(P_\Lambda)}(T^\xi,\hat{\beta_i}):\widetilde{P}_{c_i}\to\widetilde{P}_{b_i},$$ 
$$\tau_i=\Hom_{K^b(P_\Lambda)}(T^\xi,\hat{\tau_i}):\widetilde{P}_{b_i}\to\widetilde{P}_{a_i},$$ 
for all $i\in\{1,\dots,r\}$, and 
$$\eta=\Hom_{K^b(P_\Lambda)}(T^\xi,\hat{\eta}):\widetilde{P}_{t(\eta)}\to\widetilde{P}_{s(\eta)},$$ 
for any arrow $\eta\in Q_1$ different from $\alpha_i,\beta_i$, $i\in\{1,\dots,r\}$. We observe also that these 
homomorphisms correspond to the arrows of the quiver $Q(\xi)$ defining the algebra $\Lambda(\xi)$. 

We will identify the homomorphisms $\alpha_i,\beta_i,\tau_i$ and $\eta$ with the elements of $e_{c_i}(\rad\Gamma)e_{a_i}$, 
$e_{b_i}(\rad\Gamma)e_{c_i}$, $e_{a_i}(\rad\Gamma)e_{b_i}$ and $e_{s(\eta)}(\rad\Gamma)e_{t(\eta)}$, respectively, 
corresponding to them. We note the following obvious fact.  

\begin{lem}\label{lem:5.3} The elements $\alpha_i,\beta_i,\tau_i$, $i\in\{1,\dots,r\}$, and 
$\eta\in Q_1$, different from $\alpha_i,\beta_i$, $i\in\{1,\dots,r\}$, generate $\rad\Gamma$.
\end{lem}


In the next proposition we use notation established for definition of the algebra $\Lambda(\xi)$. 

\begin{prop}\label{prop:5.4} The relations $(1)$-$(5)$ defining the algebra $\Lambda(\xi)$ hold also in $\Gamma(\xi)$. \end{prop} 

\begin{proof} (1) We fix $i\in\{1,\dots,r\}$ and prove that relations from $(1)$ hold in three steps. 
First, we prove that the equality $\nu_i\delta_i=\beta_i\alpha_i+c^*_{\tilde{\nu_i}}A^*_{\tilde{\nu_i}}$ holds. Indeed, we 
have $\tilde{\nu_i}=\gamma_i$, so $c^*_{\tilde{\nu_i}}=c_{\gamma_i}$, and moreover 
$$\mathcal{O}(\alpha_i)=(\alpha_i\mbox{ }\beta_i\mbox{ }\gamma_i\dots g^{n_{\gamma_i}-3}(\gamma_i)),$$ 
hence $\mathcal{O}^*(\tau_i)=(\tau_i\mbox{ }\gamma_i\dots (g^*)^{-1}(\tau_i))$. It is easy to see that $A^*_{\gamma_i}$ is 
obtained from $A'_{\gamma_i}$ by replacing all paths of the form $\alpha_j\beta_j$ by $\tau_j$, $j\in\{1,\dots,r\}$. 
Further, $\beta_i\alpha_i$ in $\Gamma(\xi)$ is identified with the map $T^\xi_{a_i}\to T^\xi_{b_i}$ given by 
$v_i\delta_i-c_{\gamma_i}A'_{\gamma_i}:P_{a_i}\to P_{b_i}$. It is now clear that the equality $\beta_i\alpha_i=
\nu_i\delta_i-c^*_{\gamma_i}A^*_{\gamma_i}$ holds in $\Gamma(\xi)$, so we are done (note that $\tau_j$ as an element 
of $\Gamma(\xi)$ is identified with the map $T^\xi_{b_j}\to T^\xi_{a_j}$ given by $\alpha_j\beta_j:P_{b_j}\to P_{a_j}$). 

Since $\alpha_i\tau_i$ in $\Gamma$ is identified with the map $\hat{\alpha_i}\hat{\tau_i}:T^\xi_{b_i}\to 
T^\xi_{c_i}$ given by $\alpha_i\beta_i:P_{b_i}\to P_{a_i}$ (in degree $0$), we deduce that $\alpha_i\tau_i=0$ in $\Gamma$, 
because $\hat{\alpha_i}\hat{\tau_i}$ is homotopic to zero, as it is given by the homomorphism which factors through 
$\alpha_i$. 

The equality $\tau_i\beta_i=0$ holds in $\Gamma$, because $\hat{\tau_i}\hat{\beta_i}$ is given by 
$\alpha_i\beta_i(\nu_i\delta_i-c_{\gamma_i}A'_{\gamma_i}):P_{a_i}\to P_{a_i}$ and we have also the following equalities 
in $\Lambda$ 
$$\alpha_i\beta_i\nu_i\delta_i=\alpha_i\xi_i\delta_i=c_{\beta_i}\alpha_i A_{\beta_i}=c_{\beta_i}B_{\alpha_i}=
c_{\gamma_i}\alpha_i\beta_i A'_{\gamma_i}.$$

(2) Let $\eta$ be an arrow in $Q(\xi)^*_1$ different from $\nu_i$. Then $\eta=\tau_i$ or $\eta=\delta_i$ (for $i\in\{1,\dots,r\}$) 
or $\eta$ is an arrow of $Q_1$ different from $\alpha_i,\beta_i,\nu_i,\delta_i$, $i\in\{1,\dots,r\}$. 

If $\eta=\tau_i$, $i\in\{1,\dots,r\}$, then $\eta f^*(\eta)=\tau_i\nu_i$ and we have the equalities in $\Lambda$ 
$$\alpha_i\beta_i\nu_i=\alpha_i\xi_i=\alpha_i f(\alpha_i)=c_{\bar{\alpha}_i}A_{\bar{\alpha}_i}=c_{\rho_i}A_{\rho_i}.$$ 
Moreover, $\rho_i=\tilde{\tau}_i$, $c_{\rho}=c^*_{\tilde{\tau_i}}$, and $A^*_{\tilde{\tau_i}}$ is obtained from $A_{\rho_i}$ 
by replacing all paths $\alpha_j\beta_j$ by $\tau_j$, $j\in\{1,\dots,r\}$. Since $\tau_i$ is identified with $\alpha_i\beta_i$ 
in $\Lambda$, we conclude that the required equality $\eta f^*(\eta)=c^*_{\tilde{\eta}}A^*_{\tilde{\eta}}$ holds in $\Gamma$. 

For $\eta=\delta_i$, $i\in\{1,\dots,r\}$, we have $\eta f^*(\eta)=\delta_i\tau_i$ and the following equalities hold in 
$\Lambda$ $$\delta_i\alpha_i\beta_i=\mu_i\beta_i=\mu_i f(\mu_i)=c_{\bar{\mu_i}}A_{\bar{\mu_i}}=c_{\delta_i}A_{\delta_i},$$ 
because $\bar{\mu_i}=\delta_i$. Observe also that $\tilde{\delta}_i=\delta_i$, and $A^*_{\delta_i}$ is obtained from 
$A_{\delta_i}$ by replacing all paths $\alpha_j\beta_j$ by $\tau_j$, for $j\in\{1,\dots,r\}$. Hence the required equality 
$\eta f^*(\eta)=c^*_{\tilde{\eta}}A^*_{\tilde{\eta}}$ holds in $\Gamma$, because $\tau_i$ is given by $\alpha_i\beta_i$. 
 
Finally, assume that $\eta$ is in $Q_1$ and different from $\alpha_i,\beta_i,\nu_i,\delta_i$. Then also $f(\eta)\in Q_1$ is different 
from all these arrows, and the equality $\eta f(\eta)=c_{\bar{\eta}} A_{\bar{\eta}}$ holds in $\Lambda$. Moreover, if $\eta=\rho_i$, 
$i\in\{1,\dots,r\}$, then $\tilde{\eta}=\tau_i\neq\bar{\eta}=\alpha_i$, and otherwise $\tilde{\eta}=\bar{\eta}$, and in both 
cases we have $f^*(\eta)=f(\eta)$. Note that then $A^*_{\tilde{\eta}}=A_{\bar{\eta}}$ or $A^*_{\tilde{\eta}}$ is obtained 
from $A_{\bar{\eta}}$ by similar replacement as above, so the required equality $\eta f^*(\eta)=c^*_{\tilde{\eta}}A^*_{\tilde{\eta}}$ 
holds. 

(3) Let $\eta=\delta_i$ for some $i\in\{1,\dots,r\}$. Then we have $f^*(\eta)=\tau_i$, and $g^*(\tau_i)=\gamma_i$ or 
$g^*(\tau_i)=\tau_j$ for some $j\in\{1,\dots,r\}$. If $g^*(\tau_i)=\gamma_i$, then we have the equalities 
$$\eta f^*(\eta) g^*(f^*(\eta))=\delta_i\tau_i\gamma_i=\delta_i\alpha_i\beta_i\gamma_i=\mu_i\beta_i\gamma_i = 
\mu_if(\mu_i)g(f(\mu_i))=0.$$ 
For $g^*(\tau_i)=\tau_j$, the following equalities hold 
$$\eta f^*(\eta) g^*(f^*(\eta))=\delta_i\tau_i\tau_j=\delta_i\alpha_i\beta_i\alpha_j\beta_j = \mu_i\beta_i\alpha_j\beta_j = 
\mu_i f(\mu_i) g(f(\mu_i))\beta_j=0.$$ 
Therefore $\eta f^*(\eta) g^*(f^*(\eta))=0$ holds for $\eta=\delta_i$. 

Now assume $\eta$ is an arrow in $Q_1$ such that $f^2(\eta)$ is not virtual and $f(\bar{\eta})$ is not virtual if 
$m_{\bar{\eta}}=1$ and $n_{\bar{\eta}}=3$. Moreover, we assume that $\eta$ is different from $\nu_i,\tau_i$, and $\delta_i$, for $i\in\{1,\dots,r\}$. 
Clearly, then we have $f^*(\eta)=f(\eta)$. If also $g^*(f(\eta))=g(f(\eta))$, then $\eta f^*(\eta) g^*(f^*(\eta))=\eta f(\eta) g(f(\eta))=0$, because 
it is one of relations defining $\Lambda$. Assume that $g^*(f(\eta))\neq g(f(\eta))$. This is the case only when $g^*(f(\eta))=\tau_i$ for some 
$i\in\{1,\dots,r\}$, and hence $f(\eta)=g^{-1}(\alpha_i)=\sigma_i$ and $\eta=f^{-1}(\sigma_i)=f(\rho_i)$. But then we have 
$$\eta f^*(\eta) g^*(f^*(\eta))=\eta\sigma_i\tau_i=\eta\sigma_i\alpha_i\beta_i = \eta f(\eta) g(f(\eta))\beta_i=0,$$ 
because of the relations defining $\Lambda$ and restrictions imposed on $\eta$.

(4) Assume $\eta$ is an arrow in $Q(\xi)^*_1$ different from $\tau_i$ with $m_{\nu_i}n_{\nu_i}=3$, $(g^*)^{-1}(\nu_i)$, $\nu_i$, for 
$i\in\{1,\dots,r\}$, and $\eta\in Q_1$ such that $f(\eta)$ is virtual or $f^2(\eta)$ is virtual with $m_{f(\eta)}=1$ and $n_{f(\eta)}=3$. 
In particular, if $g^*(\eta)=g(\eta)$ and $f^*(g^*(\eta))=f(g(\eta))$, then $\eta g^*(\eta) f^*(g^*(\eta))=\eta g(\eta) f(g(\eta))=0$, by the 
restrictions imposed on $\eta$ and the zero relations defining $\Lambda$. 

If $g^*(\eta)\neq g(\eta)$, that is $g^*(\eta)=\tau_i$ for some $i\in\{1,\dots,r\}$, then 
$$\eta g^*(\eta) f^*(g^*(\eta))=\eta\tau_i\nu_i=\eta\alpha_i\beta_i\nu_i=\eta\alpha_i\xi_i=\eta g(\eta) f(g(\eta))=0.$$ 

Assume now that $\eta=\tau_i$ for some $i\in\{1,\dots,r\}$. Then $g^*(\eta)=\gamma_i$ or $g^*(\eta)=\tau_j$ for some 
$j\in\{1,\dots,r\}$. Suppose $g^*(\eta)=\gamma_i$. Then we have 
$$\eta g^*(\eta) f^*(g^*(\eta))=\tau_i\gamma_i f(\gamma_i)=\alpha_i\beta_i\gamma_i f(\gamma_i)=0,$$ 
because $\beta_i\gamma_i f(\gamma_i)=\beta_i g(\beta_i)f(g(\beta_i))=0$, due to assumptions imposed on $\tau_i$. If 
$g^*(\tau_i)=\tau_j$, then 
$$\eta g^*(\eta) f^*(g^*(\eta))=\tau_i\tau_j\nu_j=\alpha_i\beta_i\alpha_j\beta_j\nu_j=\alpha_i\beta_i\alpha_j\xi_j=
c_{\nu_i}\alpha_i\beta_i A_{\nu_i}=$$ 
$$=c_{\nu_i}\alpha_i\beta_i\nu_i\delta_i g(\delta_i)\dots \nu_j=
c_{\nu_i}\alpha_i\xi_i\delta_i g(\delta_i)\dots\nu_j=c_{\nu_i}c_{\alpha_i}B_{\alpha_i}g(\delta_i)\dots\nu_j=0,$$ 
because $B_{\alpha_i}$ belongs to $\soc\Lambda$ (Proposition \ref{prop:2.6}).

\end{proof}

\begin{theorem}\label{thm:5.5} The algebras $\Lambda(\xi)$ and $\Gamma(\xi)$ are isomorphic. \end{theorem}

\begin{proof} It follows from Theorem \ref{thm:5.2} that $\Gamma(\xi)$ is a finite-dimensional symmetric algebra. Moreover, 
by Lemma \ref{lem:5.3} and Proposition \ref{prop:5.4} we conclude that the Gabriel quivers $Q_{\Lambda(\xi)}$ and 
$Q_{\Gamma(\xi)}$ coincide, and the relations defining $\Lambda(\xi)$ are satisfied in $\Gamma(\xi)$. Hence $\Gamma(\xi)$ 
is an epimorphic image of the algebra $\Lambda(\xi)$. Finally, bacause $\Gamma(\xi)$ is symmetric, we may then conclude, as 
in Proposition \ref{prop:3.9}, that $\Gamma(\xi)$ has the same basis as $\Lambda(\xi)$, so $\dim_K\Lambda(\xi)=\dim_K\Gamma(\xi)$. 
Therefore, $\Lambda(\xi)$ and $\Gamma(\xi)$ are indeed isomorphic as $K$-algebras. 
\end{proof}

Concluding, the proof of Main Theorem is now complete. 

For each vertex $x\in Q(\xi)_0=Q_0$, we denote by $P^\xi_x$ the indecomposable projective module $e_x\Lambda(\xi)$ in 
$\mod \Lambda(\xi)$ at $x$. Consider the following complexes in the homotopy category $K^b(P_{\Lambda(\xi)})$ of projective 
modules in $\mod\Lambda(\xi)$: 
$$\hat{T}^\xi_x:\; \xymatrix{0\ar[r] & P^\xi_x\ar[r] & 0}$$ 
concentrated in degree $0$, for all vertices $x\neq c_1,\dots,c_r$, and 
$$\hat{T}^\xi_{c_i}:\; \xymatrix{0\ar[r] & P_{c_i}^\xi\ar[r]^{\beta_i}& P_{b_i}^\xi\ar[r] & 0}$$ 
concentrated in degrees $1$ and $0$, for all $i\in\{1,\dots,r\}$. We set 
$$\hat{T}^\xi:=\bigoplus_{x\in Q(\xi)_0}\hat{T}^\xi_x.$$ 
Moreover, let 
$$\hat{P}'=\bigoplus_{i=1}^r P^\xi_{c_i}\quad\mbox{and}\quad \hat{P}''=\bigoplus_{x\in Q(\xi)_0\setminus\{c_1,\dots,c_r\}}P^\xi_x.$$ 
Let $\beta:\bigoplus_{i=1}^r P^\xi_{c_i}\to\bigoplus_{i=1}^r P^\xi_{b_i}$ be the diagonal homomorphism given by the 
homomorphisms $\beta_i:P^\xi_{c_i}\to P^\xi_{b_i}$. Then $\beta$ is a left $\add(\hat{P}'')$-apprroximation of $\hat{P}'$ in 
$\mod \Lambda(\xi)$, since $\beta_i$ is the unique arrow in $Q(\xi)$ ending at $c_i$, for any $i\in\{1,\dots,r\}$. Hence, 
applying Proposition \ref{pro:1.5}, we obtain that $\hat{T}^\xi$ is a tilting complex in $K^b(P_{\Lambda(\xi)})$. We define 
$$\hat{\Lambda}(\xi)=\End_{K^b(P_{\Lambda(\xi)})}(\hat{T}^\xi).$$ 

\begin{theorem}\label{thm:5.6} The algebras $\Lambda$ and $\hat{\Lambda}(\xi)$ are isomorphic. \end{theorem} 

\begin{proof} We observe that $\hat{P}_x=\Hom_{K^b(P_{\Lambda(\xi)})}(\hat{T}^\xi,\hat{T}^\xi_x)$, $x\in Q_0$, form a 
complete family of pairwise non-isomorphic indecomposable projective modules in $\mod\hat{\Lambda}(\xi)$. We define the 
following morphisms between indecomposable summands of $\hat{T}^\xi$ in $K^b(P_{\Lambda(\xi)})$: 
$$\hat{\alpha}_i:\hat{T}^\xi_{c_i}\to\hat{T}^\xi_{a_i},\quad\mbox{given by }\tau_i:P^\xi_{b_i}\to P^\xi_{a_i},$$
$$\hat{\beta}_i:\hat{T}^\xi_{b_i}\to\hat{T}^\xi_{c_i},\quad\mbox{given by }id:P^\xi_{b_i}\to P^\xi_{b_i},$$ 
$$\hat{\epsilon}_i:\hat{T}^\xi_{a_i}\to\hat{T}^\xi_{b_i},\quad\mbox{given by }\beta_i\alpha_i:P^\xi_{a_i}\to P^\xi_{b_i},$$ 
for all $i\in\{1,\dots,r\}$, and 
$$\hat{\eta}:\hat{T}^\xi_{t(\eta)}\to\hat{T}^\xi_{s(\eta)},\quad\mbox{given by }\eta:P^\xi_{t(\eta)}\to P^\xi_{s(\eta)},$$ 
for any arrow $\eta\in Q(\xi)_1$ different from the arrows $\alpha_i,\beta_i$, $i\in\{1,\dots,r\}$. We obtain then the 
homomorphisms between indecomposable projective modules in $\mod\hat{\Lambda}(\xi)$: 
$$\alpha_i=\Hom_{K^b(P_{\Lambda(\xi)})}(\hat{T}^\xi,\hat{\alpha}_i):\hat{P}_{c_i}\to\hat{P}_{a_i},$$ 
$$\beta_i=\Hom_{K^b(P_{\Lambda(\xi)})}(\hat{T}^\xi,\hat{\beta}_i):\hat{P}_{b_i}\to\hat{P}_{c_i},$$ 
$$\epsilon_i=\Hom_{K^b(P_{\Lambda(\xi)})}(\hat{T}^\xi,\hat{\epsilon}_i):\hat{P}_{a_i}\to\hat{P}_{b_i},$$ 
for all $i\in\{1,\dots,r\}$, and 
$$\eta=\Hom_{K^b(P_{\Lambda(\xi)})}(\hat{T}^\xi,\hat{\eta}_i):\hat{P}_{t(\eta)}\to\hat{P}_{s(\eta)},$$ 
for any arrow $\eta\in Q(\xi)_1$ different from $\alpha_i,\beta_i$, $i\in\{1,\dots,r\}$. Recall that we have in $\Lambda(\xi)$ 
the relations: 
$$\nu_i\delta_i=\beta_i\alpha_i+c_{\tilde{\nu}_i}A^*_{\tilde{\nu}_i}, \alpha_i\tau_i=0,\tau_i\beta_i=0,$$
$$\tau_i\nu_i=c_{\tilde{\tau}_i}A^*_{\tilde{\tau}_i}, \delta_i\tau_i=c_{\delta_i}A^*_{\delta_i},$$ 
for all $i\in\{1,\dots,r\}$. Moreover, we have in $K^b(P_{\Lambda(\xi)})$ the equalities 
$\hat{\tau_i}=\hat{\alpha_i}\hat{\beta_i}$, $\hat{\epsilon}_i\hat{\alpha}_i=0$ and $\hat{\beta}_i\hat{\epsilon}_i=0$, 
for any $i\in\{1,\dots,r\}$. Note that $\epsilon_i$ is not irreducible, so there is no arrow in $Q_{\hat{\Lambda}(\xi)}$ 
corresponding to $\epsilon_i$, for any $i\in\{1,\dots,r\}$.  

Consider the quiver $\hat{Q}(\xi)=(\hat{Q}(\xi)_0,\hat{Q}(\xi)_1,s,t)$ defined as follows. We take $\hat{Q}(\xi)_0= Q(\xi)_0=Q_0$, 
and the set $\hat{Q}(\xi)_1$ of arrows is obtained from the set $Q(\xi)_1$ of arrows of $Q(\xi)$ by the following operations: 
\begin{itemize}
\item replacing the arrows $\xymatrix{c_i\ar[r]^{\alpha_i}&a_i}$ and $\xymatrix{b_i\ar[r]^{\beta_i}&c_i}$ by the arrows 
$\xymatrix{a_i\ar[r]^{\alpha_i}&c_i}$ and $\xymatrix{c_i\ar[r]^{\beta_i}&b_i}$, 

\item removing the arrows $\xymatrix{a_i\ar[r]^{\tau_i}&b_i}$,  
\end{itemize}
for all $i\in\{1,\dots,r\}$. 

We observe that the quiver $\hat{Q}(\xi)$ is the quiver obtained from the triangulation quiver $(Q,f)$, defining the algebra 
$\Lambda=\Lambda(Q,f,m_\bullet,c_\bullet)$, by removing the virtual arrows $\xi_i$ and $\mu_i$, for $i\in\{1,\dots,r\}$. 
Moreover, it follows from the above relations that the radical of $\hat{\Lambda}(\xi)$ is generated by the arrows of $\hat{Q}(\xi)$. 
We also mention that for any arrow $\eta\in \hat{Q}(\xi)_1$ different from $\alpha_i,\beta_i$, $i\in\{1,\dots,r\}$, we have 
in $\hat{Q}(\xi)$ the path $\hat{A}_\eta$ obtained from the path $A^*_\eta$ in $Q(\xi)^*$ by replacing every arrow $\tau_i$ by the path $\alpha_i\beta_i$, and hence $\hat{A}_\eta$ coincides with the path $A_\eta$ in $Q$. Clearly, we have also in $\hat{Q}(\xi)$ the paths 
$\hat{A}_{\alpha_i}=A_{\alpha_i}$ and $\hat{A}_{\beta_i}=A_{\beta_i}$, for $i\in\{1,\dots,r\}$. 

Finally, note that the following relations in $\hat{\Lambda}(\xi)$ are consequences of the relations in $\Lambda(\xi)$ presented 
above: 
$$\nu_i\delta_i\alpha_i=\epsilon_i\alpha_i+c_{\tilde{\nu_i}}A^*_{\tilde{\nu_i}}\alpha_i=c_{\tilde{\nu_i}}\hat{A}_{\tilde{\nu_i}},$$ 
$$\beta_i\nu_i\delta_i=\beta_i\epsilon_i+c_{\tilde{\nu_i}}\beta_i A^*_{\tilde{\nu_i}}=
c_{\tilde{\nu_i}}\hat{A}_{\beta_i}=c_{\beta_i}\hat{A}_{\beta_i},$$
$$\delta_i\alpha_i\beta_i=\delta_i\tau_i=c_{\tilde{\delta_i}}A^*_{\tilde{\delta_i}}=c_{\delta_i}\hat{A}_{\delta_i},$$
$$\alpha_i\beta_i\nu_i=\tau_i\nu_i=c_{\tilde{\tau_i}}A^*_{\tilde{\tau_i}}=c_{\bar{\alpha}_i}\hat{A}_{\bar{\alpha}_i}.$$ 

Then we may conclude that the algebra $\hat{\Lambda}(\xi)$ is given by the quiver $\hat{Q}(\xi)$ and the same relations as the 
algebra $\Lambda$. Therefore, the algebras $\hat{\Lambda}(\xi)$ and $\Lambda$ are isomorphic. \end{proof} 

In the above notation we have $\mathcal{O}_i=(\xi_i\mbox{ }\mu_i)$, for any $i\in\{1,\dots,r\}$, and $\xi=(\xi_1,\dots,\xi_r)\in
\mathcal{O}_1\times\dots\times\mathcal{O}_r$ is a sequence of virtual arrows with $s(\xi_i)=c_i$ and $t(\xi_i)=d_i$. We may also consider 
the associated sequence $\mu=(\mu_1,\dots,\mu_r)\in\mathcal{O}_1\times\dots\times\mathcal{O}_r$ of arrows in the opposite direction, that is, 
$s(\mu_i)=d_i$ and $t(\mu_i)=c_i$, for $i\in\{1,\dots,r\}$. Then we have the following complexes in the homotopy category $K^b(P_\Lambda)$: 
$$T^{(\xi,\mu)}_x: 0 \to P_x \to 0,$$ 
concentrated in degree $0$, for all vertices $x\in Q_0$ different from $c_i,d_i$, $i\in\{1,\dots,r\}$, and 
$$T^{(\xi,\mu)}_{c_i}:\xymatrix{0 \ar[r] & P_{c_i}\ar[r]^{\alpha_i} & P_{a_i}\ar[r]&0,}$$
$$T^{(\xi,\mu)}_{d_i}:\xymatrix{0 \ar[r] & P_{d_i}\ar[r]^{\nu_i} & P_{b_i}\ar[r]&0,}$$ 
concentrated in degrees $1$ and $0$, for all $i\in\{1,\dots,r\}$. Applying arguments as in the proof of Lemma \ref{lem:5.1}, we obtain the 
following fact. 

\begin{lem}\label{lem:5.7} The complex $T^{(\xi,\mu)}:=\bigoplus_{x\in Q_0}T^{(\xi,\mu)}_x$ is a tilting complex in $K^b(P_\Lambda)$. \end{lem} 

We define the algebra $\Lambda(\xi,\mu):=\End_{K^b(P_\Lambda)}(T^{(\xi,\mu)})$. 

\begin{theorem}\label{thm:5.8} The algebras $\Lambda$ and $\Lambda(\xi,\mu)$ are isomorphic. \end{theorem} 

\begin{proof} Note that the modules $\tilde{P}_x=\Hom_{K^b(P_\Lambda)}(T^{(\xi,\mu)},T^{(\xi,\mu)}_x)$, $x\in Q_0$, form a complete family of 
pairwise non-isomorphic indecomposable projective modules in $\mod\Lambda(\xi,\mu)$. Observe also that we have the following morphisms in 
$K^b(P_\Lambda)$ between indecomposable direct summands of $T^{(\xi,\mu)}$: 
$$\hat{\alpha}_i:T^{(\xi,\mu)}_{a_i}\to T^{(\xi,\mu)}_{c_i},\mbox{ given by }id:P_{a_i}\to P_{a_i},$$
$$\hat{\beta}_i:T^{(\xi,\mu)}_{c_i}\to T^{(\xi,\mu)}_{b_i},\mbox{ given by }\nu_i\delta_i-c_{\gamma_i}A_{\gamma_i}':P_{a_i}\to P_{b_i},$$
$$\hat{\tau}_i:T^{(\xi,\mu)}_{b_i}\to T^{(\xi,\mu)}_{a_i},\mbox{ given by }\alpha_i\beta_i:P_{b_i}\to P_{a_i},$$
$$\hat{\nu}_i:T^{(\xi,\mu)}_{b_i}\to T^{(\xi,\mu)}_{d_i},\mbox{ given by }id:P_{b_i}\to P_{b_i},$$
$$\hat{\delta}_i:T^{(\xi,\mu)}_{d_i}\to T^{(\xi,\mu)}_{a_i},\mbox{ given by }\alpha_i\beta_i-c_{\rho_i}A'_{\rho_i}:P_{b_i}\to P_{a_i},$$
$$\hat{\varepsilon}_i:T^{(\xi,\mu)}_{a_i}\to T^{(\xi,\mu)}_{b_i},\mbox{ given by }\nu_i\delta_i:P_{a_i}\to P_{b_i},$$
for all $i\in\{1,\dots,r\}$, and 
$$\hat{\eta}:T^{(\xi,\mu)}_{t(\eta)}\to T^{(\xi,\mu)}_{s(\eta)},\mbox{ given by }\eta:P_{t(\eta)}\to P_{s(\eta)},$$ 
for any arrow $\eta\in Q_1$ different from the arrows $\alpha_i,\beta_i,\nu_i,\delta_i$, $i\in\{1,\dots,r\}$. Applying the covariant functor 
$\Hom_{K^b(P_\Lambda)}(T^{(\xi,\mu)},-)$ we obtain the following homomorphisms between indecomposable projective modules in $\mod\Lambda(\xi,\mu)$: 
$$\alpha_i=\Hom_{K^b(P_\Lambda)}(T^{(\xi,\mu)},\hat{\alpha}_i):\tilde{P}_{a_i}\to\tilde{P}_{c_i},$$
$$\beta_i=\Hom_{K^b(P_\Lambda)}(T^{(\xi,\mu)},\hat{\beta}_i):\tilde{P}_{c_i}\to\tilde{P}_{b_i},$$
$$\tau_i=\Hom_{K^b(P_\Lambda)}(T^{(\xi,\mu)},\hat{\tau}_i):\tilde{P}_{b_i}\to\tilde{P}_{a_i},$$
$$\nu_i=\Hom_{K^b(P_\Lambda)}(T^{(\xi,\mu)},\hat{\nu}_i):\tilde{P}_{b_i}\to\tilde{P}_{d_i},$$
$$\delta_i=\Hom_{K^b(P_\Lambda)}(T^{(\xi,\mu)},\hat{\delta}_i):\tilde{P}_{d_i}\to\tilde{P}_{a_i},$$
$$\varepsilon_i=\Hom_{K^b(P_\Lambda)}(T^{(\xi,\mu)},\hat{\varepsilon}_i):\tilde{P}_{a_i}\to\tilde{P}_{b_i},$$
for all $i\in\{1,\dots,r\}$, and 
$$\eta=\Hom_{K^b(P_\Lambda)}(T^{(\xi,\mu)},\hat{\eta}):\tilde{P}_{t(\eta)}\to\tilde{P}_{s(\eta)},$$ 
for all arrows $\eta\in Q_1$ different from $\alpha_i,\beta_i,\nu_i,\delta_i$, $i\in\{1,\dots,r\}$. It is not difficult to observe that the 
homomorphisms $\alpha_i,\beta_i,\tau_i,\nu_i,\delta_i,\epsilon_i$ and $\eta\in Q_1$ different than $\alpha_i,\beta_i,\nu_i,\delta_i$, 
$i\in\{1,\dots,r\}$, generate the radical $\rad\Lambda(\xi,\mu)$ of $\Lambda(\xi,\mu)$ (after obvious identification with the corresponding 
elements of the algebra). 

For each $i\in\{1,\dots,r\}$, we have the following relations 
$$\varepsilon_i=\beta_i\alpha_i+c_{\gamma_i}A^{**}_{\gamma_i}, \alpha_i\tau_i=0, \tau_i\beta_i=0,$$ 
$$\tau_i=\delta_i\nu_i+c_{\rho_i}A^{**}_{\rho_i}, \nu_i\varepsilon_i=0, \varepsilon_i\delta_i=0,$$ 
where the paths $A_{\gamma_i}^{**}$ and $A_{\rho_i}^{**}$ are obtained from the paths $A_{\gamma_i}'$ and $A_{\rho_i}'$ by replacing all subpaths 
$\alpha_j\beta_j$ by $\tau_j$, and all subpaths $\nu_j\delta_j$ by $\varepsilon_j$, $j\in\{1,\dots,r\}$. In particular, we conclude that the 
generators $\tau_i,\varepsilon_i$, $i\in\{1,\dots,r\}$, are superfluous, which means that the corresponding arrows do not occur in the Gabriel 
quiver $Q_{\Lambda(\xi,\mu)}$ of $\Lambda(\xi,\mu)$. Hence $Q_{\Lambda(\xi,\mu)}$ contains subquivers 
$$
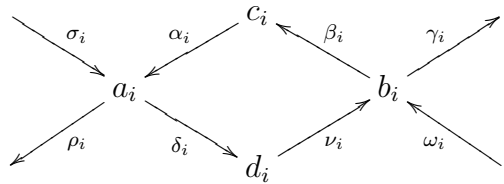
\xymatrix@R=0.5cm@C=1.2cm{\ar[rd]^{\sigma_i}&&c_i\ar[ld]_{\alpha_i}&& \\ 
&a_i\ar[ld]^{\rho_i}\ar[rd]_{\delta_i}&&b_i\ar[lu]_{\beta_i}\ar[ru]^{\gamma_i} 
& \\ &&d_i\ar[ru]_{\nu_i}&&\ar[lu]^{\omega_i} }$$ 
for $i\in\{1,\dots,r\}$, where possibly one of the equalities $\gamma_i=\sigma_i$ or $\rho_i=\omega_i$ holds. The above relations provide the 
following four relations in $\Lambda(\xi,\mu)$: 
$$(1)\qquad\qquad\nu_i\beta_i\alpha_i+c_{\gamma_i}\nu_i A^{**}_{\gamma_i}=0,\quad \beta_i\alpha_i\delta_i+c_{\gamma_i}A^{**}_{\gamma_i}\delta_i=0,$$ 
$$\qquad\qquad\quad\alpha_i\delta_i\nu_i+c_{\rho_i}\alpha_i A^{**}_{\rho_i}=0,\quad \delta_i\nu_i\beta_i+c_{\rho_i}A^{**}_{\rho_i}\beta_i=0.$$ 

Now, let $(Q',f')$ be the triangulation quiver obtained from the triangulation quiver $(Q,f)$ by replacing all subquivers 
$$
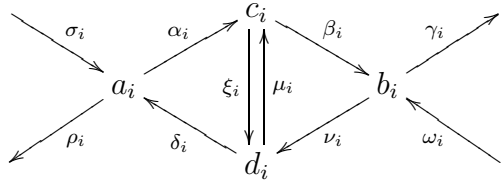
\xymatrix@R=0.5cm@C=1.2cm{\ar[rd]^{\sigma_i}&&c_i\ar[rd]^{\beta_i}\ar@<-0.1cm>[dd]_{\xi_i}&& \\ 
&a_i\ar[ld]^{\rho_i}\ar[ru]^{\alpha_i}&&b_i\ar[ld]^{\nu_i}\ar[ru]^{\gamma_i} 
& \\ &&d_i\ar[lu]^{\delta_i}\ar@<-0.1cm>[uu]_{\mu_i}&&\ar[lu]^{\omega_i} }$$ 
by the subquivers 
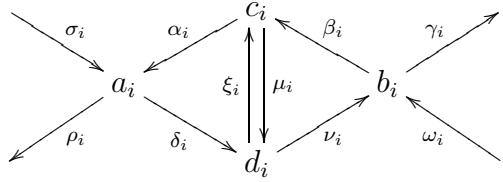
$$\xymatrix@R=0.5cm@C=1.2cm{\ar[rd]^{\sigma_i}&&c_i\ar[ld]_{\alpha_i}\ar@<0.1cm>[dd]^{\mu_i}&& \\ 
&a_i\ar[ld]^{\rho_i}\ar[rd]_{\delta_i}&&b_i\ar[lu]_{\beta_i}\ar[ru]^{\gamma_i} 
& \\ &&d_i\ar[ru]_{\nu_i}\ar@<0.1cm>[uu]^{\xi_i}&&\ar[lu]^{\omega_i} }$$ 
for $i\in\{1,\dots,r\}$, with the $f'$-orbits $(\alpha_i\mbox{ }\delta_i\mbox{ }\xi_i)$ and $(\beta_i\mbox{ }\mu_i\mbox{ }\nu_i)$, 
and $f'(\eta)=f(\eta)$ for any arrow $\eta\in Q_1$ different from $\alpha_i,\beta_i,\nu_i,\delta_i,\xi_i,\mu_i$. Let $g'$ be the associated 
permutation $g':=\overline{f'}$ of $Q_1'$. Then we write $\mathcal{O}'(\eta)$ for the $g'$-orbit of an arrow $\eta$ in $Q_1'$, and 
$\mathcal{O}(g')$ for the set of all $g'$-orbits in $Q_1'$. It is clear from definition that 
$$\mathcal{O}'(\alpha_i)=(\alpha_i\mbox{ }\rho_i\mbox{ }\dots\mbox{ }\omega_i\mbox{ }\beta_i)\mbox{ and }
\mathcal{O}'(\nu_i)=(\nu_i\mbox{ }\gamma_i\mbox{ }\dots\mbox{ }\sigma_i\mbox{ }\delta_i),$$ 
and obviously $\mathcal{O}'(\xi_i)=(\xi_i\mbox{ }\mu_i)$, for $i\in\{1,\dots,r\}$. We may also define the weight and parameter functions 
$m'_\bullet:\mathcal{O}(g')\to\mathbb{N}^*$ and $c'_\bullet:\mathcal{O}(g')\to K^*$ of $(Q',f')$ setting as follows: 
$$m'_{\mathcal{O}'(\alpha_i)}=m_{\mathcal{O}(\nu_i)}, m'_{\mathcal{O}'(\nu_i)}=m_{\mathcal{O}(\alpha_i)}, m'_{\mathcal{O}'(\xi_i)}=1,$$ 
$$c'_{\mathcal{O}'(\alpha_i)}=c_{\mathcal{O}(\nu_i)}, c'_{\mathcal{O}'(\nu_i)}=c_{\mathcal{O}(\alpha_i)}, c'_{\mathcal{O}'(\xi_i)}=-1,$$ 
for $i\in\{1,\dots,r\}$, and $m'_{\mathcal{O}'(\eta)}=m_{\mathcal{O}(\eta)}$ and $c'_{\mathcal{O}'(\eta)}=c_{\mathcal{O}(\eta)}$ for the 
remaining $g'$-orbits $\mathcal{O}'(\eta)$ in $Q_1'$ - being then the $g$-orbits $\mathcal{O}'(\eta)=\mathcal{O}(\eta)$ in $Q_1$. 

We claim that the algebra $\Lambda(\xi,\mu)$ is isomorphic to the weighted triangulation algebra $\Lambda':=\Lambda(Q',f',m'_\bullet,c'_\bullet)$. 
Observe first that the Gabriel quivers $Q_{\Lambda(\xi,\mu)}$ and $Q_{\Lambda'}$ coincide. Let $\hat{A}_\eta$, for an arrow $\eta\in Q_1'$, denote 
the path $\hat{A}_\eta=(\eta g'(\eta)\dots (g')^{n_{\mathcal{O}'(\eta)}-1}(\eta))^{m_{\mathcal{O}'(\eta)}-1}\eta g'(\eta)\dots (g')^{n_{\mathcal{O}'(\eta)}-2}(\eta)$ 
in $Q'$ along the $g'$-orbit of $\eta$. By definition, for any $i\in\{1,\dots,r\}$, we have the equalities in $\Lambda'$: 
$$\nu_i\beta_i\alpha_i=(-\xi_i)\alpha_i=-c'_{\nu_i}\hat{A}_{\nu_i},\quad \beta_i\alpha_i\delta_i=\beta_i(-\mu_i)=-c'_{\gamma_i}\hat{A}_{\gamma_i},$$  
$$ \alpha_i\delta_i\nu_i=(-\mu_i)\nu_i=-c'_{\alpha_i}\hat{A}_{\alpha_i},\quad \delta_i\nu_i\beta_i=\delta_i(-\xi_i)=-c'_{\rho_i}\hat{A}_{\rho_i},$$ 
so the following four relations hold (in $\Lambda '$): 
$$(2)\qquad\qquad\nu_i\beta_i\alpha_i+c'_{\nu_i}\hat{A}_{\nu_i}=0,\quad \beta_i\alpha_i\delta_i+c'_{\gamma_i}\hat{A}_{\gamma_i}=0,$$  
$$\qquad\qquad\quad\mbox{ }\alpha_i\delta_i\nu_i+c'_{\alpha_i}\hat{A}_{\alpha_i}=0,\quad \delta_i\nu_i\beta_i+c'_{\rho_i}\hat{A}_{\rho_i}=0.$$  

We will show first that the above relations also hold in $\Lambda(\xi,\mu)$. According to the relations presented before (1), it is sufficient 
to prove that the following four equalities: $\hat{A}_{\nu_i}=\nu_i A^{**}_{\gamma_i}$, $\hat{A}_{\gamma_i}=A^{**}_{\gamma_i}\delta_i$, 
$\hat{A}_{\alpha_i}=\alpha_i A^{**}_{\gamma_i}$ and $\hat{A}_{\rho_i}=A^{**}_{\rho_i}\beta_i$ hold in $\Lambda(\xi,\mu)$, for $i\in\{1,\dots,r\}$ 
(the equalities of scalars are obvious from the definition of parameter function $c'_\bullet$). We only prove that $\hat{A}_{\nu_i}=\nu_i A_{\gamma_i}^{**},$ 
since the proof of the remaining three equalities is similar. 

First, observe that $\hat{A}_{\nu_i}$, regarded as a path in $Q_{\Lambda(\xi,\mu)}$, is given by a morphism $T^{(\xi,\mu)}_{a_i}\to T^{(\xi,\mu)}_{d_i}$ 
induced by a homomorphism $h:P_{a_i}\to P_{b_i}$ in $\mod\Lambda$, which is identified with the element in $\Lambda$ obtained from $A_{\gamma_i}'$ by 
exchanging each subpath $\alpha_j\beta_j$, $j\in\{1,\dots,r\}$, by $\alpha_j\beta_j-c_{\rho_j}A'_{\rho_j}$. But every path of the form $\rho= 
\gamma_i g(\gamma_i)\dots \sigma_j A'_{\rho_j}\gamma_j g(\gamma_j)$ is zero in $\Lambda$. Indeed, suppose first that $\rho_j=\omega_j$, or equivalently 
$n_{\nu_j}=3$. Then $A'_{\rho_j}=(\rho_j\nu_j\delta_j)^{m_{\nu_j}-1}\rho_j$ and $f(\rho_j)=f(\omega_j)=\gamma_j$, and hence $\rho$ admits either a 
proper subpath of the form $\sigma_j\rho_j\gamma_j=\sigma_j f(\sigma_j) f^2(\sigma_j)$, for $m_{\nu_j}=1$, or a (proper) subpath $\delta_j\rho_j\gamma_j= 
\delta_j g(\delta_j) f(g(\delta_j))$, otherwise. In the first case, the indicated subpath belongs to the (right) socle $\soc\Lambda$ of $\Lambda$ 
(see Proposition \ref{prop:2.6}(3)) and it is a proper subpath of $\rho$, so $\rho=0$ in $\Lambda$. In the second case, the indicated subpath is 
zero in $\Lambda$, due to relations of type (3) in the definition of weighted triangulation algebra (see Definition \ref{def:2.3}), because 
$f(\delta_j)=\alpha_j$ is not virtual and $n_{\alpha_j}\geqslant 4$ (note that we have $n_{\nu_j}=3$). Now, suppose that $n_{\nu_j}\geqslant 4$. 
If additionally $n_{\alpha_j}=3$, then $\gamma_j=\sigma_j$, $f^{2}(\sigma_j)=\omega_j$ is not virtual and $f(\bar{\sigma_j})=f(g(\omega_j))=f(\nu_j)=\mu_j$ 
is virtual but with $n_{\bar{\sigma_j}}=n_{g(\omega_j)}=n_{\nu_j}\geqslant 4$, and consequently, we have $\rho=0$ in $\Lambda$, since $\rho$ has 
a subpath $\sigma_j \rho_j g(\rho_j)=\sigma_j f(\sigma_j) g(f(\sigma_j))$ which equals zero in $\Lambda$, by relations of type (2) in Definition \ref{def:2.3}. 
Eventually, let $n_{\alpha_j}\geqslant 4$. Then $\rho$ admits subpaths of the forms: $\omega_j\gamma_j g(\gamma_j)=\omega_j f(\omega_j)g(f(\omega_j))$ 
and $\rho '= \omega'_j g(\omega '_j) f(g(\omega '_j)=\omega '_j \omega_j\gamma_j$, where $\omega_j'=g^{-1}(\omega_j)$. Here we have either 
$\rho '=0$ in $\Lambda$, by Definition \ref{def:2.3}(3), or the arrow $\omega_j'$ satisfies one of the following conditions: $f(\omega_j')$ is virtual 
or $f^2(\omega_j')$ is virtual with $m_{f(\omega_j')}=1$ and $n_{f(\omega_j')}=3$. Clearly $f(\omega_j')=f(g^{-1}(\omega_j))=\bar{\omega_j}=g(f^2(\omega_j))$ 
and $f^2(\omega_j')=f(\bar{\omega_j})$. In the case, when $f(\omega_j')=\bar{\omega_j}$ is a virtual arrow, we have 
$\omega_j f(\omega_j) g(f(\omega_j))=c_{\bar{\omega_j}} \bar{\omega_j} f(\bar{\omega_j})=c_{\bar{\omega_j}}c_{\omega_j}A_{\omega_j}$, and hence 
$\rho$ admits a proper subpath of the form $\rho_j\dots g^{-1}(\omega_j) A_{\omega_j}=\rho_j\dots g^{n_{\rho_j}-2}(\omega_j)B_{g^{-1}(\omega_j)}$, 
which is an element of the (left) socle of $\Lambda$, so $\rho=0$ in $\Lambda$. It remains to see that, if $f^2(\omega_j')=f(\omega_j)$ is virtual 
with $m_{f(\omega_j')}=1$ and $n_{f(\omega_j')}=3$, then $m_{\bar{\omega_j}}=1$ and $n_{\bar{\omega_j}}=3$, so $\rho$ admits a subpath of the form 
$$\gamma_j g(\gamma_j)\omega_j' \omega_j \gamma_j=\gamma_j g(\gamma_j) \omega_j' \omega_j f(\omega_j)=c_{\bar{\omega_j}} 
\gamma_j g(\gamma_j) \omega_j' A_{\bar{\omega_j}}= $$ 
$$c_{\bar{\omega_j}} \gamma_j g(\gamma_j) \omega_j' \bar{\omega_j} g(\bar{\omega_j})=c_{\bar{\omega_j}}c_{\bar{\omega_j}'} \gamma_j 
g(\gamma_j)\bar{\omega_j}' g(\bar{\omega_j}),$$ 
being zero in $\Lambda$, because $g(\gamma_j)\bar{\omega_j}' g(\bar{\omega_j})=g(\gamma_j) f(g(\gamma_j)) f^2(g(\gamma_j))$ belongs to the 
socle of $\Lambda$ (see Proposition \ref{prop:2.6}). This proves that indeed every path of the form $\rho$ is zero in $\Lambda$, so $h$ may be 
identified (in $\Lambda$) with $A_{\gamma_i}'$. In turn, by definition of morphisms $\tau_i$, we conclude that $\nu_i A^{**}_{\gamma_i}$ 
may be also identified with $A_{\gamma_i}'$, and therefore, we have $\hat{A}_{\nu_i}=\nu_i A^{**}_{\gamma_i}$ in $\Lambda(\xi,\mu)$. 
Other equalities are proven in a similar way. As a result, we infer that all relations of the above form (2) are satisfied in $\Lambda(\xi,\mu)$. 

All other relations in $\Lambda '$ of type $\eta f'(\eta)=c_{\bar{\eta}}\hat{A}_{\bar{\eta}}$ also hold in $\Lambda(\xi,\mu)$, which can be easy 
deduced from the definition of morphisms corresponding to arrows of $Q_{\Lambda(\xi,\mu)}$. In particular, using description of arrows and 
the above presented arguments one may also check that the remaining relations satisfied in $\Lambda '$ (see Definition \ref{def:2.3}(2) and (3)) 
hold in $\Lambda(\xi,\mu)$. Summing up, we have shown that $\Lambda(\xi,\mu)$ is isomorphic to $\Lambda '$. 

In the final step, we observe that in fact $\Lambda'$ is isomorphic to $\Lambda$. Namely, there is an isomorphism of $K$-algebras $\Psi:\Lambda '\to \Lambda$ 
induced by the isomorphism of path algebras $\Phi:KQ'\to KQ$ such that 
$$\Phi(e_{c_i})=e_{d_i}, \Phi(e_{d_i})=e_{c_i},\Phi(\alpha_i)=\delta_i, \Phi(\beta_i)=\nu_i,$$ 
$$\Phi(\nu_i)=\beta_i,\Phi(\delta_i)=\alpha_i, \Phi(\xi_i)=-\xi_i,\Phi(\mu_i)=-\mu_i,$$ 
$$\Phi(e_x)=e_x\mbox{ for all vertices }x\in Q_0'\mbox{ different from }c_i, d_i,$$ 
$$\Phi(\eta)=\eta\mbox{ for all arrows }\alpha\in Q_1'\mbox{ different from }\alpha_i,\beta_i,\nu_i,\delta_i,$$ 
for $i\in\{1,\dots,r\}$. \end{proof} 

\begin{rem} The above theorem explains why in Definition \ref{def:3.4} we consider only sequences $\xi=(\xi_1,\dots,\xi_r)$ of virtual arrows comming 
from pairwise different $g$-orbits (of length $2$). In fact, the Definition \ref{def:3.4} makes sense even if $\xi$ does not satisfy this assumption. 
But, using construction analoguous to the one described above (Lemma \ref{lem:5.7} and Theorem \ref{thm:5.8}) we can prove that, if $\xi$ is a sequence 
of virtual arrows (from orbits of length $2$) and $\xi_j=g(\xi_i)$ for some $i\neq j$, then the virtual mutation $\Lambda(\xi)$ (in a wider sense) is 
isomorphic to the virtual mutation $\Lambda(\hat{\xi})$, where $\hat{\xi}$ is a sequence obtained from $\xi$ by cancelling the arrows $\xi_i$ and $\xi_j$. 
The situation in which $\xi_j=\xi_i$, for some $i\neq j$, is covered by the Theorem \ref{thm:5.6}. \end{rem}

\section{Virtual edge deformations of weighted \\ surface algebras}\label{sec:6}

In this section, by a surface we mean a connected, compact, two-dimensional orientable real manifold, with or without 
boundary. It is well known that every surface $S$ admits an additional structure of a finite two-dimensional triangular 
cell complex and hence a triangulation, by the deep Triangulation Theorem (see for example \cite[Section 2.3]{Ca}). 

For a positive natural number $n$, we denote by $D^n$ the unit disc in the $n$-dimensional Euclidean space $\mathbb{R}^n$, 
formed by all points of distance $\leqslant 1$ from the origin. Then the boundary $\partial D^n$ of $D^n$ is the unit sphere 
$S^{n-1}$ in $\mathbb{R}^n$ formed by all points of distance $1$ from the origin. Furhter, by an $n$-cell we mean a 
topological space homeomorphic to the open disc $O^n=int(D^n)=D^n\setminus\partial D^n$. In particular, $S^0=\partial D^1$ 
consists of two points and $\partial D^0=D^0$ is a single point. We refer to \cite[Appendix]{Ha2} for some basic topological 
facts about cell complexes. 

Let $S$ be a surface. By a finite two-dimensional triangular cell complex structure on $S$ we mean a finite family of maps 
$\phi_i^n:D^n_i\to S$, with $n\in\{0,1,2\}$ and $D^n_i=D^n$, satisfying the following conditions:
\begin{enumerate}[(1)] 
\item each $\phi^n_i$ restricts to a homeomorphism $O^n\to\phi^n_i(O^n)$ and the $n$-cells $e^n_i:=\phi^n_i(O^n)$ of $S$ 
are pairwise disjoint and their union is $S$. 

\item For each two-dimensional cell $e^2_i$, $\phi^2_i(\partial D_i^2)$ is the union of $k$ $1$-cells and $m$ $0$-cells, where $k\in 
\{2,3\}$ and $m\in\{1,2,3\}$, and different from $\xymatrix{\bullet\ar@{-}@<0.1cm>[r]&\bullet\ar@{-}@<0.1cm>[l]}$. 
\end{enumerate}
Then the closures $\phi^2_i(D^2_i)$ of all $2$-cells $e_i^2$ are called \emph{triangles} of $S$, and the closures $\phi^1_i(D^1_i)$ 
of all $1$-cells $e^1_i$ are called \emph{edges} of $S$. The collection of all triangles is said to be a \emph{triangulation} 
of $S$. We assume that such a triangulation $T$ of $S$ has at least two different edges, so then $T$ is a finite collection 
of triangles of the form 
$$
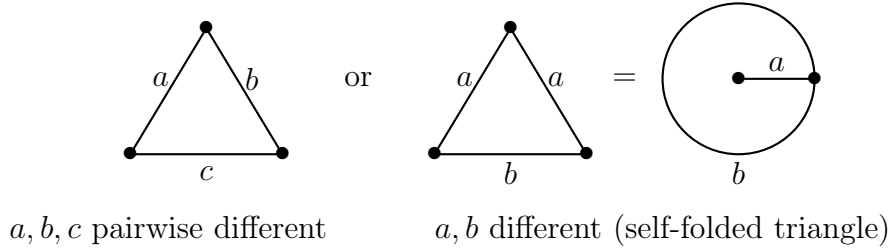
\begin{tikzpicture}
\draw[thick](-1,0)--(1,0);
\draw[thick](-1,0)--(0,1.67);
\draw[thick](1,0)--(0,1.67);
\node() at (-1,0){$\bullet$};
\node() at (1,0){$\bullet$};
\node() at (0,1.67){$\bullet$};
\node() at (-0.6,1){$a$};
\node() at (0.6,1){$a$};
\node() at (0,-0.25){$b$};

\draw[thick](-5,0)--(-3,0);
\draw[thick](-5,0)--(-4,1.67);
\draw[thick](-3,0)--(-4,1.67);
\node() at (-5,0){$\bullet$};
\node() at (-3,0){$\bullet$};
\node() at (-4,1.67){$\bullet$};
\node() at (-4.6,1){$a$};
\node() at (-3.4,1){$b$};
\node() at (-4,-0.25){$c$};

\draw[thick](3,1) circle (1);
\node() at (3,1){$\bullet$};
\node() at (4,1){$\bullet$};
\draw[thick](3,1)--(4,1);
\node() at (3.5,1.2){$a$};
\node() at (3,-0.25){$b$};

\node() at (1.5,1){$=$};
\node() at (-2,1){or};
\node() at (-4.5,-1){$a,b,c$ pairwise different};
\node() at (2,-1){$a,b$ different (self-folded triangle)};

\end{tikzpicture}$$ 
such that every edge is either the edge of exactly two triangles, is the self-folded edge, or lies on the boundary of $S$. We 
note that a given surface $S$ admits many finite two-dimensional triangular cell complex structures, and hence triangulations. 

In this section, by a directed triangulated surface we mean a pair $(S,\overrightarrow{T})$, where $S$ is a surface, $T$ a 
triangulation of $S$, and $\overrightarrow{T}$ a coherent orientation of triangles in $T$. We may assume that all triangles 
of $T$ have clockwise orientation in $\overrightarrow{T}$. Then the associated triangulation quiver $(Q(S,\overrightarrow{T}),f)$ 
is defined as follows:
\begin{itemize}
\item The set $Q(S,\overrightarrow{T})_0$ of vertices consists of the edges of $T$.
\item The set $Q(S,\overrightarrow{T})_1$ of arrows and permutation $f:Q(S,\overrightarrow{T})_1\to Q(S,\overrightarrow{T})_1$ 
are given by the orientation of triangles in $T$, namely 
\begin{enumerate}[(a)]
\item $\xymatrix{a\ar[rr]^{\alpha}& &b\ar[ld]^{\beta}\\ &c\ar[lu]^{\gamma}&}$ $\quad$ $f(\alpha)=\beta, f(\beta)=\gamma, f(\gamma)=\alpha$ 

for any oriented triangle $\Delta=(a b c )$ in $\overrightarrow{T}$, and $a,b,c$ pairwise different edges,
\item $\xymatrix@R=0.01cm{&\ar@(lu,ld)[d]_{\alpha}&\\& a\ar@<0.2cm>[r]^{\beta}  &b \ar@<0.1cm>[l]^{\gamma}\\ && }$ $\quad$ $f(\alpha)=\beta, f(\beta)=\gamma, f(\gamma)=\alpha$

for a self-folded triangle $\Delta=(aab)$ in $\overrightarrow{T}$ ($a,b$ different edges), and 
\item $\xymatrix@R=0.01cm{&\ar@(lu,ld)[d]_{\alpha}\\& a }$ $\quad$ $f(\alpha)=\alpha,$ 

for any boundary edge $a$ of $T$. 
\end{enumerate}
\end{itemize} 

\begin{df}\label{df:6.1} Let $(S,\overrightarrow{T})$ be a directed triangulated surface and $I$ a non-empty set of edges in $T$ 
(possibly all the edges). The \emph{blow-up} of $(S,\overrightarrow{T})$ at $I$ is the directed triangulated surface 
$(S,\overrightarrow{T_I})$ obtained from $(S,\overrightarrow{T})$ by replacing each edge $i\in I$ by a $2$-triangle disc 
$$
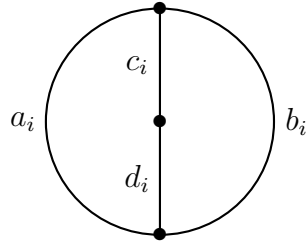
\begin{tikzpicture}[scale=1.5]
\draw[thick] (0,-1)--(0,1); 
\draw[thick](0,0) circle (1);
\node() at (0,0){$\bullet$};
\node() at (0,-1){$\bullet$}; 
\node() at (0,1){$\bullet$}; 
\node() at (1.2,0){$b_i$};
\node() at (-1.2,0){$a_i$};
\node() at (-0.2,0.5){$c_i$};
\node() at (-0.2,-0.5){$d_i$};
\end{tikzpicture}$$
with $(a_i\mbox{ }c_i\mbox{ }d_i)$ and $(c_i\mbox{ }b_i\mbox{ }d_i)$ in $\overrightarrow{T_I}$. We note that since we require 
$\overrightarrow{T_I}$ to be a coherent orientation of triangles in $T_I$, this blow-up is uniquely determined by $I$ 
(and the coherent orientation $\overrightarrow{T}$ of triangles in $T$). \end{df} 

Let $(S,\overrightarrow{T_I})$ be the blow-up of a directed triangulated surface $(S,\overrightarrow{T})$ at a set $I$ of edges 
of $T$. Let $(Q,f)$ denote the triangulation quiver $(Q(S,\overrightarrow{T}),f)$ and $(Q^I,f^I)$ the triangulation quiver 
associated to $(S,\overrightarrow{T_I})$. We write $g:Q_1\to Q_1$ for the permutation associated to $f$ and $g^I:Q^I_1\to Q^I_1$ 
for the permutation associated to $f^I$. For an arrow $\eta\in Q^I_1$, we denote by $\mathcal{O}^I(\eta)$ the $g^I$-orbit of 
$\eta$ in $Q^I_1$. Moreover, let $\mathcal{O}(g^I)$ be the set of all $g^I$-orbits in $Q^I_1$. For each edge $i\in I$, we 
abbreviate $(S,\overrightarrow{T_i})=(S,\overrightarrow{T_{\{i\}}})$, $f_i=f^{\{i\}}$ and $g_i=g^{\{i\}}$. We mention that 
$(S,\overrightarrow{T_I})$ (respectively, $(Q^I,f^I)$) is obtained from $(S,\overrightarrow{T})$ (respectively, from $(Q,f)$) 
by an iterated application of single blow-ups at all edges from $I$ (in arbitrary fixed order).

We will illustrate the changes for single blow-up at an edge $i$ of $T$.

\begin{enumerate}[(1)] 
\item Assume $i$ is a common edge of two triangles in $T$ 
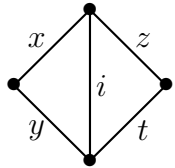
$$\begin{tikzpicture}
\draw[thick](-1,0)--(0,1);
\draw[thick](-1,0)--(0,-1);
\draw[thick](1,0)--(0,1);
\draw[thick](1,0)--(0,-1);
\draw[thick](0,-1)--(0,1);
\node() at (-0.7,0.6){$x$};
\node() at (-0.7,-0.6){$y$};
\node() at (0.7,0.6){$z$};
\node() at (0.7,-0.6){$t$};
\node() at (0.15,0){$i$};
\node() at (-1,0){$\bullet$};
\node() at (1,0){$\bullet$};
\node() at (0,1){$\bullet$};
\node() at (0,-1){$\bullet$};
\end{tikzpicture}$$
with $(x\mbox{ }i\mbox{ }y)$ and $(i\mbox{ }z\mbox{ }t)$ in $\overrightarrow{T}$, where possibly $x=y$ or $z=t$. Then 
$(Q,f)$ has arrows 
$$
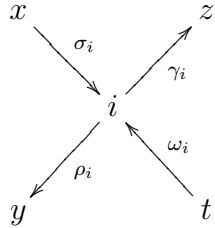
\xymatrix{x\ar[rd]^{\sigma_i}&&z \\ &i\ar[ld]^{\rho_i}\ar[ru]_{\gamma_i}& \\ y&&t\ar[lu]_{\omega_i}}$$
with $f(\sigma_i)=\rho_i$, $f(\omega_i)=\gamma_i$, $g(\sigma_i)=\gamma_i$ and $g(\omega_i)=\rho_i$. Then the blow-up 
$(S,\overrightarrow{T_i})$ contains triangles 
$$
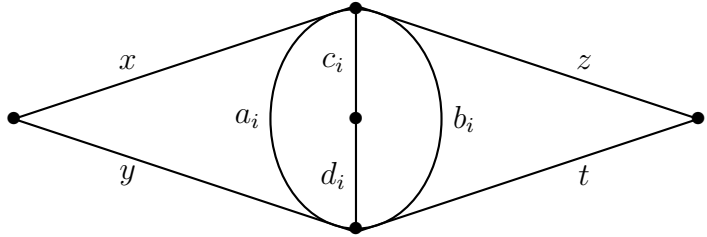
\begin{tikzpicture}[scale=1.5]
\draw[thick] (0,-1)--(0,1); 
\draw[thick](0,0) ellipse [x radius = 0.75, y radius = 0.97];
\node() at (0,0){$\bullet$};
\node() at (0,-0.975){$\bullet$}; 
\node() at (0,0.975){$\bullet$}; 
\node() at (0.95,0){$b_i$};
\node() at (-0.95,0){$a_i$};
\node() at (-0.2,0.5){$c_i$};
\node() at (-0.2,-0.5){$d_i$};

\draw[thick](-3,0)--(0,1);
\draw[thick](-3,0)--(0,-1);
\draw[thick](3,0)--(0,1);
\draw[thick](3,0)--(0,-1);

\node() at (-3,0){$\bullet$}; 
\node() at (3,0){$\bullet$};

\node() at (-2,0.5){$x$};
\node() at (-2,-0.5){$y$};
\node() at (2,0.5){$z$};
\node() at (2,-0.5){$t$}; 
\end{tikzpicture}$$
with $(x\mbox{ }a_i\mbox{ }y)$, $(a_i\mbox{ }c_i\mbox{ }d_i)$, $(c_i\mbox{ }b_i\mbox{ }d_i)$, $(b_i\mbox{ }z\mbox{ }t)\in 
\overrightarrow{T_i}$, so $(Q(S,\overrightarrow{T_i}),f_i)$ contains a subquiver of the form 
$$
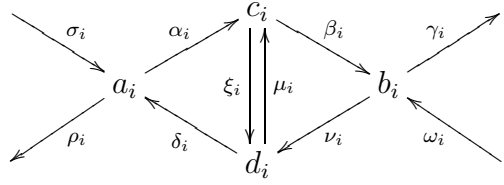
\xymatrix@R=0.5cm@C=1.2cm{\ar[rd]^{\sigma_i}&&c_i\ar[rd]^{\beta_i}\ar@<-0.1cm>[dd]_{\xi_i}&& \\ 
&a_i\ar[ld]^{\rho_i}\ar[ru]^{\alpha_i}&&b_i\ar[ld]^{\nu_i}\ar[ru]^{\gamma_i} 
& \\ &&d_i\ar[lu]^{\delta_i}\ar@<-0.1cm>[uu]_{\mu_i}&&\ar[lu]^{\omega_i} }$$
with $f_i$-orbits $(\alpha_i\mbox{ }\xi_i\mbox{ }\delta_i)$, $(\beta_i\mbox{ }\nu_i\mbox{ }\mu_i)$, $f_i(\sigma_i)=\rho_i$, 
$f_i(\omega_i)=\gamma_i$, and $g_i(\sigma_i)=\alpha_i$, $g_i(\alpha_i)=\beta_i$, $g_i(\beta_i)=\gamma_i$, $g_i(\omega_i)=\nu_i$, 
$g_i(\nu_i)=\delta_i$, $g_i(\delta_i)=\rho_i$, $g_i(\xi_i)=\mu_i$ and $g_i(\mu_i)=\xi_i$.

\item Assume that $i$ is a self-folded edge of the following self-folded triangle in $T$ 
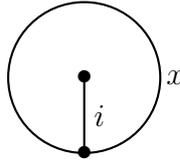
$$\begin{tikzpicture}
\draw[thick](0,0) circle (1);
\draw[thick](0,0)--(0,-1);
\node() at (0,0){$\bullet$};
\node() at (0,-1){$\bullet$};
\node() at (0.2,-0.5){$i$};
\node() at (1.2,0){$x$};
\end{tikzpicture}$$
with $(i\mbox{ }i\mbox{ }x)\in\overrightarrow{T}$, and hence $(Q,f)$ admits a subquiver 
$$
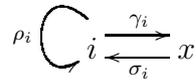
\xymatrix@R=0.01cm{&\ar@(lu,ld)[d]_{\rho_i}&\\& i\ar@<0.2cm>[r]^{\gamma_i}  &x \ar@<0.1cm>[l]^{\sigma_i}\\ && }$$
with $f$-orbit $(\rho_i\mbox{ }\gamma_i\mbox{ }\sigma_i)$, and $g(\rho_i)=\rho_i$, $g(\sigma_i)=\gamma_i$. 

Then the blow-up $(S,\overrightarrow{T_i})$ contains the triangles 
$$
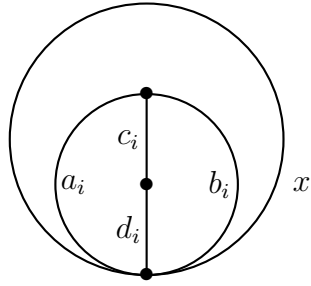
\begin{tikzpicture}[scale=1.2]
\draw[thick] (0,-1)--(0,1); 
\draw[thick](0,0) circle (1);
\node() at (0,0){$\bullet$};
\node() at (0,-1){$\bullet$}; 
\node() at (0,1){$\bullet$}; 
\node() at (0.8,0){$b_i$};
\node() at (-0.8,0){$a_i$};
\node() at (-0.2,0.5){$c_i$};
\node() at (-0.2,-0.5){$d_i$};
\draw[thick](0,0.5) circle (1.5);
\node() at (1.7,0){$x$};
\end{tikzpicture}$$
with $(x\mbox{ }b_i\mbox{ }a_i)$, $(a_i\mbox{ }c_i\mbox{ }d_i)$, $(c_i\mbox{ }b_i\mbox{ }d_i)$ in $\overrightarrow{T_i}$, 
and so $(Q(S,\overrightarrow{T_i}),f_i)$ has a subquiver of the form 
$$
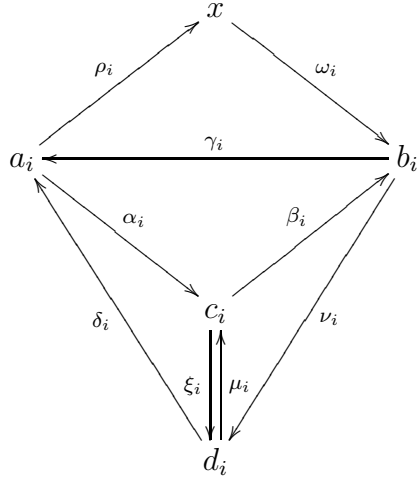
\xymatrix@R=1.5cm@C=2cm{&x\ar[rd]^{\omega_i}& \\
a_i\ar[ru]^{\rho_i}\ar[rd]^{\alpha_i}& &  b_i\ar[ll]_{\gamma_i}\ar[ldd]^{\nu_i} \\ 
&c_i\ar[ru]^{\beta_i}\ar@<-0.07cm>[d]_{\xi_i}& \\ 
&d_i\ar@<-0.07cm>[u]_{\mu_i}\ar[luu]^{\delta_i}&}$$
where $f_i$-orbits are $(\alpha_i\mbox{ }\xi_i\mbox{ }\delta_i)$, $(\beta_i\mbox{ }\nu_i\mbox{ }\mu_i)$, 
$(\sigma_i\mbox{ }\rho_i\mbox{ }\gamma_i)$. Moreover, we have $g_i$-orbits $(\xi_i\mbox{ }\mu_i)$, $(\alpha_i\mbox{ }\beta_i\mbox{ }\gamma_i)$ 
and $(\delta_i\mbox{ }\rho_i\dots \omega_i\mbox{ }\nu_i)$. 

\item Assume $i$ is a border edge of $T$, so we have in $T$ a triangle 
$$
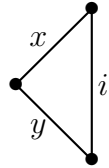
\begin{tikzpicture}
\draw[thick](-1,0)--(0,1);
\draw[thick](-1,0)--(0,-1);
\draw[thick](0,-1)--(0,1);
\node() at (-0.7,0.6){$x$};
\node() at (-0.7,-0.6){$y$};
\node() at (0.15,0){$i$};
\node() at (-1,0){$\bullet$};
\node() at (0,1){$\bullet$};
\node() at (0,-1){$\bullet$};
\end{tikzpicture}$$
with $(x\mbox{ }i\mbox{ }y)$ in $\overrightarrow{T}$ and possibly $x=y$. Then $(Q,f)$ contains a subquiver 
$$
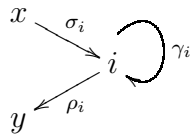
\xymatrix@R=0.01cm{x\ar[rdd]^{\sigma_i}&& \\ &\ar@(ur,dr)[d]^{\gamma_i}& \\ 
&i\ar[ldd]^{\rho_i}& \\ && \\ y&&}$$
with $f(\gamma_i)=\gamma_i$, $f(\sigma_i)=\rho_i$, $g(\sigma_i)=\gamma_i$, $g(\gamma_i)=\rho_i$. 

The blow-up $(S,\overrightarrow{T_i})$ is then containing the triangles 
$$
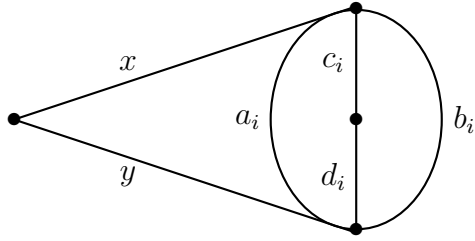
\begin{tikzpicture}[scale=1.5]
\draw[thick] (0,-1)--(0,1); 
\draw[thick](0,0) ellipse [x radius = 0.75, y radius = 0.97];
\node() at (0,0){$\bullet$};
\node() at (0,-0.975){$\bullet$}; 
\node() at (0,0.975){$\bullet$}; 
\node() at (0.95,0){$b_i$};
\node() at (-0.95,0){$a_i$};
\node() at (-0.2,0.5){$c_i$};
\node() at (-0.2,-0.5){$d_i$};

\draw[thick](-3,0)--(0,1);
\draw[thick](-3,0)--(0,-1);

\node() at (-3,0){$\bullet$}; 

\node() at (-2,0.5){$x$};
\node() at (-2,-0.5){$y$}; 
\end{tikzpicture}$$ 
with $(x\mbox{ }a_i\mbox{ }y)$, $(a_i\mbox{ }c_i\mbox{ }b_i)$, $(c_i\mbox{ }b_i\mbox{ }d_i)$ in $\overrightarrow{T_i}$, and 
$b_i$ the boundary edge. In this case $(Q(S,\overrightarrow{T_i}),f_i)$ contains a subquiver of the form 
$$\qquad\qquad\qquad
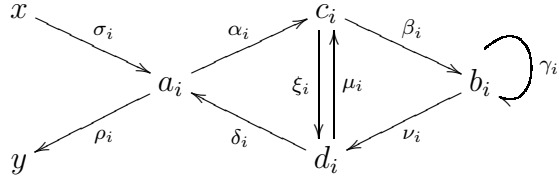
\xymatrix@R=0.01cm@C=1.5cm{
x\ar[rddd]^{\sigma_i}&&c_i\ar[rddd]^{\beta_i}\ar@<-0.1cm>[dddddd]_{\xi_i}&& \\
&&&& \\
&&&\ar@(ur,dr)[d]^{\gamma_i} & \\ 
&a_i\ar[lddd]^{\rho_i}\ar[ruuu]^{\alpha_i}&&b_i\ar[lddd]^{\nu_i}&
&&&& \\
&&&& \\ 
&&&& \\
y&&d_i\ar[luuu]^{\delta_i}\ar@<-0.1cm>[uuuuuu]_{\mu_i}&&}$$
with $f_i$-orbits $(\alpha_i\mbox{ }\xi_i\mbox{ }\delta_i)$, $(\beta_i\mbox{ }\nu_i\mbox{ }\mu_i)$, $(\gamma_i)$, 
$f(\sigma_i)=\rho_i$, and the $g_i$-orbits $(\xi_i\mbox{ }\mu_i)$ and 
$(\alpha_i\mbox{ }\beta_i\mbox{ }\gamma_i\mbox{ }\nu_i\mbox{ }\delta_i\mbox{ }\rho_i\dots\sigma_i)$.
\end{enumerate} 

Therefore, for the single blow-up $(S,\overrightarrow{T_i})$ of $(S,\overrightarrow{T})$ at an edge $i$, the set 
$\mathcal{O}(g_i)$ of the $g_i$-orbits in $Q(S,\overrightarrow{T_i})_1$ has the following structure: we have the new orbit 
$\mathcal{O}_i=(\xi_i\mbox{ }\mu_i)$ of length $2$, one or two orbits in $\mathcal{O}(g)$ are extended by the arrows 
$\alpha_i,\beta_i,\nu_i,\delta_i$, and the remaining orbits in $\mathcal{O}(g)$ become orbits in $\mathcal{O}(g_i)$. 

In general, for an arbitrary non-empty set $I$ of edges of $T$, the set $\mathcal{O}(g^I)$ of all $g^I$-orbits in $Q^I_1=
Q(S,\overrightarrow{T_I})_1$ consists of the new orbits $\mathcal{O}_i=(\xi_i\mbox{ }\mu_i)$, $i\in I$, of length $2$, and 
extensions $\mathcal{O}^I(\eta)$ of the $g$-orbits $\mathcal{O}(\eta)$, $\eta\in Q_1$. Hence, for any choice of a weight 
function $m_\bullet:\mathcal{O}(g)\to\mathbb{N}^*$ and a parameter function $c_\bullet:\mathcal{O}(g)\to K^*$, we may 
consider the weight function $m^I_\bullet:\mathcal{O}(g^I)\to\mathbb{N}^*$ and the parameter function 
$c^I_\bullet:\mathcal{O}(g^I)\to K^*$ by setting 
\begin{itemize}
\item $m^I_{\mathcal{O}_i}=1$ and $c^I_{\mathcal{O}_i}=1$, for any $i\in I$, and 
\item $m^I_{\mathcal{O}^I(\eta)}=m_{\mathcal{O}(\eta)}$ and $c^I_{\mathcal{O}^I(\eta)}=c_{\mathcal{O}(\eta)}$, for any arrow 
$\eta\in Q_1$. \end{itemize}

\begin{df}\label{df:6.2} The weighted surface algebra $\Lambda(S,\overrightarrow{T_I},m^I_\bullet,c^I_\bullet)$ is said to 
be the \emph{blow-up} of the weighted surface algebra $\Lambda(S,\overrightarrow{T},m_\bullet,c_\bullet)$ at the set $I$ of edges 
of $T$.\end{df} 

Let $\Lambda_I=\Lambda(S,\overrightarrow{T_I},m^I_\bullet,c^I_\bullet)$ be the blow-up of a weighted surface algebra 
$\Lambda=\Lambda(S,\overrightarrow{T},m_\bullet,c_\bullet)$ at a non-empty set $I$ of edges of $T$. Then for each edge 
$i\in I$, we have the pair of virtual arrows 
$$\xymatrix{d_i\ar@<-0.1cm>[r]_{\mu_i} & c_i\ar@<-0.1cm>[l]_{\xi_i}}$$
from the $g^I$-orbit $\mathcal{O}_i=(\xi_i\mbox{ }\mu_i)$. Moreover, let $\epsilon:I\to\{-1,1\}$ be an arbitrary function. We 
assign to $\epsilon$ a sequence of virtual arrows $\underline{\epsilon}=(\epsilon_i)_{i\in I}$, with $\epsilon_i\in\mathcal{O}_i$, 
where for any $i\in I$, we set 
$$\epsilon_i:=\left\{\begin{array}{cc}\xi_i & \quad\mbox{if }\epsilon(i)=1 \\ \mu_i & \quad\mbox{if }\epsilon(i)=-1. \end{array} \right.$$

\begin{df}\label{df:6.3} Let $\Lambda=\Lambda(S,\overrightarrow{T},m_\bullet,c_\bullet)$ be a weighted surface algebra, $I$ 
a non-empty set of edges of $T$, and $\epsilon:I\to\{-1,1\}$ a function. Then the virtual mutation $\Lambda^\epsilon_I:= 
\Lambda(S,\overrightarrow{T}_I,m^I_\bullet,c^I_\bullet,\underline{\epsilon})$ of the weighted surface algebra $\Lambda_I=
\Lambda(S,\overrightarrow{T}_I,m^I_\bullet,c^I_\bullet)$ is said to be a \emph{virtual edge deformation} of $\Lambda=
\Lambda(S,\overrightarrow{T},m_\bullet,c_\bullet)$ at the set $I$ of edges in $T$. \end{df} 

We have the following direct consequence of the main result of this paper. 

\begin{theorem}\label{thm:6.4} Let $\Lambda^\epsilon_I=\Lambda(S,\overrightarrow{T}_I,m^I_\bullet,c^I_\bullet,\underline{\epsilon})$ 
be a virtual edge deformation of a weighted surface algebra $\Lambda=\Lambda(S,\overrightarrow{T},m_\bullet,c_\bullet)$. Then 
$\Lambda^\epsilon_I$ is a finite-dimensional, symmetric, tame periodic algebra of period $4$.
\end{theorem}

\begin{exmp}\label{ex:6.5} Let $(S,\overrightarrow{T})$ be the directed surface given by the torus $S=\mathbb{T}$ with the 
triangulation $T$ and orientation $\overrightarrow{T}$ of triangles as follows
$$
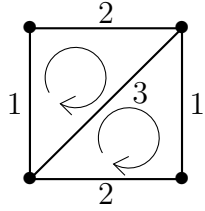
\begin{tikzpicture} 
\draw[thick](-1,1)--(1,1)--(1,-1)--(-1,-1)--(-1,1);
\draw[thick](-1,-1)--(1,1);
\node() at (-1,-1){$\bullet$};
\node() at (-1,1){$\bullet$};
\node() at (1,-1){$\bullet$};
\node() at (1,1){$\bullet$}; 

\node() at (0,1.2){$2$};
\node() at (0,-1.2){$2$}; 
\node() at (-1.2,0){$1$};
\node() at (1.2,0){$1$}; 
\node() at (0.45,0.15){$3$};
\draw(-0.6,0) arc (-120:210:0.4);
\draw(-0.4,0.1)--(-0.6,0)--(-0.45,-0.15);
\draw(0.1,-0.8) arc (-120:210:0.4);
\draw(0.3,-0.7)--(0.1,-0.8)--(0.25,-0.95);
 
\end{tikzpicture}$$ 
Then $(Q(S,\overrightarrow{T}),f)$ is the Markov triangulation quiver (see \cite[Section 5]{ES4} and \cite{Lam}) 
$$\xymatrix@R=1.5cm{1 \ar@<-0.1cm>[rr]_{\sigma}\ar@<0.1cm>[rr]^{\gamma} & & 
2\ar@<-0.1cm>[ld]_{\eta}\ar@<0.1cm>[ld]^{\omega} \\ 
& 3\ar@<-0.1cm>[lu]_{\phi}\ar@<0.1cm>[lu]^{\rho} & }$$
with the $f$-orbits $(\gamma\mbox{ }\omega\mbox{ }\rho)$ and $(\sigma\mbox{ }\eta\mbox{ }\phi)$, given by the upper and the 
lower oriented triangles of $\overrightarrow{T}$, respectively. We note that $\mathcal{O}(g)$ consists of one $g$-orbit 
$\mathcal{O}(\gamma)=(\gamma\mbox{ }\eta\mbox{ }\rho\mbox{ }\sigma\mbox{ }\omega\mbox{ }\phi)$. Let $I=\{2,3\}$ be the chosen 
set of edges of $T$. Then the blow-up $(S,\overrightarrow{T_I})$ of $(S,\overrightarrow{T})$ at $I$ is of the form
$$
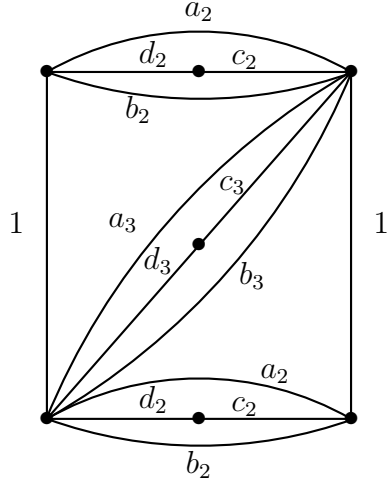
\begin{tikzpicture}[scale=2] 
\draw[thick](-1,1)--(1,1)--(1,-1.3)--(-1,-1.3)--(-1,1);
\draw[thick](-1,-1.3)--(1,1);
\node() at (-1,-1.3){$\bullet$};
\node() at (-1,1){$\bullet$};
\node() at (1,-1.3){$\bullet$};
\node() at (1,1){$\bullet$};
\node() at (0,1){$\bullet$};
\node() at (0,-1.3){$\bullet$};
\node() at (0,-0.15){$\bullet$}; 

\draw[thick](1,1) arc (60:120:2);   
\draw[thick](-1,1) arc (-110:-70:2.95);
\draw[thick](1,1) arc (120:158:4.7);
\draw[thick](-1,-1.3) arc (-60:-22:4.7);
\draw[thick](1,-1.3) arc (60:120:2);   
\draw[thick](-1,-1.3) arc (-110:-70:2.95);

\node() at (-1.2,0){$1$};
\node() at (1.2,0){$1$}; 

\node() at (0,1.4){$a_2$};
\node() at (-0.3,1.12){$d_2$};
\node() at (0.3,1.08){$c_2$}; 
\node() at (0.5,-1){$a_2$};
\node() at (-0.3,-1.17){$d_2$};
\node() at (0.3,-1.22){$c_2$};

\node() at (-0.5,0){$a_3$};
\node() at (-0.27,-0.25){$d_3$};
\node() at (0.22,0.25){$c_3$};
 
\node() at (-0.4,0.75){$b_2$};
\node() at (0.35,-0.35){$b_3$};
\node() at (0,-1.6){$b_2$};
\end{tikzpicture}$$
with the following orientation of triangles of $T_I$: $(1\mbox{ }b_2\mbox{ }a_3)$, $(1\mbox{ }a_2\mbox{ }b_3)$, 
$(a_2\mbox{ }c_2\mbox{ }d_2)$, $(c_2\mbox{ }b_2\mbox{ }d_2)$, $(a_3\mbox{ }c_3\mbox{ }d_3)$, $(c_3\mbox{ }b_3\mbox{ }d_3)$. 
Hence the associated triangulation quiver $(Q^I,f^I)$ of $(S,\overrightarrow{T_I})$ is of the form 
$$
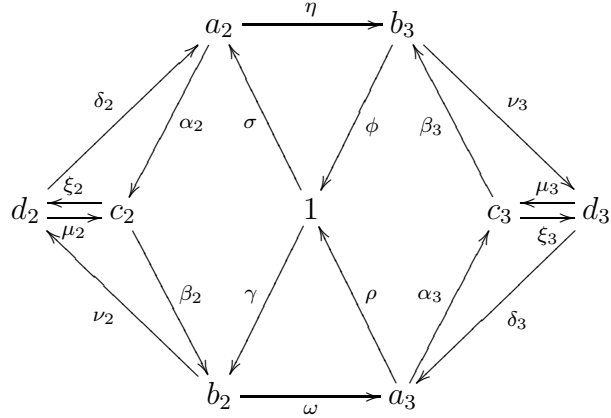
\xymatrix@C=0.7cm@R=1.9cm{
&&a_2\ar[ld]^{\alpha_2}\ar[rr]^{\eta}&&b_3\ar[ld]^{\phi}\ar[rrd]^{\nu_3}&& \\ 
d_2\ar[rru]^{\delta_2}\ar@<-0.1cm>[r]_{\mu_2} & c_2\ar@<-0.1cm>[l]_{\xi_2}\ar[rd]^{\beta_2}&& 
1\ar[lu]^{\sigma}\ar[ld]_{\gamma}&& 
c_3\ar@<-0.1cm>[r]_{\xi_3}\ar[lu]^{\beta_3} & d_3\ar@<-0.1cm>[l]_{\mu_3}\ar[lld]^{\delta_3} \\ 
&&b_2\ar[llu]^{\nu_2}\ar[rr]_{\omega}&&a_3\ar[lu]_{\rho}\ar[ru]^{\alpha_3}&& }$$ 
with $f^I$-orbits $(\gamma\mbox{ }\omega\mbox{ }\rho)$, $(\sigma\mbox{ }\eta\mbox{ }\phi)$, 
$(\alpha_2\mbox{ }\xi_2\mbox{ }\delta_2)$, $(\beta_2\mbox{ }\nu_2\mbox{ }\mu_2)$, 
$(\alpha_3\mbox{ }\xi_3\mbox{ }\delta_3)$, $(\beta_3\mbox{ }\nu_3\mbox{ }\mu_3)$. Moreover, the set $\mathcal{O}(g^I)$ of 
$g^I$-orbits in $Q^I_1$ consists of the two orbits 
$$\mathcal{O}_2=\mathcal{O}^I(\xi_2)=(\xi_2\mbox{ }\mu_2)=\mathcal{O}^I(\mu_2)\mbox{ and }
\mathcal{O}_3=\mathcal{O}^I(\xi_3)=(\xi_3\mbox{ }\mu_3)=\mathcal{O}^I(\mu_3)$$ 
of length $2$, an the following orbit of length $14$ 
$$\mathcal{O}^I(\gamma)=
(\gamma\mbox{ }\nu_2\mbox{ }\delta_2\mbox{ }\eta\mbox{ }\nu_3\mbox{ }\delta_3\mbox{ }\rho\mbox{ }\sigma\mbox{ }\alpha_2\mbox{ }\beta_2
\mbox{ }\omega\mbox{ }\alpha_3\mbox{ }\beta_3\mbox{ }\phi).$$ 

Let $m_\bullet:\mathcal{O}(g)\to\mathbb{N}^*$ and $c_\bullet:\mathcal{O}(g)\to K^*$ be the weight and parameter functions 
with $m_{\mathcal{O}(\gamma)}=1=c_{\mathcal{O}(\gamma)}$. Then the weight function $m^I_\bullet:\mathcal{O}(g^I)\to\mathbb{N}^*$ 
and the parameter function $c^I_\bullet:\mathcal{O}(g^I)\to K^*$ are both constantly equal $1$. Consider also the associated 
weighted surface algebras 
$$\Lambda=\Lambda(S,\overrightarrow{T},m_\bullet,c_\bullet)\mbox{ and }\Lambda_I=
\Lambda(S,\overrightarrow{T_I},m^I_\bullet,c^I_\bullet).$$ 
It follows from general theory (Theorem \ref{thm:2.4}) that 
$$\dim_K\Lambda= 6^2 =36\quad\mbox{and}\quad\dim_K\Lambda_I=2^2+2^2+14^2=204.$$

Let $\epsilon:I\to\{-1,1\}$ be the function with $\epsilon(2)=\epsilon(3)=1$. Then the associated sequence of virtual 
arrows $\underline{\epsilon}\in\mathcal{O}_2\times\mathcal{O}_3$ is equal to $(\xi_2,\xi_3)$. Take the virtual edge 
deformation algebra $\Lambda^\epsilon_I$ of $\Lambda$ at $I$, with respect to $\epsilon$. Directly from Definition \ref{def:3.4} we 
obtain that the algebra $\Lambda^\epsilon_I$ is given by the quiver $Q^I(\underline{\epsilon})$ of the form 
$$
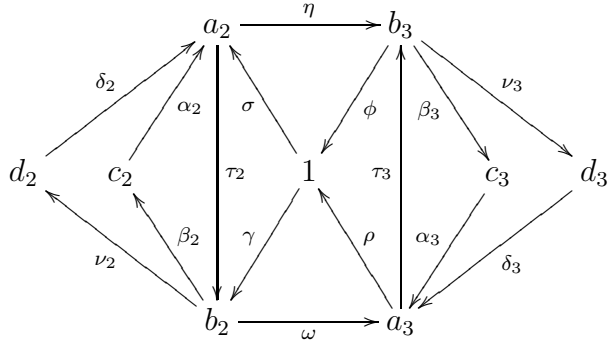
\xymatrix@C=0.7cm@R=1.4cm{
&&a_2\ar[rr]^{\eta}\ar[dd]^{\tau_2}&&b_3\ar[ld]^{\phi}\ar[rrd]^{\nu_3}\ar[rd]_{\beta_3}&& \\ 
d_2\ar[rru]^{\delta_2} & c_2\ar[ru]_{\alpha_2}&& 
1\ar[lu]^{\sigma}\ar[ld]_{\gamma}&& 
c_3\ar[ld]_{\alpha_3} & d_3\ar[lld]^{\delta_3} \\ 
&&b_2\ar[llu]^{\nu_2}\ar[rr]_{\omega}\ar[lu]_{\beta_2}&&a_3\ar[lu]_{\rho}\ar[uu]^{\tau_3}&& }$$
and the relations: 
\begin{enumerate}[(1)]
\item $\nu_2\delta_2=\beta_2\alpha_2+\omega\tau_3\phi\gamma\nu_2\delta_2\eta\nu_3\delta_3\rho\sigma$, 
$\nu_3\delta_3=\beta_3\alpha_3+\phi\gamma\nu_2\delta_2\eta\nu_3\delta_3\rho\sigma\tau_2\omega$, $\newline$ 
$\alpha_2\tau_2=0$, $\alpha_3\tau_3=0$, $\tau_2\beta_2=0$, $\tau_3\beta_3=0$,  

\item $\tau_2\nu_2=\eta\nu_3\delta_3\rho\sigma\tau_2\omega\tau_3\phi\gamma\nu_2$, 
$\tau_3\nu_3=\rho\sigma\tau_2\omega\tau_3\phi\gamma\nu_2\delta_2\eta\nu_3$, $\newline$  
$\delta_2\tau_2=\delta_2\eta\nu_3\delta_3\rho\sigma\tau_2\omega\tau_3\phi\gamma$, 
$\delta_3\tau_3=\delta_3\rho\sigma\tau_2\omega\tau_3\phi\gamma\nu_2\delta_2\eta$, $\newline$
$\gamma\omega=\sigma\tau_2\omega\tau_3\phi\gamma\nu_2\delta_2\eta\nu_3\delta_3$, 
$\omega\rho=\nu_2\delta_2\eta\nu_3\delta_3\rho\sigma\tau_2\omega\tau_3\phi$, $\newline$
$\rho\gamma=\tau_3\phi\gamma\nu_2\delta_2\eta\nu_3\delta_3\rho\sigma\tau_2$, 
$\sigma\eta=\gamma\nu_2\delta_2\eta\nu_3\delta_3\rho\sigma\tau_2\omega\tau_3$, $\newline$
$\eta\phi=\tau_2\omega\tau_3\phi\gamma\nu_2\delta_2\eta\nu_3\delta_3\rho$, 
$\phi\sigma=\nu_3\delta_3\rho\sigma\tau_2\omega\tau_3\phi\gamma\nu_2\delta_2$, 

\item $\gamma\omega\tau_3=0$, $\omega\rho\sigma=0$, $\rho\gamma\nu_2=0$, $\sigma\eta\nu_3=0$, $\eta\phi\gamma=0$, 
$\phi\sigma\tau_2=0$,

\item $\omega\tau_3\nu_3=0$, $\rho\sigma\eta=0$, $\sigma\tau_2\nu_2=0$, $\phi\gamma\omega=0$, 

\item $\alpha_2\eta\nu_3\delta_3\rho\sigma\tau_2\omega\tau_3\phi\gamma\nu_2=0$, 
$\alpha_3\rho\sigma\tau_2\omega\tau_3\phi\gamma\nu_2\delta_2\eta\nu_3=0$, $\newline$
$\delta_2\eta\nu_3\delta_3\rho\sigma\tau_2\omega\tau_3\phi\gamma\beta_2=0$,
$\delta_3\rho\sigma\tau_2\omega\tau_3\phi\gamma\nu_2\delta_2\eta\beta_3=0$.
\end{enumerate} 

We also observe that we have only one $g^*$-orbit of arrows in the quiver $Q^I(\underline{\epsilon})^*$ consisting of all 
arrows of $Q^I(\underline{\epsilon})$ different from $\alpha_2,\beta_2,\alpha_3,\beta_3$. Hence, for any arrow $\theta$ in 
$Q^I(\underline{\epsilon})^*_1$ we have $n^*_{\theta}=10$, while $n^\nu_{\theta}=2$, and clearly $m^*_\theta=1$. Then, 
applying Theorem \ref{thm:3.14}, we conclude that 
$$\dim_K\Lambda^\epsilon_I=\sum_{\theta\in Q^I(\underline{\epsilon})^*_1} m^*_\theta(n^*_\theta+n^\nu_\theta)+ 
m^*_{\delta_2}(n^*_{\delta_2}+n^\nu_{\delta_2})+m^*_{\delta_3}(n^*_{\delta_3}+n^\nu_{\delta_3})=$$
$$=12(12+2)+(12+2)+(12+2)=196.$$
\end{exmp}

\begin{exmp}\label{ex:6.6} Let $(S,\overrightarrow{T})$ be the collection of three self-folded triangles 
$$
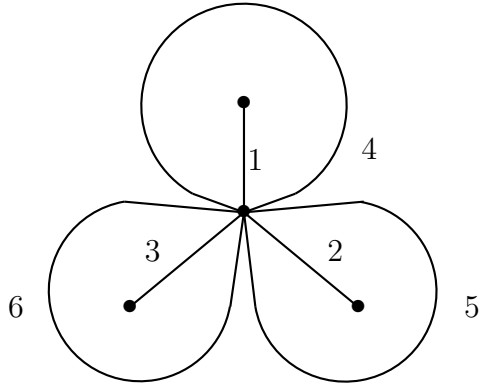
\begin{tikzpicture}[scale=1.5]
\draw[thick](-1,0)--(0,0.835)--(1,0);
\draw[thick](0,0.835)--(0,1.8); 

\draw[thick](0.1,0) arc (-170:80:0.8);
\draw[thick](0.1,0)--(0,0.835)--(1.05,0.93);

\draw[thick](-1.05,0.93) arc (100:350:0.8);
\draw[thick](-0.12,0)--(0,0.835)--(-1.05,0.93); 

\draw[thick](0.45,1) arc (-60:240:0.9);
\draw[thick](0.45,1)--(0,0.835)--(-0.45,1);

\node() at (0,0.835){$\bullet$};
\node() at (-1,0){$\bullet$};
\node() at (1,0){$\bullet$};
\node() at (0,1.8){$\bullet$};

\node() at (0.1,1.3){$1$};
\node() at (0.8,0.5){$2$};
\node() at (-0.8,0.5){$3$};
\node() at (1.1,1.4){$4$};
\node() at (2,0){$5$};
\node() at (-2,0){$6$};

\end{tikzpicture}$$
with $(1\mbox{ }1\mbox{ }4)$, $(2\mbox{ }2\mbox{ }5)$, $(3\mbox{ }3\mbox{ }6)$ and $(4\mbox{ }5\mbox{ }6)$ in $\overrightarrow{T}$. 
Then the associated triangulation quiver $(Q(S,\overrightarrow{T}),f)$ is of the form 
$$
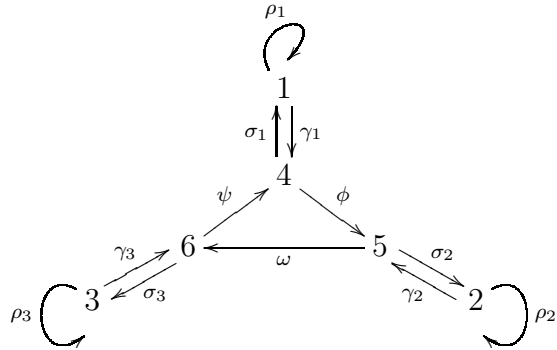
\xymatrix@R=0.01cm{
& &&&&& \\ 
& && 1\ar@<0.1cm>[dddd]^{\gamma_1}\ar@(lu,ru)[u]^{\rho_1}&&& \\ 
& &&&&& \\ 
&&&&&& \\ 
&&&&&& \\
& && 4\ar[rddd]^{\phi}\ar@<0.1cm>[uuuu]^{\sigma_1}&&& \\
&&&&&&\\&&&&&&\\ 
&  & 6\ar[ruuu]^{\psi}\ar@<0.1cm>[ldd]^{\sigma_3} &&5\ar[ll]^{\omega}\ar@<0.1cm>[rdd]^{\sigma_2}& & \\ 
&  &  &&&   & \\ 
&3\ar@<0.1cm>[ruu]^{\gamma_3}\ar@(ul,dl)[d]_{\rho_3} &     &&&  2\ar@<0.1cm>[luu]^{\gamma_2}\ar@(ur,dr)[d]^{\rho_2} & \\ 
&&&&&& \\ }$$
with $f$-orbits $(\gamma_1\mbox{ }\sigma_1\mbox{ }\rho_1)$, $(\gamma_2\mbox{ }\sigma_2\mbox{ }\rho_2)$, 
$(\gamma_3\mbox{ }\sigma_3\mbox{ }\rho_3)$ and $(\phi\mbox{ }\omega\mbox{ }\psi)$. Moreover, $\mathcal{O}(g)$ consists of the 
orbits $\mathcal{O}(\rho_1)=(\rho_1)$, $\mathcal{O}(\rho_2)=(\rho_2)$, $\mathcal{O}(\rho_3)=(\rho_3)$ and $\mathcal{O}(\gamma_1)=
(\gamma_1\mbox{ }\phi\mbox{ }\sigma_2\mbox{ }\gamma_2\mbox{ }\omega\mbox{ }\sigma_3\mbox{ }\gamma_3\mbox{ }\psi\mbox{ }\sigma_1)$. 
We take the weight function $m_\bullet:\mathcal{O}(g)\to \mathbb{N}^*$ and the parameter function 
$c_\bullet:\mathcal{O}(g)\to K^*$ defined as follows: 
$$m_{\mathcal{O}(\rho_1)}=m_{\mathcal{O}(\rho_2)}=m_{\mathcal{O}(\rho_3)}=m\geqslant 3\mbox{ and }m_{\mathcal{O}(\gamma_1)}=1,$$ 
$$c_{\mathcal{O}(\rho_1)}=c_{\mathcal{O}(\rho_2)}=c_{\mathcal{O}(\rho_3)}=1\mbox{ and }c_{\mathcal{O}(\gamma_1)}=\lambda\in K^*.$$

Let $I=\{1,2,3\}$ be the chosen set of edges of $T$. Then the blow-up $(S,\overrightarrow{T_I})$ of $(S,\overrightarrow{T})$ 
at $I$ is of the form 
$$
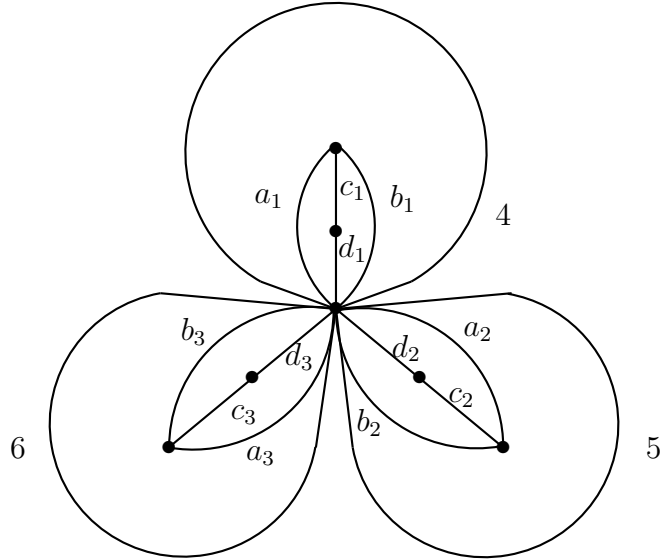
\begin{tikzpicture}[scale=2.2]
\draw[thick](-1,0)--(0,0.835)--(1,0);
\draw[thick](0,0.835)--(0,1.8); 

\draw[thick](0.1,0) arc (-170:80:0.8);
\draw[thick](0.1,0)--(0,0.835)--(1.05,0.93);

\draw[thick](-1.05,0.93) arc (100:350:0.8);
\draw[thick](-0.12,0)--(0,0.835)--(-1.05,0.93); 

\draw[thick](0.45,1) arc (-60:240:0.9);
\draw[thick](0.45,1)--(0,0.835)--(-0.45,1);

\draw[thick](0,0.835) arc (-50:50:0.65);
\draw[thick](0,1.83) arc (130:230:0.65);

\draw[thick](1,0) arc (0:100:0.84);
\draw[thick](0,0.835) arc (-180:-80:0.84);

\draw[thick](-1,0) arc (-100:0:0.84);
\draw[thick](0,0.835) arc (80:180:0.85);

\node() at (0,0.835){$\bullet$};
\node() at (-1,0){$\bullet$};
\node() at (1,0){$\bullet$};
\node() at (0,1.8){$\bullet$};
\node() at (0.5,0.4175){$\bullet$};
\node() at (0,1.3){$\bullet$};
\node() at (-0.5,0.4175){$\bullet$};

\node() at (-0.4,1.5){$a_1$};
\node() at (0.4,1.5){$b_1$};
\node() at (0.1,1.55){$c_1$};
\node() at (0.1,1.2){$d_1$};

\node() at (0.85,0.7){$a_2$};
\node() at (0.2,0.15){$b_2$};
\node() at (0.75,0.3){$c_2$};
\node() at (0.42,0.6){$d_2$};

\node() at (-0.45,-0.05){$a_3$};
\node() at (-0.85,0.7){$b_3$};
\node() at (-0.55,0.2){$c_3$};
\node() at (-0.22,0.55){$d_3$};

\node() at (1,1.4){$4$};
\node() at (1.9,0){$5$};
\node() at (-1.9,0){$6$};

\end{tikzpicture}$$
with the following orientation of triangles: $(4\mbox{ }5\mbox{ }6)$, $(b_1\mbox{ }a_1\mbox{ }4)$, $(b_2\mbox{ }a_2\mbox{ }5)$, 
$(b_3\mbox{ }a_3\mbox{ }6)$, and $(a_i\mbox{ }c_i\mbox{ }d_i), (c_i\mbox{ }b_i\mbox{ }d_i)$, for $i\in I$. Hence the associated 
triangulation quiver $(Q^I,f^I)=(Q(S,\overrightarrow{T_I}),f^I)$ is of the form 
$$\xymatrix@C=1.5cm@R=1.1cm{
&&&     d_1\ar@<0.25cm>[rdd]^{\delta_1}\ar@<-0.1cm>[d]_{\mu_1}    &&& \\ 
&&&     c_1\ar[ld]^{\beta_1}\ar[u]_{\xi_1}   &&& \\ 
&&   b_1\ar@<0.25cm>[ruu]^{\nu_1}\ar[rr]_{\rho_1} & &  a_1\ar[lu]^{\alpha_1}\ar[ld]^{\gamma_1} && \\
&&&       4\ar[lu]^{\sigma_1}\ar[rd]^{\phi}     &&& \\
& a_3\ar[r]^{\gamma_3}\ar[d]_{\alpha_3} & 6\ar[ru]^{\psi}\ar[d]^{\sigma_3} & & 
5\ar[ll]^{\omega}\ar[r]^{\sigma_2} & b_2\ar[ld]^{\rho_2}\ar@<0.25cm>[rdd]^{\nu_2} & \\ 
& c_3\ar[r]_{\beta_3}\ar@<-0.1cm>[ld]_{\xi_3} & b_3\ar[lu]^{\rho_3}\ar@<0.25cm>[lld]^{\nu_3} && a_2\ar[u]^{\gamma_2}\ar[r]_{\alpha_2} 
& c_2\ar[u]_{\beta_2}\ar@<-0.1cm>[rd]_{\xi_2} & \\
d_3\ar@<0.25cm>[ruu]^{\delta_3}\ar@<-0.1cm>[ru]_{\mu_3}&&&
&&& d_2\ar@<0.25cm>[llu]^{\delta_2}\ar@<-0.1cm>[lu]_{\mu_2}}$$
with $f^I$-orbits $(\phi\mbox{ }\omega\mbox{ }\psi)$ and $(\alpha_i\mbox{ }\xi_i\mbox{ }\delta_i)$, $(\beta_i\mbox{ }\nu_i\mbox{ }\mu_i)$, 
$(\rho_i\mbox{ }\gamma_i\mbox{ }\sigma_i)$, for $i\in\{1,2,3\}$. Further, the set $\mathcal{O}(g^I)$ of all $g^I$-orbits in 
$(Q^I,f^I)$ consists of the orbits 
$$\mathcal{O}_i=(\xi_i\mbox{ }\mu_i)\quad\mbox{ }\mathcal{O}^I(\alpha_i)=(\alpha_i\mbox{ }\beta_i\mbox{ }\rho_i), i\in\{1,2,3\},$$ 
and one larger orbit of length $15$
$$\mathcal{O}^I(\gamma_1)=(\gamma_1\mbox{ }\phi\mbox{ }\sigma_2\mbox{ }\nu_2\mbox{ }\delta_2\mbox{ }\gamma_2\mbox{ }\omega\mbox{ }
\sigma_3\mbox{ }\nu_3\mbox{ }\delta_3\mbox{ }\gamma_3\mbox{ }\psi\mbox{ }\sigma_1\mbox{ }\nu_1\mbox{ }\delta_1).$$
Moreover, the weight function $m^I_\bullet:\mathcal{O}(g^I)\to\mathbb{N}^*$ and the parameter function 
$c^I_\bullet:\mathcal{O}(g^I)\to K^*$ are given by $m^I_{\mathcal{O}_i}=1$, $m^I_{\mathcal{O}^I(\alpha_i)}=m$, $c^I_{\mathcal{O}_i}=1$, 
$c^I_{\mathcal{O}^I(\alpha_i)}=1$, for $i\in\{1,2,3\}$, and $m^I_{\mathcal{O}^I(\gamma_1)}=1, c^I_{\mathcal{O}^I(\gamma_1)}=\lambda$. 
Consider the associated weighted surface algebras 
$$\Lambda=\Lambda(S,\overrightarrow{T},m_\bullet,c_\bullet)\mbox{ and }\Lambda_I=\Lambda(S,\overrightarrow{T_I},m^I_\bullet,c^I_\bullet)$$
It follows from Theorem \ref{thm:2.4} that 
$$\dim_K\Lambda=m+m+m+ 9^2=3m+81,\mbox{ and}$$
$$\dim_K\Lambda_I=4+4+4+m\cdot 3^2+m\cdot 3^2+m\cdot 3^2+15^2=27m+237.$$ 

Let $\epsilon:I\to\{-1,1\}$ be the function given by $\epsilon(1)=\epsilon(2)=1$ and $\epsilon(3)=-1$. Then the associated 
sequence of virtual arrows $\underline{\epsilon}\in\mathcal{O}_1\times\mathcal{O}_2\times\mathcal{O}_3$ is equal to 
$(\xi_1,\xi_2,\mu_3)$. Then the virtual edge deformation $\Lambda^\epsilon_I=\Lambda(S,\overrightarrow{T_I},m^I_\bullet,c^I_\bullet,\underline{\epsilon})$ 
of $\Lambda$ at $I$ (with respect to $\epsilon$) is defined by the quiver $Q^I(\underline{\epsilon})$ of the form 
$$
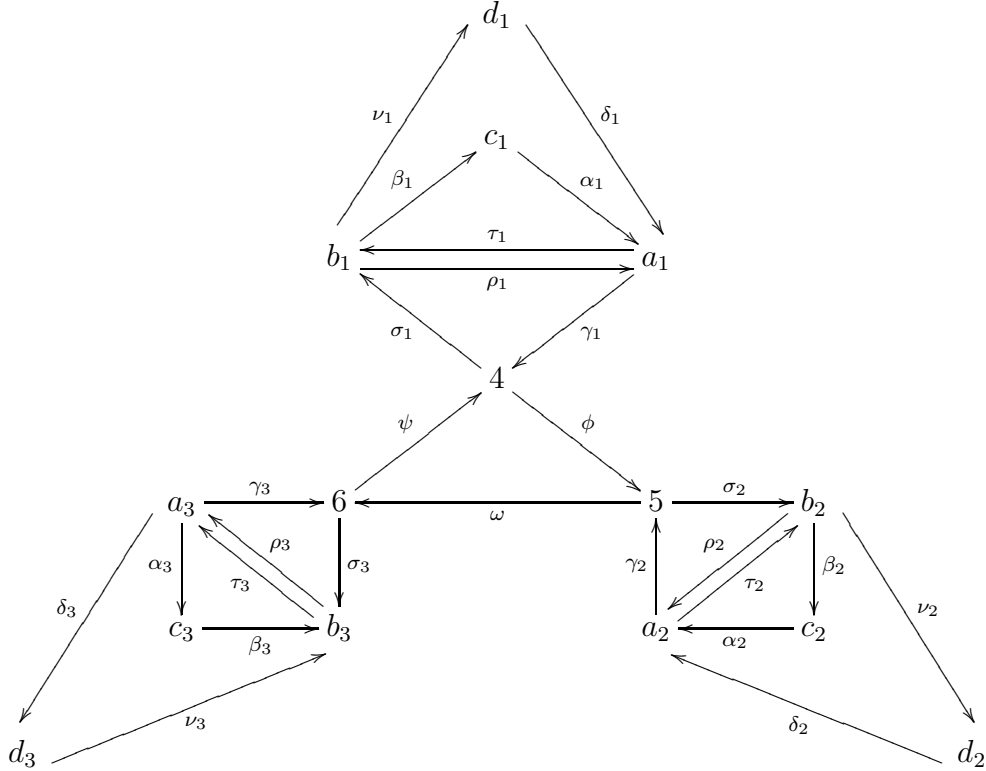
\xymatrix@C=1.5cm@R=1.1cm{
&&&     d_1\ar@<0.25cm>[rdd]^{\delta_1}    &&& \\ 
&&&     c_1\ar[rd]^{\alpha_1}   &&& \\ 
&&   b_1\ar@<0.25cm>[ruu]^{\nu_1}\ar@<-0.15cm>[rr]_{\rho_1}\ar[ru]^{\beta_1} & &  
a_1\ar[ld]^{\gamma_1}\ar@<-0.1cm>[ll]_{\tau_1} && \\
&&&       4\ar[lu]^{\sigma_1}\ar[rd]^{\phi}     &&& \\
& a_3\ar@<-0.25cm>[ldd]_{\delta_3}\ar[r]^{\gamma_3}\ar[d]_{\alpha_3} & 6\ar[ru]^{\psi}\ar[d]^{\sigma_3} & & 
5\ar[ll]^{\omega}\ar[r]^{\sigma_2} & b_2\ar@<-0.1cm>[ld]_{\rho_2}\ar@<0.25cm>[rdd]^{\nu_2}\ar[d]^{\beta_2} & \\ 
& c_3\ar[r]_{\beta_3} & b_3\ar@<-0.1cm>[lu]_{\rho_3}\ar@<0.1cm>[lu]^{\tau_3} && 
a_2\ar[u]^{\gamma_2}\ar@<-0.1cm>[ru]_{\tau_2} 
& c_2\ar[l]^{\alpha_2} & \\
d_3\ar@<-0.25cm>[rru]_{\nu_3}&&&
&&& d_2\ar@<0.25cm>[llu]^{\delta_2}}$$
We recall also that the associated quiver $Q^I(\underline{\epsilon})^*$ is obtained from $Q^I(\underline{\epsilon})$ by removing 
the arrows $\alpha_1,\beta_1,\alpha_2,\beta_2,\nu_3,\delta_3$. Further, the set $\mathcal{O}^I(g^*)$ of all $(g^I)^*$-orbits of 
the associated permutation $(g^I)^*:Q^I(\underline{\epsilon})^*_1\to Q^I(\underline{\epsilon})^*_1$ consists of the orbits 
$$(\tau_1\mbox{ }\rho_1),(\tau_2\mbox{ }\rho_2),(\alpha_3\mbox{ }\beta_3\mbox{ }\rho_3),\mbox{ and}$$
$$(\tau_3\mbox{ }\gamma_3\mbox{ }\psi\mbox{ }\sigma_1\mbox{ }\nu_1\mbox{ }\delta_1\mbox{ }\gamma_1\mbox{ }\phi
\mbox{ }\sigma_2\mbox{ }\nu_2\mbox{ }\delta_2\mbox{ }\gamma_2\mbox{ }\omega\mbox{ }\sigma_3).$$
Moreover, the associated permutation $(f^I)^*:Q^I(\underline{\epsilon})^*_1\to Q^I(\underline{\epsilon})^*_1$ has the following 
orbits 
$$(\nu_1\mbox{ }\delta_1\mbox{ }\tau_1),(\gamma_1\mbox{ }\sigma_1\mbox{ }\rho_1),(\nu_2\mbox{ }\delta_2\mbox{ }\tau_2),
(\gamma_2\mbox{ }\sigma_2\mbox{ }\rho_2),(\alpha_3\mbox{ }\beta_3\mbox{ }\tau_3),(\gamma_3\mbox{ }\sigma_3\mbox{ }\rho_3),(\phi\mbox{ }\omega\mbox{ }\psi).$$
Then it follows from Definition \ref{def:3.4} that the algebra $\Lambda^\epsilon_I$ is given by the quiver 
$Q^I(\underline{\epsilon})$ and the following relations (here we abbreviate $f^*:=(f^I)^*$ and $g^*:=(g^I)^*$):
\begin{enumerate}[(1)] 
\item $\nu_1\delta_1=\beta_1\alpha_1+(\rho_1\tau_1)^{m-1}\rho_1$, $\alpha_1\tau_1=0$, $\tau_1\beta_1=0$, $\newline$ 
$\nu_2\delta_2=\beta_2\alpha_2+(\rho_2\tau_2)^{m-1}\rho_2$, $\alpha_2\tau_2=0$, $\tau_2\beta_2=0$, $\newline$ 
$\alpha_3\beta_3=\delta_3\nu_3+\lambda\gamma_3\psi\sigma_1\nu_1\delta_1\gamma_1\phi\sigma_2\nu_2\delta_2\gamma_2\omega\sigma_3$, $\nu_3\tau_3=0$, $\tau_3\delta_3=0$, 

\item $\delta_1\tau_1=\lambda\delta_1\gamma_1\phi\sigma_2\nu_2\delta_2\gamma_2\omega\sigma_3\tau_3\gamma_3\psi\sigma_1$, 
$\tau_1\nu_1=\lambda\gamma_1\phi\sigma_2\nu_2\delta_2\gamma_2\omega\sigma_3\tau_3\gamma_3\psi\sigma_1\nu_1$, 
$\newline$ 
$\rho_1\gamma_1=\lambda\nu_1\delta_1\gamma_1\phi\sigma_2\nu_2\delta_2\gamma_2\omega\sigma_3\tau_3\gamma_3\psi$, 
$\sigma_1\rho_1=\lambda\phi\sigma_2\nu_2\delta_2\gamma_2\omega\sigma_3\tau_3\gamma_3\psi\sigma_1\nu_1\delta_1$, 
$\newline$ 
$\delta_2\tau_2=\lambda\delta_2\gamma_2\omega\sigma_3\tau_3\gamma_3\psi\sigma_1\nu_1\delta_1\gamma_1\phi\sigma_2$, 
$\tau_2\nu_2=\lambda\gamma_2\omega\sigma_3\tau_3\gamma_3\psi\sigma_1\nu_1\delta_1\gamma_1\phi\sigma_2\nu_2$, 
$\newline$ 
$\rho_2\gamma_2=\lambda\nu_2\delta_2\gamma_2\omega\sigma_3\tau_3\gamma_3\psi\sigma_1\nu_1\delta_1\gamma_1\phi$, 
$\sigma_2\rho_2=\lambda\omega\sigma_3\tau_3\gamma_3\psi\sigma_1\nu_1\delta_1\gamma_1\phi\sigma_2\nu_2\delta_2$, 
$\newline$
$\rho_3\gamma_3=\lambda\tau_3\gamma_3\psi\sigma_1\nu_1\delta_1\gamma_1\phi\sigma_2\nu_2\delta_2\gamma_2\omega$,
$\sigma_3\rho_3=\lambda\psi\sigma_1\nu_1\delta_1\gamma_1\phi\sigma_2\nu_2\delta_2\gamma_2\omega\sigma_3\tau_3$, 
$\newline$ 
$\gamma_1\sigma_1=(\tau_1\rho_1)^{m-1}\tau_1$, $\gamma_2\sigma_2=(\tau_2\rho_2)^{m-1}\tau_2$, 
$\gamma_3\sigma_3=(\alpha_3\beta_3\rho_3)^{m-1}\alpha_3\beta_3$, $\newline$ 
$\beta_3\tau_3=(\beta_3\rho_3\alpha_3)^{m-1}\beta_3\rho_3$, 
$\tau_3\alpha_3=(\rho_3\alpha_3\beta_3)^{m-1}\rho_3\alpha_3$, 

\item $\eta f^*(\eta) g^*(f^*(\eta))=0$ for any arrow $\eta$ from the set 
$$\{\rho_1,\gamma_1,\sigma_1,\psi,\phi,\omega,\rho_2,\gamma_2,\sigma_2,\rho_3,\gamma_3,\sigma_3\},$$ 
 
\item $\eta g^*(\eta) f^*(g^*(\eta))=0$ for any arrow $\eta$ from the set 
$$\{\delta_1,\tau_1,\rho_1,\gamma_1,\psi,\phi,\omega,\delta_2,\tau_2,\rho_2,\gamma_2,
\beta_3,\tau_3,\gamma_3,\sigma_3\},$$ 
 
\item $\alpha_1\gamma_1\phi\sigma_2\nu_2\delta_2\gamma_2\omega\sigma_3\tau_3\gamma_3\psi\sigma_1\nu_1=0$, 
$\delta_1\gamma_1\phi\sigma_2\nu_2\delta_2\gamma_2\omega\sigma_3\tau_3\gamma_3\psi\sigma_1\beta_1=0$, $\newline$ 
$\alpha_2\gamma_2\omega\sigma_3\tau_3\gamma_3\psi\sigma_1\nu_1\delta_1\gamma_1\phi\sigma_2\nu_2=0$, 
$\delta_2\gamma_2\omega\sigma_3\tau_3\gamma_3\psi\sigma_1\nu_1\delta_1\gamma_1\phi\sigma_2\beta_2=0$, $\newline$ 
$\nu_3(\rho_3\alpha_3\beta_3)^{m-1}\rho_3\alpha_3=0$, $(\beta_3\rho_3\alpha_3)^{m-1}\beta_3\rho_3\delta_3=0$.
\end{enumerate}

Moreover, applying Theorem \ref{thm:3.14} we have 
$$\dim_K\Lambda^\epsilon_I=4m+4m+12m+14\cdot 16+16+16+4m=24m+256.$$
\end{exmp}

\begin{rem}\label{rem:6.7} We would like to mention that there are virtual edge deformations $\Lambda^\epsilon_I$ of weighted 
surface algebras $\Lambda$ whose Gabriel quivers contain an arbitrary large number of subquivers of the forms 
$$
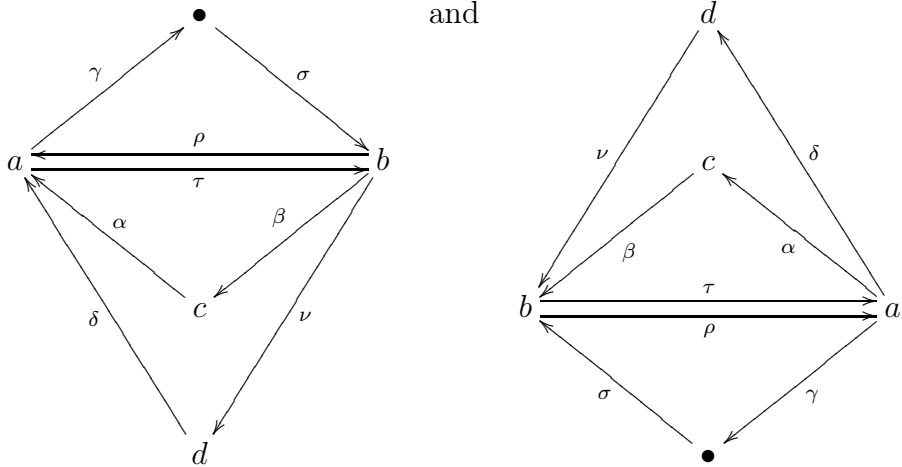
\xymatrix@R=1.5cm@C=2cm{&\bullet \ar[rd]^{\sigma}& \\
a\ar[ru]^{\gamma}\ar@<-0.1cm>[rr]_{\tau}& &  b\ar@<-0.1cm>[ll]_{\rho}\ar[ldd]^{\nu}\ar[ld]_{\beta}\\ 
&c\ar[lu]_{\alpha}& \\ 
&d\ar[luu]^{\delta}&}\quad\mbox{and}\quad\xymatrix@R=1.5cm@C=2cm{&d\ar[ldd]_{\nu} & \\ 
& c\ar[ld]^{\beta} \\ 
b\ar@<0.1cm>[rr]^{\tau}\ar@<-0.1cm>[rr]_{\rho} & & \ar[ld]^{\gamma} 
a\ar[luu]_{\delta}\ar[lu]^{\alpha} \\ 
& \ar[lu]^{\sigma} \bullet &}$$
We may obtain such algebras in the following way. 

Let $(S,\overrightarrow{T})$ be a directed triangulated surface with non-empty boundary and $X$ a fixed set of boundary edges 
of $T$. Then one may enlarge $(S,\overrightarrow{T})$ to the directed surface $(S(X),\overrightarrow{T(X)})$ by gluing each 
edge $x\in X$ with the self-folded triangle 
$$
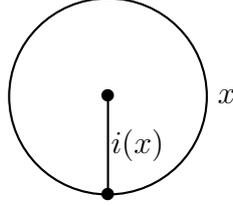
\begin{tikzpicture}[scale=1.3]
\draw[thick](0,0) circle (1);
\draw[thick](0,0)--(0,-1);
\node() at (0,0){$\bullet$};
\node() at (0,-1){$\bullet$};
\node() at (0.3,-0.5){$i(x)$};
\node() at (1.2,0){$x$};
\end{tikzpicture}$$
Next, we take the set $I(X)$ of all created self-folded edges $i(x)$, $x\in X$, and the associated blow-up 
$(S(X),\overrightarrow{T(X)}_{I(X)})$, so we have in $\overrightarrow{T(X)}_{I(X)}$ the triangles 
$$
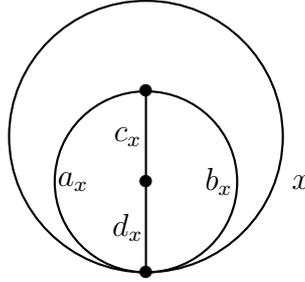
\begin{tikzpicture}[scale=1.2]
\draw[thick] (0,-1)--(0,1); 
\draw[thick](0,0) circle (1);
\node() at (0,0){$\bullet$};
\node() at (0,-1){$\bullet$}; 
\node() at (0,1){$\bullet$}; 
\node() at (0.8,0){$b_x$};
\node() at (-0.8,0){$a_x$};
\node() at (-0.2,0.5){$c_x$};
\node() at (-0.2,-0.5){$d_x$};
\draw[thick](0,0.5) circle (1.5);
\node() at (1.7,0){$x$};
\end{tikzpicture}$$
with orientation $(x\mbox{ }b_x\mbox{ }a_x)$, $(a_x\mbox{ }c_x\mbox{ }d_x)$ and $(c_x\mbox{ }b_x\mbox{ }d_x)$ in 
$\overrightarrow{T(X)}_{I(X)}$, for all $x\in X$. Hence the associated triangulation quiver 
$(Q(S(X),\overrightarrow{T(X)}_{I(X)}),f^{I(X)})$ contains subquivers 
$$
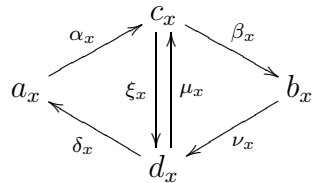
\xymatrix@R=0.5cm@C=1.2cm{&c_x\ar[rd]^{\beta_x}\ar@<-0.1cm>[dd]_{\xi_x}& \\ 
a_x\ar[ru]^{\alpha_x}&&b_x\ar[ld]^{\nu_x} \\ &d_x\ar[lu]^{\delta_x}\ar@<-0.1cm>[uu]_{\mu_x}& }$$
for $x\in X$. We also note that the triangulation quiver $(Q(S(X),\overrightarrow{T(X)}),f^X)$ has the subquivers 
$$
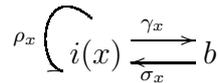
\xymatrix@R=0.01cm{&\ar@(lu,ld)[d]_{\rho_x}&\\& i(x)\ar@<0.2cm>[r]^{\gamma_x}  &b \ar@<0.1cm>[l]^{\sigma_x}\\ && }$$
with $f^X$-orbits $(\gamma_x\mbox{ }\sigma_x\mbox{ }\rho_x)$, and hence $g^X(\rho_x)=\rho_x$, for all $x\in X$, where $g^X$ 
is the permutation induced by $f^X$. Finally, we take a weight function $m^X_\bullet:\mathcal{O}(g^X)\to\mathbb{N}^*$ with 
value $m\geqslant 3$ on all the orbits of $\rho_x$, $x\in X$, and an arbitrary parameter function $c^X_\bullet:\mathcal{O}(g^X)\to K^*$. 
Then the suitable edge deformations $\Lambda^\epsilon_{I(X)}$ of 
$\Lambda=\Lambda(S(X),\overrightarrow{T(X)},m^X_\bullet,c^X_\bullet)$ with respect to $I(X)$ and functions 
$\epsilon:I(X)\to\{-1,1\}$ provide the required algebras. 

We also note that the directed triangulated surfaces and algebras considered in Example \ref{ex:6.6} are special cases of the 
above procedures, with the starting directed triangulated surface being the single triangle 
$$
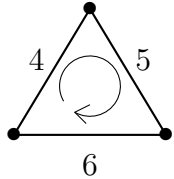
\begin{tikzpicture}
\draw[thick](-1,0)--(0,1.67)--(1,0)--(-1,0);
\node() at (-1,0){$\bullet$};
\node() at (1,0){$\bullet$};
\node() at (0,1.67){$\bullet$};
\node() at (-0.7,1){$4$};
\node() at (0.7,1){$5$};
\node() at (0,-0.4){$6$};
\draw(-0.2,0.3) arc (-120:210:0.4);
\draw(0,0.4)--(-0.2,0.3)--(-0.05,0.15);
\end{tikzpicture}$$
and the set $X=\{4,5,6\}$ of its boundary edges. 
\end{rem}


\begin{thebibliography}{00} 
\bibitem{ASS} I. Assem, D. Simson, A. Skowro\'nski, Elements of the Representation Theory of 
Associative Algebras 1: Techniques of Representation Theory, London Mathematical Society Student 
Texts, vol. 65, Cambridge University Press, Cambridge, 2006. 

\bibitem{BES} J. Bia{\l}kowski, K. Erdmann, A. Skowro\'nski, Periodicity of self-injective algebras 
of polynomial growth, J. Algebra 443 (2015) 200--269.

\bibitem{BiHS} J. Bia{\l}owski, T. Holm, A. Skowro\'nski, Derived equivalences for tame weakly 
symmetric algebras having only periodic modules, J. Algebra 269 (2003) 652--668.

\bibitem{BoHS} R. Bocian, T. Holm, A. Skowro\'nski, Derived equivalence classification of weakly 
symmetric algebras of Euclidean type, J. Pure Appl. Algebra 191 (2004) 43--74.

\bibitem{BIKR} I. Burban, O. Iyama, B. Keller, I. Reiten, Cluster tilting for one-dimensional 
hypersurface singularities, Adv. Math. 217 (2008) 2443--2484.

\bibitem{Ca} S. C. Carlson, Topology of Surfaces. Knots and Manifolds, A First Undergraduate 
Course, Wiley, New York, 2001.

\bibitem{CB} W. Crawley-Boevey, Tame algebras and generic modules, Proc. London Math. Soc. 63 
(1991) 241--265.

\bibitem{CR} J. Chuang, J. Rickard, Representations of finite groups and tilting, in: Handbook 
on Tilting Theory, London Math. Soc. Lecture Note Ser., vol. 332, Cambridge University Press, 
Cambridge, 2007, pp. 359--391. 

\bibitem{DWZ1} H. Derksen, J. Weyman, A. Zelevinsky, Quivers with potentials and their 
representations. I. Mutations, Selecta Math. (N. S.) 14 (2008) 59--119.

\bibitem{DWZ2} W. Derksen, J. Weyman, A. Zelevinsky, Quivers with potentials and their 
representations. II. Applications to cluster algebras, J. Amer. Math. Soc. 23 (2010) 749--790.

\bibitem{Du1} A. Dugas, Periodic algebras and self-injective algebras of finite type, J. Pure 
Appl. Algebra 214 (2010) 990-1000.

\bibitem{Du2} A. Dugas, A construction of derived equivalent pairs of symmetric algebras, Proc. 
Amer. Math. Soc. 143 (2015) 2281--2300.

\bibitem{E} K. Erdmann, Blocks of Tame Representation Type and Related Algebras, Lecture Notes 
Math., vol. 1428, Springer-Verlag, Berlin-Heidelberg, 1990.

\bibitem{ES0} K. Erdmann, A. Skowro\'nski, The stable Calabi-Yau dimension of tame symmetric 
algebras, J. Math. Soc. Japan 58 (2006) 97--128.

\bibitem{ES1} K. Erdmann, A. Skowro\'nski, Periodic algebras, in: Trends in Representation Theory of Algebras and Related Topics, in: Eur. Math. Soc. Congress Reports, Europen Math. Soc., Z\"urich, 2008, pp. 201--251.

\bibitem{ES2} K. Erdmann, A. Skowro\'nski, Weighted surface algebras, J. Algebra 505 (2018) 490--558.

\bibitem{ES3} K. Erdmann, A. Skowro\'nski, Higher tetrahedral algebras, Algebr. Represent. Theory 22 
(2019) 387--406.

\bibitem{ES4} K. Erdmann, A. Skowro\'nski, Algebras of generalized quaternion type, Adv. Math. 349 
(2019) 1036-1116.

\bibitem{ES5} K. Erdmann, A. Skowro\'nski, Weighted surface algebras: General version, J. Algebra 
544 (2020) 170--227.

\bibitem{ES6} K. Erdmann, A. Skowro\'nski, Higher spherical algebras, Arch. Math. 114 (2020) 25--39. 

\bibitem{ES7} K. Erdmann, A. Skowro\'nski, From Brauer graph algebras to biserial weighted surface algebras, 
J. Algebr. Comb. 51 (2020) 51--88.

\bibitem{ES8} K. Erdmann, A. Skowro\'nski, Algebras of generalized dihedral type, Nagoya Math. J., in press, 
https://doi.org/10.1017/nmj.2019.1.

\bibitem{ES9} K. Erdmann, A. Skowro\'nski, Weighted surface algebras: General version. Corrigendum, 
http://arxiv.org/abs/2005.05731. 

\bibitem{FST1} A. Felikson, M. Schapiro, P. Tumarkin, Skew symmetric cluster algebras of finite 
mutation type, J. Eur. Math. Soc. 14 (2012) 1135--1180.

\bibitem{FG} V. Fock, A. Gontcharov, Moduli spaces of local systems in Teichm\"uller theory, Publ. Math. 
Inst. Hautes \'Etudes Sci. 103 (2006) 1-211. 

\bibitem{FST2} S. Fomin, M. Schapiro, D. Thurston, Cluster algebras and triangulated surfaces. Part I. 
Cluster complexes, Acta Math. 201 (2008) 83--146.

\bibitem{GLFS} C. Geiss, D. Labardini-Fragoso, J. Schr\"oer, The representation type of Jacobian 
algebras, Adv. Math. 290 (2016) 364--452.

\bibitem{GSV} M. Gekhtman, M. Schapiro, A. Vainshtein, Cluster algebras and Poisson geometry, 
Mosc. Math. J. 3 (2003) 899-934.

\bibitem{Ha1} D. Happel, Triangulated Categories in the Representation Theory of 
Finite-Dimensional Algebras, London Math. Soc. Lecture Note Ser., vol. 119, Cambridge 
University Press, Cambridge, 1988.

\bibitem{Ha2} A. Hatcher, Algebraic Topology, Cambridge University Press, Cambridge, 2002.

\bibitem{Ho} T. Holm, Derived equivalence classification of algebras of dihedral, semidihedral, 
and quaterion type, J. Algebra 211 (1999) 159--205.

\bibitem{Ka} M. Kauer, Derived equivalence of graph algebras, in: Trends in the Representation 
Theory of Finite-Dimensional Algebras, in: Contemp. Math., vol. 229, Amer. Math. Soc., 
Providence, RI, 1998, pp. 201--213.  

\bibitem{KZ} H. Krause, G. Zwara, Stable equivalence and generic modules, Bull. London Math. 
Soc. 32 (2000) 615-618.

\bibitem{LF} D. Labardini-Fragoso, Quivers with potentials associated to triangulated 
surfaces, Proc. London Math. Soc. 98 (2009) 797--839.

\bibitem{La} S. Ladkani, From groups to clusters, in: Representation Theory - Current Trends 
and Perspectives, in: Eur. Math. Soc. Congress Reports, European Math. Soc., Z\"urich, 2017, 
pp. 427--500.

\bibitem{Lam} P. Lampe, Diophantine equations via cluster transformations, J. Algebra 462 (2016) 
320-337.

\bibitem{MS} R.J. Marsh, S. Schroll, The geometry of Brauer graph algebras and cluster 
mutations, J. Algebra 419 (2014) 141--166.

\bibitem{O} T. Okuyama, Some examples of derived equivalent blocks of finite groups, Preprint, 1998.

\bibitem{Ric1} J. Rickard, Morita theory for derived categories, J. London Math. Soc. 39 (1989) 436--456.

\bibitem{Ric2} J. Rickard, Derived categories and stable equivalence, J. Pure Appl. Algebra 61 (1989) 303--317. 

\bibitem{Ric3} J. Rickard, Derived equivalences as derived functors, J. London Math. Soc. 43 (1991) 37--48. 

\bibitem{SS} D. Simson, A. Skowro\'nski, Elements of the Representation Theory of Associative Algebras 3: 
Representation-Infinite Tilted Algebras, London Mathematical Society Student Texts, vol. 72, Cambridge 
University Press, Cambridge, 2007.

\bibitem{S1} A. Skowro\'nski, Selfinjective algebras of polynomial growth, Math. Ann. 285 (1989) 177--199. 

\bibitem{S2} A. Skowro\'nski, Selfinjective algebras: finite and tame type, in: Trends in Representation 
Theory of Algebras and Related Topics, in: Contemp. Math., vol. 406, Amer. Math. Soc., Providence, RI, 
2006, pp. 169--238.

\bibitem{SY} A. Skowro\'nski, K. Yamagata, Frobenius Algebras I. Basic Representation Theory, Eur. Math. 
Soc. Textbooks in Math., European Math. Soc., Z\"urich, 2011.

\bibitem{Sk} A. Skowyrski, Two tilts of higher spherical algebras, http://arxiv.org/abs/1911.07060.

\bibitem{VD} Y. Vladivieso-Diaz, Jacobian algebras with periodic module categories and exponential growth, 
J. Algebra 449 (2016) 163--174.

\end{thebibliography}
\end{document}